\documentclass[11pt]{article}

\usepackage[toc,page,header]{appendix}
\usepackage[normalem]{ulem}
\usepackage{amsfonts, latexsym, amssymb, amsthm, amsmath, mathrsfs, esint, color, authblk, fullpage}
\usepackage {amsmath,refcount,makerobust}
\usepackage{amscd}
\usepackage{enumerate}
\usepackage{pstricks}
\usepackage{pst-node}
\usepackage{mathtools}
\usepackage{graphicx, wrapfig}
\graphicspath{  }
\usepackage{relsize}
\usepackage{hyperref}
\usepackage{comment}
\usepackage{scrlfile}
\usepackage{srcltx}
\usepackage{nicefrac}

\renewcommand{\b}{\beta}

\newcommand\e {\varepsilon}

\newcommand\stiff{{\rm stiff}}
\newcommand\soft{{\rm soft}}

\newcommand{\ext}[1]{\widetilde{#1}}

\def\rme{{\rm e}}
\def\rmi{{\rm i}}

\def\RR{\R}

\newcommand{\R}{\mathbf{R}}

\def\C{\mathbf {C}}

\def\M{\mathcal{M}}

\def\AA{\mathcal{A}}

\def\SS{\mathcal{S}}

\def\l{\lambda}
\def\hom{{\rm hom}}
\def\dom{{\rm dom}}

\newcommand\wtto{\xrightharpoonup{2}}

\newcommand\stto{\xrightarrow{2}}

\DeclareMathOperator{\Sp}{Sp}

\newcommand\G{G}

\newcommand\A{\mathcal A}

\newcommand\N{\mathbf{N}}
\newcommand\Z{\mathbf{Z}}

\newcommand\esssup{\mathop{\operatorname{ess\,sup}}}

\newcommand\supp{\mathop{\operatorname{supp}}}
\newcommand\dist{\operatorname{dist}}

\def\b{\beta}

\def\XXint#1#2#3{{\setbox0=\hbox{$#1{#2#3}{\int}$}
     \vcenter{\hbox{$#2#3$}}\kern-.5\wd0}}

\newcounter{bei}

\newcommand{\eq}[1]{(\ref{#1})}

\newtheorem{theorem}{Theorem}[section]
\newtheorem{lemma}[theorem]{Lemma}
\newtheorem{proposition}[theorem]{Proposition}
\newtheorem{corollary}[theorem]{Corollary}
\newtheorem{assumption}[theorem]{Assumption}
\newtheorem{definition}[theorem]{Definition}
\newtheorem{remark}[theorem]{Remark}
\theoremstyle{remark}

\usepackage{mathtools} 
\usepackage[normalem]{ulem}

\newcommand\restrict[2]{{
  \left.\kern-\nulldelimiterspace 
  #1 
  \vphantom{\big|} 
  \right|_{#2} 
  }}

\usepackage{color}
\usepackage{ cancel }
\usepackage[color]{changebar}
\cbcolor{red}
\begin{document}
	
\title{{\sc Homogenisation and spectral convergence of high-contrast convolution type operators}}


\author[1]{Mikhail Cherdantsev}
\author[2]{Andrey Piatnitski}
\author[3]{Igor Vel\v{c}i\'{c}\,}

\affil[1]{School of Mathematics, Cardiff University, Senghennydd Road, Cardiff, CF24 4AG, United Kingdom, CherdantsevM@cardiff.ac.uk, ORCID 0000-0002-5175-5767}
\affil[2]{Higher School of Modern Mathematics MIPT, 1st Klimentovskiy per., 115184 Moscow, Russia and The Arctic University of Norway, Campus in Narvik, P.O. Box 385, Narvik 8505, Norway, apiatnitski@gmail.com, ORCID 0000-0002-9874-6227}
\affil[3]{Faculty of Electrical Engineering and Computing, University of Zagreb, Unska 3, 10000 Zagreb, Croatia, Igor.Velcic@fer.hr, ORCID 0000-0003-2494-2230}

\date{}
\maketitle

\vspace{-6mm}

\noindent {\bf Acknowledgements:} 

The work of Igor Velcic was supported by the Croatian Science Foundation under the project number HRZZ-IP-2022-10-518.

The work was partially supported by Pure Maths in Norway Foundation and UiT Aurora project MASCOT.

\begin{abstract}
 The paper deals with homogenisation problems for high-contrast
symmetric convolution-type operators with integrable kernels in media
with a periodic microstructure.  We adapt the two-scale convergence method
to nonlocal convolution-type operators and obtain the homogenisation result
both for problems stated in the whole space and in bounded domains with the homogeneous Dirichlet
boundary condition.

Our main focus is on spectral  analysis. We describe the spectrum
of the limit two-scale operator  and characterize the limit behaviour of the spectrum of the original problem
as the microstructure period tends to zero. It is shown that the spectrum of the limit operator  is a subset
the limit  of the spectrum of the original operator, and that  they need not coincide.
\vskip 0.15cm


\vskip 0.15cm


\vskip 0.4cm

\end{abstract}

\tableofcontents

\section{Introduction}

This work is devoted to the homogenisation of high-contrast symmetric convolution-type operators with integrable kernels in periodic media. In the first part of the paper, we show that the two-scale convergence method (see \cite{Allaire}) applies to the class of operators under consideration and present several technical results  that help us to homogenise this family of operators. We then analyse the corresponding spectral problems in two settings: problems posed in the whole space, and boundary value problems in bounded Lipschitz domains. In the latter case, we impose homogeneous Dirichlet boundary conditions in the complement of the domain.

The  spectrum of the limit homogenised operator is non-trivial. Following the ideas of \cite{Zhikov2000} we introduce an auxiliary   Zhikov's $\beta$-function,  and,   describe the spectrum of the limit two-scale operator in terms of this function.

Next we study the limit behaviour of the spectrum of the original operator in the whole space setting and characterize the Hausdorff limit of this spectrum as the microstructure period tends to zero. In particular, we show that the spectrum of the limit operator is always a subset of the said limit, and that the opposite inclusion need not hold, neither in the whole space nor in a bounded domain. We show that the additional limit spectrum is associated with the the quasiperiodic quasi-modes supported on the soft component. Remarkably, the soft component need not be infinite for this to happen (cf. the discussion below). Indeed, in the case of disconnected soft inclusions it is enough for the convolution kernel to have sufficiently large support in order that inclusions ``communicate'' with their neighbours.

For the generic Lipschitz domain we face the usual in periodic homogenisation problem of disagreement between the periodic microstructure and the domain boundary. In particular, the Hausdorff limit of the spectra does not exist in general in the presence of the boundary. However, the Hausdorff limit of the spectra may exist for domains of particular shape.  To illustrate this, we study a special case when the domain is rectangular  and assume that the small parameter $\e$ goes to zero along a discrete subsequence so that the geometry of the microstructure in the boundary layer is self-congruent along this subsequence. We show that in this case the limit of the spectra exists in the Hausdorff  sense and provide its characterisation.

Finally, in the whole space setting, adapting approach of \cite{CKS2023}, we establish norm-resolvent convergence result with explicit bounds via scaled Gelfand transform, and, as a consequence, obtain bounds on the rate of spectral convergence.

Various processes in the models of population biology, porous media and chemistry of polymers
are often described in terms of evolution equations of the form $\partial_t u=A u$ with a non-local convolution type operator $A$ and the corresponding stationary equations. The non-locality of $A$ reflects the  non-local nature of the interaction in these models.
One of the models of this type, the so-called contact model in $\R^d$,  has been actively studied in the existing literature, see \cite{KoKuPir2008}, \cite{KMPZ2017}, \cite{KoPirZh2016}.

In a typical case the kernel of the operator $A$ is a product of a convolution kernel
$a(x-y)$ and a coefficient $\Lambda(x,y)$.  The kernel $a(x-y)$  specifies the intensity of interaction in the model depending on the distance. It determines the localization properties of $A$. The coefficient $\Lambda(x,y)$ represents
the local characteristics of the environment.
Raising the question of the long-time behaviour of these processes and of the macroscopic description of models with
a microstructure, we arrive at the upscaling or homogenisation problems for non-local convolution-type operators.

Rigorous homogenisation results for moderate-contrast zero order convolution type operators in periodic
environments have been obtained in
\cite{PiZhiz2017}, \cite{PiZhiz2019}. In \cite{PiZhiz2017} it was shown that under natural moment and
coerciveness conditions a family of symmetric operators with periodic coefficients admits homogenisation, the effective
operator being a second order elliptic differential operator with constant coefficients. For a non-symmetric operators
the homogenisation takes place in moving coordinates, see \cite{PiZhiz2019}.
Sharp in order estimates for the rate of convergence in the operator norms have been obtained in \cite{PSSZ2023}.
Homogenization problems for symmetric convolution-type operators with random statistically homogeneous coefficients have been considered in \cite{PiZhiz2020}. It was proved that under the same moment and coerciveness conditions as in the
periodic case  the almost sure homogenisation result holds, and the limit operator is a second order elliptic differential
operator with constant coefficients. In the ergodic media the limit operator is deterministic.

Non-local operators and functionals of convolution type in perforated domains have been investigated by the variational methods in \cite{BrP2021}--\cite{BrP2022} and, in the case of more complicated geometry, in \cite{BCE2021}.
Homogenization result for a high-contrast convolution type evolution equation was obtained in the recent
work \cite{PiZhiz_arxiv},  where  the correctors and semigroup approximation techniques were used.

High-contrast differential equations with rapidly oscillating coefficients have been widely studied in the
existing literature starting from   \cite{ADH1990}. At present there are many works devoted to this topic. 
However, it turned out that 
the asymptotic behaviour of the spectrum
of these operators is a rather delicate problem. It was addressed successfully in \cite{Zhikov2000}. 

The approach developed in this article relies on the two-scale convergence technique.
However, the two-scale resolvent convergence result we obtain implies only ``half'' of the Hausdorff convergence of the spectra, namely, that the limiting spectrum contains the    spectrum of the limit operator. The inverse inclusion (let us focus on the whole space setting at the moment, to avoid boundary layer effects) requires additional assumptions on the operator and / or geometry of the soft component, and is not true in general. For example, in the periodic setting of \cite{Zhikov2000},\cite{Zhikov2004} the Hausdorff convergence of the spectra holds provided that the soft component is a collection of disconnected finite size inclusions. When such assumptions are not satisfied, it may happen that the limiting spectrum is strictly larger than that of the limit operator. In high-contrast problem this situation was first rigorously analysed in \cite{Coo2018}, see also \cite{CKS2023}, which provides deeper insight in the setting of \cite{Coo2018}, where the ``additional'' spectrum is attributed to the quasiperiodic quasi-modes supported on the {\it infinite} soft component, which are not captured by the two-scale limit. For other approaches to norm resolvent estimates for high contrast PDEs we refer to \cite{CC,CKVZ,gloria1}.

A similar picture can be observed in other settings.  In \cite{BCVZ2022} the authors investigate the limit resolvent equation, limit spectrum and limit evolution for high contrast thin elastic plates. It turns out that in one of the regimes the limit spectrum is strictly larger than the spectrum of two-scale limit operator. The limit behaviour of the  spectrum of high-contrast elliptic differential operators in random statistically
homogeneous environments was studied in  \cite{CCV23}. There, the additional limiting spectrum not accounted by the two-scale limit operator is of a different nature and is due to stochastic fluctuations of arrangement of inclusions from the ergodic average. However, in case of a bounded domain this extra spectrum is not present in the limit \cite{CCV2019}. We also mention \cite{APZ} for results on semigroup convergence and the spectrum of the limit operator.




In the context of boundary layer spectrum we mention \cite{AlCo1998}, which focuses on the high frequency spectrum  for moderate-contrast elliptic PDEs in a bounded domain.
Making use of the Floquet-Bloch transform  the authors characterize the limit spectrum; they also characterize the limit boundary spectrum for a rectangular domain and a discrete subsequence of the microscopic parameter.

Finally, we note that this work contains new extension result, which is simpler and, in fact, more natural for the framework of integral operators, compared to the one used in previous works in the area. Moreover, the `minimal' assumptions on the geometry of the soft / stiff components necessary for the extension result have been relaxed, in particular, they do not require any regularity of the boundary of the sets. We also develop a regularisation technique for the `bounded energy' sequences of functions allowing for $H^1$ bounds, which leads to an elegant proof of compactness results.

{\bf Structure of the paper}

In the next section we set the problem, define the family of operators $\A_\e$, and state our main results: in Section \ref{s2.1} we describe the limit two-scale operator and its spectrum and state the spectral inclusion result; in Section \ref{s2.2} we state spectral convergence results for the whole space setting and a rectangular domain; finally, in Section \ref{s2.2} we bounds for the norm resolvent and spectral convergence for the whole space setting.

In Section \ref{two-scalecon} we establish well-posedness of the corrector problem \eqref{ex102} and  prove the first main result of the paper Theorem \ref{thmandrei1}.

Section \ref{s5} is devoted to the analysis of the spectrum of the limit operator $\AA$. There we study properties of function $\beta$, prove Theorem \ref{thm2.3} and provide a number of example for possible structure of the spectrum of the operator $\AA_\soft$.

In Section \ref{s:6} we address the question of spectral convergence and prove Theorems \ref{th2.5}, \ref{th2.6} and \ref{th2.7}.

Finally, in Section \ref{s7} we prove our norm-resolvent and spectral convergence bounds --- Theorem \ref{th2.8} for the case $S=\R^d$.

In Appendix \ref{a1} we provide a new simplified (compared to previous works in the area) extension  theorem; Appendix \ref{a2} provides regularisation and compactness for `bounded energy' functions; finally, in Appendix \ref{a3} we establish two-scale convergence properties for convolution energies.

\bigskip

\section{Problem setting and main results} \label{secpbmsetting}

We begin with the description of the geometry of the medium. We work with the periodicity cell $Y=[0,1)^d$ and denote by $Y^{\#}=\R^d/\Z^d$  the corresponding flat torus with quotient topology. We will use $\#$ in the subscript or superscript to denote periodic sets, spaces of periodic functions and  associated operators. Let $Y_\stiff^{\#}$ and $Y_\soft^{\#}$   be open disjoint periodic sets such that $\overline{Y_\stiff^{\#} \cup Y_\soft^{\#}} = \R^d$. They represent ``stiff'' and ``soft'' components of the medium respectively. Denote $Y_\stiff:= Y_\stiff^{\#}\cap Y$ and $Y_\soft:= Y_\soft^{\#}\cap Y$.

By $L^2_\#(Y)$ we denote the space of $L^2(Y)$ functions extended by periodicity to  $\R^d$.  By $L^2(Y_\soft)$ and $L^2(Y_\stiff)$ we denote the subspaces of $L^2(Y)$ whose elements vanish on $Y_\stiff$ and $Y_\soft$ respectively;  $L^2_\#(Y_\soft)$ and $L^2_\#(Y_\stiff)$ denote the spaces of their periodic extensions. For a measurable set $S\in \R^d$ we denote $L^2_\#(S\times Y_\soft) : = L^2(S;L^2_\#(Y_\soft))$, i.e. the space of functions from $L^2(S\times Y)$ which vanish for $y\in Y_\stiff$ and periodically extended in $y$ variable.

For $A \subset \RR^d$, $\mathbf{1}_A$ denotes the characteristic function of the set $A$, and $|A|$ stands for its Lebesgue measure. For $r>0$ we denote $A^r:= \{x\in \R^d: \dist(x,A) < r\}$ and $A_r:= \{x\in A: \dist(x,\partial A) > r\}$.
By
$C^k(A)$ we denote the set of $k$ times continuously differentiable functions on $A$, and  by $C_0^k(A)$ we denote the set of $k$ times continuously differentiable functions which are compactly supported in $A$. $H_0^k(A)$ denotes the closure of $C_0^k(A)$ with respect to $H^k$ norm.
By
${H}^k_\#(Y)$ we denote the Sobolev space of periodic functions on the torus.
  For $x \in \RR^d$ and $m>0$, $\square_x^m$ denotes the cube $[x-m,x+m]^d$. We also set $\square^m:=\square_0^m $ and $\square_x:=\square_x^1 $. $B_r(x)$ denotes the open ball of radius $r$ centred at $x$, and $B_r:=B_r(0)$. We define
$$D_r:=\{(x,y) \in \R^d \times \R^d: |x-y|<r\}.   $$

Next we describe the operator. Let  $S$  denote  either $\RR^d$ or its open bounded Lipschitz subset. We consider a bounded   operator  $\AA_\e : L^2(S) \to L^2(S) $ (as above, we identify $L^2(S)$ with the subspace of $L^2(\R^d)$ whose elements vanish on the complement of $S$) defined according to
$$
\AA_\e u(x)= \frac{2}{\e^{d+2}} \int_{\mathbf{R}^d}a \Big(\frac{x-y }\e \Big)
\Lambda_\e  (x,y ) (u(x)-u(y ))dy.
$$
We make the following assumptions on the integral kernel.
\begin{assumption}
\begin{equation*}
	a \geq 0, \quad a \mbox{ is even, i.e. } a(x)=a(-x) \quad \forall x \in \mathbf{R}^d;
\end{equation*}
\begin{equation}\label{a-2}
	\mbox{there exists an ellipticity radius } r_a > 0 \mbox{ such that }	a(x) \geq c_a>0, \mbox{ for } |x|<r_a;
\end{equation}
\begin{equation*}
	x\mapsto a(x)(1+|x|^2) \in L^1(\R^d).
\end{equation*}
The function $\Lambda_\e$ encodes the periodicity and high-contrast of the medium. We put
$$\Lambda_\e (x,y )=\Lambda_0 \left(\frac{x}\e ,\frac{y }\e  \right)+\e^2p\left(\frac{x}\e , \frac{y }\e  \right), $$
where $\Lambda_0$, $p$ are symmetric $Y$-periodic functions in each argument such that
\begin{equation*} 
	\Lambda_0(x,y)=0 \textrm{ outside } Y_\stiff \times Y_\stiff,  
\end{equation*}
\begin{equation*} 
	p(x,y)= w(x,y) (1-\mathbf{1}_{Y_\stiff}(x)\mathbf{1}_{Y_\stiff}(y)),  
	 \end{equation*}
\begin{equation*} 
	0<\alpha_1<w,\,\Lambda_0\vert_{Y_\stiff \times Y_\stiff}<\alpha_2<\infty
\end{equation*}
for some $\alpha_1,\alpha_2>0$.
\end{assumption}

In terms of the geometry of the sets $Y_\stiff$ and $Y_\soft$ we require a very simple property, which in plain language can be expressed as ``the stiff component $Y_\stiff^\#$ must be ``connected'' via the convolution kernel $a$''. No other conditions, such as regularity of the boundary, are needed. 

First observe that for any open (non-empty) periodic set $Y_\stiff^\#$ the following holds: there exist $r_0, \kappa_0 >0$ such that
\begin{equation*}
	 \frac{|Y^\#_\stiff  \cap B_{r_0}(x)|}{|B_{r_0}|} \geq \kappa_0 \qquad\forall \, x \in Y_\stiff.
\end{equation*}
Furthermore, there exist $r_1>0$, $k, \overline{N}\in \N$ such that    for any two points $\eta', \eta''\in Y^\#_\stiff  \cap \square$ there exists a discrete path from $\eta'$ to $\eta''$ contained in $Y^\#_\stiff \cap \square^k$, \textit{i.e.} a set of points
\begin{equation}\label{8a}
	\{\eta_0=\eta', \eta_1, \dots, \eta_N, \eta_{N+1}=\eta''\} \subset Y^\#_\stiff \cap \square^k,
\end{equation}
such that $N\le \overline{N}$ and 	 $|\eta_{j+1}-\eta_{j}|\leq r_1$, for $j=0,1,\dots, N$.
We make the following

\begin{assumption}\label{assumptionmisja1} 
 There exist numbers $r_0, \kappa_0,r_1, k, \overline{N}$, as above such that the following inequality holds:
 \begin{equation*}
 	r_a\ge 2r_0+r_1.
 \end{equation*}
\end{assumption}
One can choose $r_0$ to be any number greater than
\begin{equation*}
\inf \Big\{r>0 :\,\, \inf_{ x \in Y_\stiff }	\frac{|Y^\#_\stiff  \cap B_{r}(x)|}{|B_{r}|} > 0\Big\}.
\end{equation*}
We can choose $r_1$ in a similar way. Both $r_0$ and $r_1$ encode geometrical properties of the set $Y_\stiff$.

Since the integral kernel is symmetric, the operator $\AA_\e$ is self-adjoint. The associated bilinear form is given by 
\begin{eqnarray} \nonumber
	a_\e  (u,v) = \frac{1}{\e^{d+2}} \int_{\mathbf{R}^d}\int_{\mathbf{R}^d} a \Big(\frac{x-\eta }\e \Big) \Lambda_\e  (x,\eta ) (u(x)-u(\eta ))(v(x)-v(\eta ))dx d\eta, \quad u,v \in L^2(S).
\end{eqnarray}
It is convenient to work with the weak formulation of the resolvent problem for $\AA_\e$: for $f_\e  \in L^2(S)$ and  $\lambda<0$ find $u_\e\in L^2(S)$ such that
\begin{eqnarray}\label{starting}
	a_\e(u_\e,v) - \lambda \int_{S} u_\e v =\int_{S} f_\e v \qquad\hbox{for all }v\in L^2(S).
\end{eqnarray}
In case of a bounded domain $S$ the assumption that $u_\e$ vanishes outside $S$ represents homogeneous Dirichlet boundary condition.

\begin{remark}
	We do not assume any regularity of the sets representing the stiff and the soft components beyond them being open. While not surprising, this is in a stark contrast with the PDEs case, where some boundary regularity is required to guarantee existence of extension from the stiff into the soft components with the control of $H^1$-norm. In the present setting, we only need to control the convolution energy, cf. \eqref{stimagrad}. It turns out that a simple piecewise constant extension by local averages does the job! The only regularity we require in Assumption \ref{assumptionmisja1} is that the stiff component is ``connected'' through the convolution kernel $a$ --- no geometrical connectedness of the stiff component is required.
	
	In the case of high-contrast PDEs ($-\nabla\cdot a_\e \nabla$) it is important for the structure of the limit spectrum whether the soft component comprises infinite connected sets or a collection of disconnected inclusions. In particular, the limit spectrum is strictly larger than the spectrum of the limit two-scale operator in the case of the former. In the present setting even if the soft component consists of disconnected inclusions they still may ``communicate'' with each other if the support of the convolution kernel $a$ is sufficiently large, cf. operators $\A_\soft^\#$ and $\A_\soft$ below.
\end{remark}

\begin{remark}
	In the case of a bounded domain $S$, one can  also study  the  problem with homogenous Neumann boundary condition. In this case, the integration in the bilinear form $a_\e$  is taken over the set $S \times S$ rather than $\R^d \times \R^d$:
	\begin{eqnarray*}
		a_\e  (u,v) = \frac{1}{\e^{d+2}} \int_{S}\int_{S} a \Big(\frac{x-\eta }\e \Big) \Lambda_\e  (x,\eta ) (u(x)-u(\eta ))(v(x)-v(\eta ))dx d\eta, \quad u,v \in L^2(S).
	\end{eqnarray*}
	The analysis and results for the Neumann problem would be analogous to  the ones obtained in the case of the  Dirichlet condition, including the analysis of the boundary spectrum, see the discussion below.
\end{remark}

\subsection{Two-scale limit operator and its spectrum}\label{s2.1}

Our first result is concerned with the two-scale resolvent limit for the operator $\AA_\e$.
We denote by $\AA$ the unbounded self-adjoint operator acting in the space $L^2(S)+L^2_\#(S\times Y_\soft)$ and associated with the bilinear form
\begin{multline} \label{6}
	a(u+z,v+b): =  a_{\rm hom}(u,v) + \int_{S} a_\soft^\#(z(x,\cdot),b(x,\cdot))dx,
	\\
	u+z, v+b\in H:=H^1_0(S)+L^2_\#(S\times Y_\soft),
\end{multline} 	
where
\begin{equation}\label{10}
	a_{\rm hom}(u,v):=\int_S  A^{\rm hom} \nabla u \cdot \nabla v \, dx, \quad u,v \in H^1_0(S),
\end{equation}
and
\begin{equation} \label{ruc11}
	a_\soft^\#(z,b):= \int_{\mathbf{R}^d}\int_{Y}  a(\xi)p(y, y+\xi)(z(y+\xi)-z(y))(b(y+\xi)-b(y)) \,dy\,d\xi, \quad z,b \in L^2_\#(Y_\soft).
\end{equation} 	
Here $A^{\rm hom}$ is the homogenised matrix of the stiff component,
\begin{equation} \label{defA}
	A^{\rm hom}_{ij}:=\int_{Y} \int_{\R^d} a(\xi) \Lambda_0 (y,y+\xi)\left(\xi_i+\chi^i(y+\xi)-\chi^i(y) \right)\xi_j\,d\xi\,dy,
\end{equation}
where $\chi^i \in L^2_\#(Y_\stiff),$ $i=1,\dots,d,$ are the corresponding homogenisation correctors   defined as the unique up to a constant (cf. Lemma \ref{lmdefA} below) solutions of the corrector problem
\begin{equation} \label{ex102}
	\int_{Y} \int_{\R^d} a(\xi) \Lambda_0 (y,y+\xi)\left(\xi_i+\chi^i(y+\xi)-\chi^i(y) \right)\left(b(y+\xi)-b(y) \right) \,d\xi\, dy=0\quad \forall b \in L^2_\#(Y).
\end{equation}
We denote by  $\AA_{\rm hom}$
and  $\AA_\soft^\# $ the self-adjoint operators associated with the forms \eqref{10} and \eqref{ruc11} respectively.

The resolvent problem for the operator $\AA$ associated with the form \eqref{6} reads
\begin{equation}\label{limit_pr}
	a(u+z,v+b) - \l \int_{S} \int_{Y} (u+z)(v+b) = \int_{S} \int_{Y} f (v+b) \qquad \forall v+b\in H.
\end{equation}
This equation can be equivalently written as the following coupled system:
\begin{eqnarray} \nonumber
	& &\int_{\R^d}  A^{\rm hom} \nabla u(x) \cdot \nabla v(x) \, dx -\lambda\int_{\R^d} \left(u(x)+\int_{Y} z(x,y)\, dy\right)v(x)\, dx\\ & & \label{ruc1} \hspace{+40ex}=\int_{\R^d} \int_{Y} f(x,y)\,dy\,  v(x) \, dx \quad \forall v \in H^1_0(S), \\ \nonumber
	& & \int_{\R^d}\int_{Y}  a(\xi)p(y, y+\xi)(z(x,y+\xi)-z(x,y))(b(y+\xi)-b(y)) \,dy\,d\xi\, \\  & &  - \lambda  \int_{Y} (u(x)+z(x,y)) b(y)dy= \int_{Y} f(x,y) b(y)dy \quad \forall b\in L^2_\#(Y_\soft)  \textrm{ for a.e. } x \in S. 	\label{ex200}
\end{eqnarray}

\begin{theorem} \label{thmandrei1}
	Let $(f_\e )_{\e>0}$ be a bounded sequence in $L^2(S)$ such that $f_\e  \xrightharpoonup{2} (\xrightarrow{2})f(x,y)\in L^2(S\times Y)$.  Then for the solution $u_\e $ of problem \eqref{starting} with $\lambda<0$ we have
	$$ u_\e  \xrightharpoonup{2} (\xrightarrow{2})\,u+z , \quad u + z \in H,   $$
	where $u+z$ is the solution to \eqref{limit_pr}.
	
\end{theorem} 	

In what follows we will often use the notation $\langle f \rangle = \int_Y f dy$.

\begin{remark} \label{remandrei3}
	Notice that while the limit operator $\AA$  acts in the space $L^2(S)+L^2_\#(S\times Y_\soft)$, the weak equation \eqref{limit_pr} (\eqref{ruc1}-\eqref{ex200}) makes sense for any right hand side $f\in L^2(S\times Y)$. The problem \eqref{limit_pr} may be written in the operator form with the help of projection operator  $\mathcal{P}:L^2(S\times Y)\to L^2(S)+L^2(S\times Y_\soft)$. It is not difficult to see that for an element $f\in L^2(S\times Y)$ one has
	\begin{equation*}
		\mathcal{P}f = |Y_\stiff|^{-1} \langle f \,\mathbf{1}_{Y_\stiff}\rangle + \mathbf{1}_{Y_\soft} (f - |Y_\stiff|^{-1} \langle f\, \mathbf{1}_{Y_\stiff}\rangle).
	\end{equation*}
	Thus \eqref{limit_pr}  reads $(\AA - \l I)(u+z) = \mathcal{P} f$.
	In particular,   Theorem \ref{thmandrei1} can be rephrased as follows:
	
	\noindent If $f_\e  \xrightharpoonup{2} (\xrightarrow{2}) f(x,y)$, then for $\lambda<0$ we have
	$$ (\AA_\e-\lambda I)^{-1} f_\e  \xrightharpoonup{2} (\xrightarrow{2})(\AA-\lambda I)^{-1} \mathcal{P}f(x,y). $$
	This property is commonly known as the weak (strong) two-scale resolvent convergence.
	
	It is well known that the 	strong (two-scale) resolvent convergence entails ``spectral inclusion''. In particular, for the operators $\AA_\e$ one has
	\begin{equation}\label{14a}
		\Sp (\AA) \subset \lim_{\e \to 0} \Sp (\AA_\e ).
	\end{equation}
	The argument is classical and rather straightforward. In the two-scale convergence context we refer to e.g. \cite{zhikov2}, and to \cite{Past} in a more general setting.

\end{remark} 	

	In \eqref{14a} and in what follows the limit notation for a sequence of sets is understood in the sense of the following
\begin{definition}
	For a family of sets $\SS_\e \subset \R$ the notation $\lim_{\e \to 0} \SS_\e$ stands for the set of all limit points of $\SS_\e$ in the sense that for any $\l \in \lim_{\e \to 0} \SS_\e$ there exists a subsequence $\e_k \to 0$ and $\l_{\e_k} \in \SS_{\e_k}$ such that $\l_{\e_k} \to \l$, and vice versa, for any converging subsequence $\l_{\e_k} \in \SS_{\e_k}$ the limit is in $\lim_{\e \to 0} \SS_\e$.
	
	 In case if a  sequence $\SS_\e$ has a limit in the sense of Hausdorff, we will write ${\rm H}\mbox{-}\lim_{\e \to 0} \SS_\e$. 	
	We recall that a set $\SS\subset\RR$ is the Hausdorff limit  of a family of sets $\SS_\e$    if
	\begin{enumerate}
		\item 	for any $\lambda \in \mathcal{S}$ there exists a sequence $(\lambda_\e)_{\e>0}$ such that  $\lambda_\e\in \SS_\e$ and $\lim\limits_{\e\to0} \lambda_\e=\lambda$;
		\item if $\lambda_\e \in \SS_\e$ is such that $\lim_{\e\to 0} \lambda_\e=\lambda$, then $\lambda
		\in \SS$.
	\end{enumerate}
\end{definition}

\begin{remark} \label{remopeq}
	It is not difficult to see that the operator $\AA_\soft^\# : L^2_\#(Y_\soft) \to L^2(Y_\soft)$ is given by
	\begin{equation}\label{14}
		\AA_\soft^\# z(y) = 2 \int_{\mathbf{R}^d}   a(\xi -y)p(y, \xi)d\xi\, z(y) - 2 \int_{\mathbf{R}^d}   a(\xi -y)p(y, \xi) \mathbf{1}_{Y_\soft^{\#}}(\xi) \mathbf{1}_{Y_\soft}(y) z(\xi) d\xi.
	\end{equation}
	We emphasise that the operator $\AA_\soft^\#$ acts on the space of periodic functions defined on $\R^d$. The target space, however is defined only over the single cell $Y$. The same is true for other operators acting on spaces of periodic functions.
	
	In the operator form the equations \eqref{ruc1}-\eqref{ex200} read
	\begin{eqnarray*}
		\AA_\hom  u-\lambda \left(u+ \langle z \rangle  \right)&=&\langle f \rangle, \\ 
		\AA_\soft^\#  z(x,\cdot)-\lambda (u(x)\mathbf{1}_{Y_\soft }(\cdot)+z(x,\cdot))&=& f(x,\cdot)\mathbf{1}_{Y_\soft }(\cdot), \quad u \in H_0^1(S), \ z \in L^2(S \times Y_\soft^{\#}).
	\end{eqnarray*} 	
\end{remark}

In order to characterize the spectrum of $\AA$ we introduce the function $\beta: \R\setminus \Sp (\AA_\soft^\# ) \to \R$:
\begin{equation} \label{ruc22}
	\beta(\lambda):= \lambda+\lambda^2 \Big\langle (\AA_\soft^\# -\l I)^{-1}  \mathbf{1}_{Y_\soft }\Big\rangle = \lambda+\lambda^2 \langle b_\l \rangle, \quad \l \in \mathbf{R}^+_0 \backslash \Sp (\AA_\soft^\# ),
\end{equation} 	
where we denote $b_\l:=(\AA_\soft^\# -\l I)^{-1}  \mathbf{1}_{Y_\soft }$.
The spectrum of the limit two-scale operator can be fully characterised by the spectrum of $\AA_\hom$, function $\beta$ and the spectrum of $\AA_\soft^\# $:
\begin{theorem}\label{thm2.3}
	\begin{equation}\label{57}
		\Sp (\AA)= \{\beta(\l) \in \Sp (\AA_\hom)\}\cup \Sp  (\AA_\soft^\# )  .
	\end{equation} 	
\end{theorem} 	

Here and in what follows, when we write $\b(\l)$ we tacitly assume that $\l$ belongs to the domain of $\b$, i.e. $\l \in \R\setminus \Sp (\AA_\soft^\# )$.

\subsection{Spectral convergence} \label{s2.2}

The spectral inclusion inverse to \eqref{14a} is not always the case. For it to hold for high-contrast problems one needs some sort of locality property for the soft component. For example, in the case of periodic elliptic PDEs, the inverse to \eqref{14a} inclusion holds only if the the soft component consists of disconnected inclusions, see \cite{Zhikov2004}. In the present setting the inverse to \eqref{14a} does not hold in general even under the mentioned geometric assumption. Indeed, if the support of the convolution kernel $a$ is sufficiently large to guarantee the nearby inclusions to ``communicate'', the limiting spectrum is strictly larger than $\Sp(\AA)$. In other words, the two-scale resolvent convergence is too restrictive and does not fully recover the asymptotic behaviour of $\AA_\e$. More specifically, instead of the `periodic' operator $\AA_\soft^\#$ one needs to consider its whole space counterpart $\AA_\soft$ defined below. Moreover, in the case when $S$ is a bounded domain, the spectrum arising from the interaction of the soft component with the boundary of $S$ persists in the limit, but is not accounted for by the operator $\AA$. It is, however, seems impossible to characterise the part of the limiting spectrum arising from the boundary for a general domain $S$.  In what follows, we define relevant objects and summarise our main results concerning spectral convergence.

We define the operator $\AA_\soft:\,L^2(Y^\#_\soft) \to L^2(Y^\#_\soft)$ as the symmetric operator associated with the bilinear form
\begin{eqnarray*}
	a_\soft(z,b)= \int_{\R^n}\int_{\R^n}  a(\xi)p(y, y+\xi)(z(y+\xi)-z(y))(b(y+\xi)-b(y)) \,dy\,d\xi,
	\quad z,b \in L^2(Y^\#_\soft).
\end{eqnarray*}
This operator has important role in characterisation of the limiting spectrum in both cases: when $S$ is bounded or $S=\R^d$.
Note that in contrast to the operator $\AA_\soft^\#$, which acts in the space of periodic functions $L^2_\#(Y_\soft)$, the operator $\AA_\soft$ acts in the space    $L^2(Y_\soft^\#)$.

It is not difficult to see that
\begin{equation*}
	\Sp(\AA_\soft^\#) \subset \Sp(\AA_\soft).
\end{equation*}
Indeed, if $\l \in \Sp(\AA_\soft^\#)$ the one can use a corresponding (approximate) periodic eigenfunction and the cut off function technique analogous to the one in the proof of Theorem \ref{th5.5} below, in order to construct an approximate $L^2$ eigenfunction for $\AA_\soft$. Alternatively, one can employ Gelfand transform resulting in $\Sp(\AA_\soft) = \cup_\theta \Sp(\AA_\soft^\theta)$, see Section  \ref{s7} for the definition of the quasi-periodic operators $\AA_\soft^\theta$. Then the claim follows directly from the observation that $\AA_\soft^\# = \AA_\soft^0$.

Further, we define the operator  $\AA_{\e,\rm soft}:\,L^2(\e Y^\#_\soft \cap S) \to L^2(\e Y^\#_\soft\cap S)$ via the associated bilinear form
\begin{multline}\label{21}
	a_{\e,\rm soft}(z,b)= \int_{\R^d}\int_{\R^d}  a(\xi)p\Big(\frac{x}\e , \frac{x+\e\xi}\e \Big)\left(z(x+\e \xi)-z(x)\right)\left(b(x+\e\xi)-b(x)\right) \,dx\,d\xi,
	\\
	z,b \in L^2(\e Y^\#_\soft \cap S).
\end{multline} 	
 (Note that in the case when $S=\R^d$ the operator $\AA_{\e,\rm soft}$ is unitarily equivalent to  $\AA_\soft$ via the rescaling, hence they have identical spectra.)

The following two assertions hold.

\begin{theorem}\label{th2.5}
	\begin{equation}\label{20}
		\{\beta(\l) \in \Sp (\AA_\hom)\} \cup \Sp(\AA_\soft) \subset \lim_{\e\to 0} \Sp(\AA_\e).
	\end{equation}
\end{theorem}

\begin{theorem}\label{th2.6}
	\begin{equation*}
		\lim_{\e\to 0} \Sp(\AA_\e) \subset	\{\beta(\l) \in \Sp (\AA_\hom)\} \cup  \lim_{\e\to 0}\Sp (\AA_{\e,\rm soft}).
	\end{equation*}
\end{theorem}

In the case when $S = \R^d$, the inclusion \eqref{20} becomes  equality, see Theorem \ref{th2.8} below. On the other hand, for a general domain $S$ the ``boundary layer'' spectrum may behave unpredictably. In general, the task of characterising the boundary layer and the associated spectrum is extremely challenging. In Section \ref{s:6} we provide its analysis for a special case when the set $S$ is a rectangular box with vertices in $\Z^d$, see \eqref{69}, and the sequence $\e = \frac{1}{N}, N\in \N$, so that the geometry of the soft component in the boundary layer is congruent for all  $\e = \frac{1}{N}$.  In this case the limiting spectrum exists in the sense of Hausdorff, and we have the following
\begin{theorem}\label{th2.7}
	\begin{equation*}
		{\rm H}\mbox{-}\lim_{N\to \infty} \Sp(\AA_{1/N}) =	\{ \beta(\l) \in \Sp (\AA_\hom)\}  \cup \left( \cup_{i=1,\dots,2^d} \Sp(\AA_\soft^{v_i})\right).
	\end{equation*}
\end{theorem}
Here $\AA_\soft^{v_i}$ denotes the operator on the part of soft component associated with the $i$-th vertex of $\R^d$, see Section \ref{s:6} and \eqref{utoest5} below for the precise definition. Notice that $\Sp(\AA_\soft) \subset \Sp(\AA_\soft^{v_i})$ for all $i$.

\subsection{ Norm resolvent and spectral convergence bounds for the case  $S=\R^d$}\label{s2.3}

Periodic problems in the whole space is a standard premise for the Floquet-Bloch theory. Applying the scaled version of the Gelfand transform $\G_\e: L^2(\R^d) \to L^2(Y^*\times  Y ), Y^*:= [-\pi,\pi]^d $,
\begin{equation*}
	(\G_\e f)(\theta, y) := \left(\frac{\e^2}{2\pi}\right)^{d/2} \sum_{n\in \Z^d} f(\e(y+n)) \rme^{-\rmi  \e \theta\cdot(y+n)},
\end{equation*}
we obtain the decomposition
\begin{equation}\label{27}
	\G_\e \A_\e  (\G_\e)^{-1} = \int_{Y^*}^\oplus \A_\e^\theta,
\end{equation}
where the self-adjoint operators $\A_\e^\theta$, $\theta \in Y^*$, are associated with the sesquilinear form
\begin{equation*} 
	\begin{split}
		\int_{Y}\int_{\R^d} a(\xi - y) \left( \e^{-2}\Lambda(y,\xi) +  p(y,\xi)\right) (\rme^{\rmi \theta\cdot(\xi - y)} u(\xi)-u(y))  \overline{(\rme^{\rmi \theta\cdot(\xi - y)}v(\xi)-v(y))}d\xi dy,
		\\
		\quad \forall u, v\in   L^2_\#(Y).
	\end{split}
\end{equation*}
The relation \eqref{27} implies that
\begin{equation*}
	\Sp (\AA_\e ) = \cup_{\theta\in Y^*} \Sp (\AA_\e^\theta).
\end{equation*}
Thus, in order to understand the limit behaviour of $\AA_\e$ and its spectrum on can analyse the family of operators $\AA_\e^\theta$ instead.

Following a new approach, recently developed in \cite{CKS2023}, we show that $\A_\e^\theta$ can be approximated in the norm resolvent sense uniformly in $\theta\in Y^*$ by a homogenised operator $\AA_\e^{h,\theta}$, associated with the sesquilinear form
\begin{multline} \label{28}
	\e^{-2}	A^\hom \theta \cdot \theta \, z \overline{\widetilde z}
	\\
	+	\int_{Y}\int_{\R^d} a(\xi - y)  p(y,\xi) (\rme^{\rmi \theta\cdot(\xi - y)} (z+v(\xi))-(z+v(y)))
	\overline{(\rme^{\rmi \theta\cdot(\xi - y)}(\widetilde z+\tilde v(\xi))-(\widetilde z+\tilde v(y)))}d\xi dy,
	\\
	\quad \forall z+v, \widetilde z+\tilde v \in   \C+ L^2_\#(Y_\soft).
\end{multline}

Our results are as follows.

\begin{theorem}  \label{th2.8}
	There exists a positive function $\overline{h}$ satisfying $\overline{h}(t) \to 0$ as $t\to 0$, $\lim_{t\to 0} \overline{h}(t)/t >0$, such that
	\begin{equation}\label{205a}
		\|  (\A_\e^\theta +1)^{-1} - (\AA_\e^{h,\theta}+1)^{-1}\|_{L^2(Y)\to L^2(Y)}\leq C  \overline{h}(\e)
	\end{equation}
	uniformly in $\theta\in Y^*$.
	Moreover,
	\begin{equation*}
		\lim_{\e\to 0} \Sp (\AA_\e) = \mathcal G : = \{\beta(\l) \geq 0\}\cup \Sp(\AA_\soft),
	\end{equation*}
	and, for any $\Lambda >0$, one has
	\begin{equation*}
		d_{H, [0, \Lambda]}\Big( \Sp (\AA_\e), \mathcal G \Big) \leq C(\Lambda) \max\{ \overline{h}(\e), \,\e^{2/3}\},
	\end{equation*}
	where
	\begin{equation*}
		d_{H, [0, \Lambda]}(A_1, A_2):= \max \big(   \dist(A_1  \cap [0, \Lambda], A_2), \,   \dist( A_1, A_2  \cap [0, \Lambda]) \big).
	\end{equation*}
\end{theorem}

The function $\overline{h}$ depends  essentially on the   decay properties of the convolution kernel $a$ at infinity, cf. \eqref{183}, \eqref{189}, \eqref{219} and \eqref{225} below. In case $a$ has a finite third moment, i.e. $a(\xi)|\xi|^3 \in L^1(\R^d)$, then we can set $\overline{h}(t) = t$, see Remark \ref{r6.6} (cf. also  {\cite[Theorem 5.6]{CKS2023}}).

%
\begin{remark}
	It is not difficult to generalize the results of the paper to the case of $d$-dimensional periodicity lattice (see \cite{mielketimofte})
	$$ \Xi= \{ \ell=\sum_{j=1}^d k_j b_j: (k_1,\dots,k_d)\in \mathbf{Z}^d  \}, $$where $\{b_1,\dots,b_d\}$ is an arbitrary basis for $\mathbf{R}^d$. The associated unit cell is
	$$ Y=\{y=\sum_{j=1}^d \gamma_j b_j: \gamma_j \in [0,1), j=1,\dots,d\}\}, $$
	such that $\RR^d$ is the disjoint union of the translated cells $\ell+Y$, if $\ell$ ranges over $\Xi$. $Y^{\#}$ can be then defined as $Y^{\#}=\RR^d/\Xi$ with the quotient topology.  For the case $\Xi=\mathbf{Z}^d$, $Y=[0,1)^d$ we obtain the case discussed here.
\end{remark}

\section{The limit two-scale operator via two-scale convergence}
\label{two-scalecon}
In this section we analyse the corrector problem \eqref{ex102} and the homogenised matrix of the stiff component and prove Theorem \ref{thmandrei1}.

\begin{lemma} \label{lmdefA} $\left.\right.$
	
	\begin{enumerate}
		\item The corrector problem \eqref{ex102} has a unique up to an additive  constant solution.
		\item For $\chi:=(\chi^1,\dots,\chi^d)$ and $\eta \in \R^d$ the function  $ \chi^\eta:=\chi\cdot\eta\in L^2_\#(Y_\stiff)$ is the unique up to an additive constant solution to the problem
			\begin{multline} \label{ex102nak1}
				\int_{Y} \int_{\R^d} a(\xi) \Lambda_0 (y,y+\xi)\left(\xi \cdot \eta+\chi^\eta(y+\xi)-\chi^\eta(y) \right)\left(b(y+\xi)-b(y) \right) \,d\xi\, dy=0
				\\
				\forall b \in L^2_\#(Y).
			\end{multline}
		\item The homogenised matrix of the stiff component $A^\hom $, cf. \eqref{defA}, is  symmetric and positive definite:
\begin{equation*} 
			\tilde{\alpha}_1 |\eta|^2  \leq   A^{\rm hom} \eta \cdot \eta,
\end{equation*}
		for some $\tilde{\alpha}_1 >0$.
		
	\end{enumerate}
\end{lemma} 	
\begin{proof}
	Claim a. follows from the Lax-Milgram theorem upon establishing  the coercivity of the form
	\begin{multline*}
		a_{\rm stiff}(\psi_1,\psi_2):= 	\int_{Y} \int_{\R^d} a(\xi) \Lambda_0 (y,y+\xi)\left(\psi_1(y+\xi)-\psi_1(y) \right)\left(\psi_2(y+\xi)-\psi_2(y) \right) \,d\xi\, dy
		\\
		\psi_1, \psi_2 \in L^2_\#(Y_\stiff)
	\end{multline*}
	for functions with zero mean  on $Y_\stiff$. Applying  Lemma \ref{lemma:crucialest} with $\M= Y_\stiff^\#$, and using the periodicity of $\psi$, we have
	\begin{multline}\label{43}
	a_{\rm stiff} (\psi,\psi) \geq	\frac{\alpha_1 c_a}{(2k)^d} \int_{(2kY_\stiff)^2 \cap D_{r_a}  } (\psi(y)-\psi(\xi))^2 \, dy\, d \xi\\ \geq
	C   \int_{(Y_\stiff)^2} (\psi(y)-\psi(\xi))^2 \,dy d\xi=C \left(\|\psi\|^2_{L^2(Y_\stiff)}-\left(\int_{Y_\stiff}\psi \right)^2\right),
\end{multline}
	which proves the first claim.
	
	Claim b. is a straightforward consequence of the linearity of problem \eqref{ex102}.
	
	Now we address part c. From \eqref{defA} and \eqref{ex102} one has
	\begin{equation*}
		A^{\rm hom}_{ij}=\int_{Y} \int_{\R^d} a(\xi) \Lambda_0 (y,y+\xi)\left(\xi_i+\chi^i(y+\xi)-\chi^i(y) \right)\left(\xi_j+{\chi}^j(y+\xi)-{\chi}^j(y)\right)\,d\xi\,dy,
	\end{equation*}
	which yields the symmetry of $A^{\rm hom}$.
	
	Similarly, from \eqref{ex102nak1} we have
	\begin{equation*}
		  A^{\hom} \eta \cdot \eta =\int_{Y} \int_{\R^d} a(\xi) \Lambda_0 (y,y+\xi)\left( \xi\cdot\eta+\chi^\eta(y+\xi)-\chi^\eta(y) \right)^2 d\xi\,dy
	\end{equation*} 	
	for every $\eta \in \R^d$. Suppose that $  A^{\hom} \eta \cdot \eta   =0$ for some $\eta \neq 0$. Choose $y_0\in \R^d$ such that
	\begin{equation*}
		\int_{y_0+Y_\stiff}   y \cdot\eta  dy = 0.
	\end{equation*}
	Note that one then has $	\int_{y_0+(Y_\stiff^\# \cap \square^{m/2})}   y \cdot\eta  dy = 0$ for any odd $m\in \N$.
	Let $\chi^\eta$  be  zero-mean on $Y_\stiff$ and set $\psi(y) =   y \cdot\eta   + \chi^\eta$. Then arguing as in \eqref{43} via Lemma \ref{lemma:crucialest}, for  $m\in \N$, we obtain
	\begin{equation*} 
		0=  A^{\hom} \eta \cdot \eta   \geq C \|\psi\|^2_{L^2(y_0+(Y_\stiff^\# \cap \square^{m/2}))} \mbox{ for some } C=C(m)>0.
	\end{equation*} 	
	We arrive at a contradiction due to the periodicity of $\chi^\eta$, which completes the proof.
\end{proof}

\noindent {\bf Proof of Theorem \ref{thmandrei1}.}
We give the proof for the case $S=\mathbf{R}^d$. The case of bounded $S$  can be dealt with in an analogous way, see also Remark \ref{marin1}.

Assume first that $f_\e \wtto f \in L^2(S\times Y)$ and consider the corresponding sequence of solutions $u_\e$  to \eqref{starting}. Applying Corollary \ref{corcan3} we
consider the decomposition
\begin{equation} \label{decomp1}
 u_\e  = \bar{u}_\e +\e \hat{u}_\e +z_\e ,
 \end{equation}
with 	$\bar{u}_\e  \in H^1(\RR^d)\cap C^\infty(\R^d)$, $\hat{u}_\e  \in L^2(\RR^d)$ and $z_\e \in L^2(\e Y_\soft^\#)$ satisfying
\begin{equation} \|\bar{u}_\e \|_{H^1(\RR^d)} \leq C, \quad\|\hat{u}_\e \|_{L^2(\RR^d)} \leq C, \quad \|z_\e \|_{L^2(\e Y_\soft^\#)} \leq C.  
\end{equation}
By the basic properties of two-scale convergence, we have, up to a subsequence,
\begin{equation}\label{52}
	\bar{u}_\e  \xrightharpoonup{2} u_0(x),\quad \nabla  \bar{u}_\e  \xrightharpoonup{2} \nabla u_0(x)+\nabla_y \bar{u}_1(x,y),\quad  \hat{u}_\e \xrightharpoonup{2} \hat{u}_1(x,y), \quad z_\e  \xrightharpoonup{2} z(x,y),
\end{equation}
for some $u_0\in H^1(\R^d)$, $\bar{u}_1 \in L^2(\R^d;H^1_\#(Y))$, $\hat{u}_1, z, \in L^2(\R^d;L^2(Y_\#))$.  In particular, one has
$$u_\e \xrightharpoonup{2} u_0(x) +z(x,y).  $$
Note that
\begin{eqnarray*}
		\bar u_\e &\rightharpoonup u_0& \mbox{ weakly  in } H^1(\R^d).
\end{eqnarray*}

 Using a change of variables we rewrite the equation \eqref{starting} in the form
\begin{multline} \label{ex0}
	\int_{\R^d} \int_{\R^d} a(\xi)\Lambda_\e (x, x+\e\xi)\left(\frac{u_\e (x+\e \xi)-u_\e (x)}\e  \right) \left(\frac{v(x+\e \xi)-v(x)}\e  \right)\,d\xi\,dx
	\\
	- \lambda \int_{\R^d} u_\e  (x)v(x) dx=\int_{\R^d} f_\e  (x) v(x) dx.
\end{multline}
In order to recover the structure of the two-scale limits $\bar{u}_1$ and $\hat u_0$, we pass to the limit as $\e \to 0$ in \eqref{ex0}  with the test functions of the form $\e \varphi(x) b(x/\e)$,  $\varphi \in C_0^\infty(\mathbf{R}^d)$, $b \in C_\#(Y)$.
Note that by the fundamental theorem of calculus we have (recall that  $\bar{u}_\e $ is smooth)
\begin{equation} \label{formula11}
	\frac{\bar{u}_\e (x+\e \xi)-\bar{u}_\e (x)}\e =\int_0^1 \xi \cdot\nabla \bar{u}_\e (x+\e    t \xi) \,dt.
\end{equation}
Then taking into account the decomposition \eqref{decomp1}, it is easy to see that
\begin{multline}\label{56}
\lim_{\e \to 0} \int_{\R^d} \int_{\R^d} 	a(\xi)\Lambda_0({x}/\e , {x/\e+\xi}) \Big[ \int_0^1  \nabla \bar{u}_\e (x+\e    t \xi) \cdot\xi  \,dt+\hat{u}_\e (x+\e \xi)-\hat{u}_\e (x)\Big]
\\
 \Big(\varphi(x+\e\xi)   b(x/\e+\xi )-\varphi(x)b(x/\e)\Big) \, d\xi\,dx 	= 0
\end{multline} 	
Passing to the limit  in \eqref{56} via Lemma \ref{twoscalelimit1} and Corollary \ref{corandrei1}, and taking into account \eqref{52}, yields
\begin{multline}\label{57a}
	 \int_{\R^d} \int_{\R^d}\int_{Y} \Lambda_0 (y,y+\xi) a(\xi) \left(  \nabla u_0(x) \cdot \xi +(\bar{u}_1+\hat{u}_1)(x,y+\xi)-(\bar{u}_1+\hat{u}_1)(x,y) \right)
	 \\
	 \varphi(x)  \left(b(y+\xi)-b(y) \right) \,dy\, d\xi \, dx=0 \quad \forall \varphi \in C_0^\infty(\mathbf{R}^d), \,b \in C_\#(Y).
\end{multline} 	
By the density argument, it follows that for a.e. $x$ the function $(\bar{u}_1+\hat{u}_1)(x,\cdot)$ solves the corrector equation \eqref{ex102nak1} with $\eta = \nabla u_0(x)$. In particular,
\begin{equation*} 
u_1:=	\bar{u}_1+\hat{u}_1=  \nabla u_0 \cdot \mathbf{\chi}  .
\end{equation*}
Furthermore, by direct inspection, cf. \eqref{defA}, for any $v \in H^1 (\R^d)$, one has
\begin{multline} \label{estimate72}
 	A^{\rm hom} \nabla u_0(x) \cdot \nabla v(x)
\\
=	\int_{Y} \int_{\R^d} a(\xi) \Lambda_0 (y,y+\xi)\left(  \nabla u_0(x) \cdot\xi +u_1(x,y+\xi)-u_1(x,y) \right)   \nabla v(x) \cdot\xi   d\xi dy
\end{multline}
for a.e. $x \in \R^d$.

We next show the validity of  \eqref{ruc1}. To this end  we pass to the limit in  \eqref{ex0} via two-scale convergence with a test function  $v \in C_0^\infty(\R^d) $. Note that the term containing $ p(x,x+\e\xi)$ on the left-hand  side of \eqref{ex0} vanishes in the limit, since $v(x+\e \xi)-v(x) = \e \int_0^1  \nabla v (x+\e    t \xi) \cdot\xi  \,dt$ (compare with \eqref{formula11}). Then resorting to Lemma \ref{twoscalelimit1} and Corollary \ref{corandrei1}, we infer, similarly to \eqref{ex0}--\eqref{57a}, that the first integral on the left-hand  side of \eqref{ex0} converges to the right-hand  side of \eqref{estimate72}.  The convergence of the remaining two integrals is straightforward. Thus, we obtain \eqref{ruc1} with the test functions from $C_0^\infty(\R^d)$, and by density argument it holds for all $v\in H^1(\R^d)$.

In the last step of the proof we derive \eqref{ex200}. Taking  in \eqref{ex0}  test functions of the form $\varphi(x)b(x/\e)$,  $\varphi \in L^2(\mathbf{R}^d)$, $b \in C_\#(Y_\soft)$, yields
\begin{multline}\label{rok1}
 \int_{\R^d} \int_{\R^d} a(\xi)p({x}/\e , {x/\e+\xi} )\left(u_\e (x+\e \xi)-u_\e (x) \right) \left(\varphi(x+\e \xi)b(x/\e+\xi )-\varphi(x)b({x}/\e ) \right)\,d\xi\,dx
 \\
 - \lambda \int_{\R^d}  u_\e  (x)\varphi(x)b({x}/\e ) dx=\int_{\R^d}  f_\e  (x) \varphi(x) b({x}/\e ) dx.
\end{multline}
Passing once again to the limit via Lemma \ref{twoscalelimit1} and Corollary \ref{corandrei1}, while taking into account \eqref{decomp1}--\eqref{52} and \eqref{formula11}, we arrive at
\begin{multline}\label{recal1}
\int_{\R^d}\varphi(x) \int_{\R^n}\int_{Y}  a(\xi)p(y, y+\xi)(z(x,y+\xi)-z(x,y))(b(y+\xi)-b(y)) \,dy\,d\xi\, dx
\\
-\lambda \int_{\R^d} \varphi(x) \int_{Y} (u_0(x)+z(x,y)) b(y)=\int_{\R^d} \varphi(x) \int_{Y} f(x,y) b(y). 	
\end{multline} 	
Then \eqref{ex200} follows by the density argument.

To complete the proof we invoke the following classical result whose proof can be found in e.g. \cite{Past}. (Recall the definition of the projector $\mathcal P$ in Remark \ref{remandrei3}.)

\begin{proposition}
	\label{prop510}
	Weak two-scale resolvent convergence is equivalent to the strong two-scale resolvent convergence, i.e. the following properties are equivalent:
	\begin{enumerate}
		\item If $f_\e \xrightharpoonup{2} f$, then for every $\lambda>0$, $(\AA_\e+ \lambda I)^{-1} f_\e \xrightharpoonup{2} (\AA+\lambda I)^{-1} \mathcal{P}f$;
		\item If $f_\e \xrightarrow{2} f$, and $f \in L^2(S)+L^2(S\times Y_\soft)$, then $(\AA_\e+ \lambda I)^{-1} f_\e \xrightarrow{2} (\AA+\lambda I)^{-1} f$.
	\end{enumerate}
\end{proposition}


\section{Spectrum of the limit operator}\label{s5}
In this section we characterise the spectrum of the limit two-scale operator.
We begin with the analysis of the spectrum of $\AA_\soft^\#$ and its relation to the spectrum of $\AA$.
\begin{proposition} \label{propsri1}
	$$\Sp  (\AA_\soft^\# ) \subset  	\Sp  (\AA).$$
\end{proposition} 	
\begin{proof}
	We adapt the argument from \cite[Proposition 4.1]{CCV23}.
	Suppose that $\l\in \R$ is in the resolvent set of $\AA$, so that \eqref{ruc1} and \eqref{ex200}   has a solution $u_0+z$  for any $f\in L^2(S \times Y)$.  First we take a non-trivial $f \in L^2(S)\setminus H^1(S)$ (i.e. we assume that $f$ is constant in variable $y$). Note that in this case $ \lambda u_0 +f$ does not vanish. Then the equation on the soft component reads (see Remark \ref{remopeq})
	$$ \AA_\soft^\#  z-\lambda z=(\lambda u_0+ f)\mathbf{1}_{Y_\soft }.$$
	Note  that for two arbitrary functions $w\in L^2(S\times Y_\soft )$ and $h\in L^2(S)$ one has $\int_S w h \in L^2(Y_\soft )$. Therefore, multiplying the above identity by ${(\l u_0 +f)}\|\l u_0 +f\|^{-2}_{L^2(S)}$  and integrating over $S,$ we conclude that the function
	\begin{equation*}
		\phi:= \frac{\int_S z {(\l u_0 +f)}\,dx}{\|\l u_0 +f\|^2_{L^2(S)}} \in L^2_\#(Y_\soft  )
	\end{equation*}
	solves the equation
	$$ \AA_\soft^\#  \phi-\lambda \phi={\mathbf 1}_{Y_\soft }.$$
	Next, we take $f = g \psi$ with arbitrary non-trivial $g\in L^2(S)$ and $\psi\in L^2(Y_\soft )$. Then for the corresponding solution of \eqref{ruc1}-\eqref{ex200}, which we denote by $\widetilde u_0 + \widetilde z$, the problem on the soft component reads
	$$ \AA_\soft^\#  \widetilde z-\lambda \widetilde z=\lambda \widetilde u_0 {\mathbf 1}_{Y_\soft }+g\psi. $$
	The difference between $\widetilde z$ and $\hat{z}:=\lambda\widetilde u_0\phi$
	satisfies
	$$ \AA_\soft^\#  (\widetilde z-\hat{z})-\lambda (\widetilde z-\hat{z})=g\psi.$$
	Multiplying the last equation by ${g}\|g\|^{-2}_{L^2(S)}$ and integrating the resulting identity over $S,$ we see that	the function
	\begin{equation*}
		\breve z: =  \frac{\int_S(\widetilde z-\hat{z}\,) {g}dx}{\|g\|^2_{L^2(S)}}
	\end{equation*}
	is a solution of
	\begin{equation*}
		\AA_\soft^\#   \breve z -\lambda \breve z=\psi.
	\end{equation*}
	Since $\psi\in L^2(Y_\soft )$ is arbitrary, the operator $\AA_\soft^\#  - \l I$ acts onto, therefore, by the bounded inverse theorem one concludes that $(\AA_\soft^\# -\l I)^{-1}$ is bounded. Indeed, since   $\AA_\soft^\#  - \l I$ is onto and self-adjoint, the kernel of  $\AA_\soft^\#  - \l I$ is trivial, hence the operator is injective.
\end{proof} 	

It order to characterise the structure of the spectrum of $\AA_\soft^\#$, it is natural to consider the following decomposition:
\begin{equation}\label{operator2}
    \AA_\soft^\# =  \AA_\soft^{\#,1}- \AA_\soft^{\#,2},
\end{equation}
where the operators $\AA_\soft^{\#,1}$ and $\AA_\soft^{\#,2}$ on $L^2_\#(Y_\soft)$ are defined by (cf. \eqref{14})
 $$ \AA_\soft^{\#,1} z(x)=m(x) z(x), \quad \AA_\soft^{\#,2} z(x)=\int_{Y} K(x,y)z(y) \,dy, \quad z \in L^2_\#(Y_\soft), $$
  with
 \begin{equation} \label{andrey1}
 m(x):=2\int_{Y} \tilde{a}(y-x) p(x,y)\,d y, \quad K(x,y):=2 \tilde{a}(y-x)p(x,y)\mathbf{1}_{Y_\soft^{\#}}(x)\mathbf{1}_{Y_\soft^{\#}}(y),
 \end{equation}
 and
 \begin{equation} \label{andrei100}
 \tilde{a}(x):=\sum_{j \in \mathbf{Z}^d} a(x+j).
 \end{equation}
Since $a$ is even,  $\tilde{a}$ is also an even function. Moreover, it is $Y$-periodic by construction.

 The spectrum of $\AA_\soft^{\#,1}$ is purely essential and coincides with the essential range of the function $m \in L^\infty_{\#}(Y)$. We next argue that $\AA_\soft^{\#,2}$ is compact. Recall the  Schur test, see e.g. \cite[Theorem 6.18]{Foll99}:  
 \begin{theorem} \label{thmfolland}
 	Let $(X,\mathcal{M},\mu)$ and $(Y,\mathcal{N},\nu)$ be $\sigma$-finite measurable space and let $K$ be $\mathcal{M} \otimes \mathcal{N}$ measurable function on $X \times Y$. Suppose that there exists $C>0$ such that $\int_X|K(x,y)|\,d\mu(x) \leq C$ for almost every $y \in Y$ and $\int_Y|K(x,y)|\,d\nu(y) \leq C$ for almost every $x \in X$, and that $1\leq p \leq \infty$. If $f \in L^p(Y)$, the integral
 	$$ \mathcal{T}f(x)=\int_Y K(x,y) f(y) \,d \nu(y)   $$
 	is finite for almost every $x \in X$, the function $\mathcal{T}f$ belongs to $L^p(X)$, and $\|\mathcal{T}f\|_{L^p(X)} \leq C\|f\|_{L^p(Y)}$.
 \end{theorem}
 The following result is also well known, see e.g. \cite{Halmos}.
 \begin{theorem} \label{propHS}
 	Let $(X,\mathcal{M},\mu)$  be $\sigma$-finite measurable space and let
 	$\mathcal{M}$  be countably generated. Let
 	$K$ be $\mathcal{M} \otimes \mathcal{M}$ measurable function on $X \times X$ such that $\int_X \int_X |K(x,y)|^2 \,d\mu(x)\,d\mu(y)< \infty$. Then the operator $\mathcal{T}:L^2(X)\to L^2(X)$ defined by
 	$$ \mathcal{T}f(x)=\int_Y K(x,y) f(y) \,d \nu(y)   $$
 	is of Hilbert-Schmidt class and thus compact.
 \end{theorem} 	
 \begin{proposition}
The operator $\AA_\soft^{\#,2}$ is compact. 	
 \end{proposition} 	
 \begin{proof}
 	It is easily seen that
 	$$\esssup_y\int_Y |K(x,y)|\,dx = \esssup_x\int_Y |K(x,y)|\,dy\leq 2\|p\|_{L^\infty}\|\tilde{a}\|_{L^1(Y)}=2\|p\|_{L^\infty}\|a\|_{L^1(\mathbf{R}^d)}. $$
 	Hence, by Theorem \ref{thmfolland} the operator $\AA_\soft^{\#,2}$ is bounded:
 	$$\|\AA_\soft^{\#,2}\|_{L^2(Y_\soft) \to L^2(Y_\soft)} \leq  2\|p\|_{L^\infty}\|a\|_{L^1(\mathbf{R}^d)}. $$
 	For each $n\in \N$ we define
 	\begin{eqnarray*}
 	& & \tilde{a}_n:=\tilde{a}\mathbf{1}_{\{\tilde{a}\leq n\}},\quad K_n(x,y):=\tilde{a}_n(y-x) p(x,y)\mathbf{1}_{Y_\soft^\#}(x)\mathbf{1}_{Y_\soft^\#}(y),\\ & &\hspace{+3ex} \AA_\soft^{\#,2,n}z(x):=\int_Y K_n(x,y)z(y)\,dy,\quad \forall z \in L^2_\#(Y_\soft).
 	 \end{eqnarray*}
 	 By  Theorem \ref{propHS} the operators $\AA_\soft^{\#,2,n}$ are compact.  Moreover,
 	\begin{equation} \label{vulk2}
 	\|\tilde{a}_n-\tilde{a}\|_{L^1 (Y)} \xrightarrow{n \to \infty}0, \quad \int_Y|K_n(x,y) -K (x,y)|\,dx\leq C\|\tilde{a}_n-\tilde{a}\|_{L^1(Y)},\quad \forall y \in Y.
 	\end{equation}
 	Then Theorem \ref{thmfolland} and \eqref{vulk2} imply
 	$$\|\AA_\soft^{\#,2,n}-\AA_\soft^{\#,2}\|_{L^2\to L^2}\to 0.$$
 	Hence $\AA_\soft^{\#,2}$ is   compact.
 \end{proof} 	

Thus, the operator $\AA_\soft^\#$ is a compact perturbation of the multiplication operator $\AA_\soft^{\#,1}$, which implies the following characterisation of the spectrum.
\begin{equation*}
\Sp (\AA_\soft^\# )=\textrm{EssRan } m \cup\{\mu_1, \mu_2\dots \},	
\end{equation*} 	
where  $\{\mu_1,\mu_2,\dots\}$ is the discrete set (possibly empty or finite) that may  have accumulation points only in $\partial (\textrm{EssRan } m)$.
If $\{\mu_1',\mu_2',\dots\} \subset \{\mu_1,\mu_2,\dots\}$ is the set of all eigenvalues such that
$\sup\{\mu_1',\mu_2',\dots\} \leq \inf \textrm{EssRan } m, $ then we can enumerate the set $\{\mu_1',\mu_2',\dots\}$ in a non-decreasing order (accounting for multiplicity). In this case $\mu_k'$ can be characterised by the Rayleigh quotients:
\begin{equation} \label{representation}
 \mu_k'=\min_{V\subset L^2(Y_\soft ), \textrm{ dim} V=k}\max_{z \in V}\frac{a_\soft^\#(z,z)}{\|z\|_{L^2(Y_\soft )}^2}.
 \end{equation}
Similarly, for the set  $\{\mu_1'',\mu_2'',\dots\} \subset \{\mu_1,\mu_2,\dots\}$ of all eigenvalues such that
$\inf\{\mu_1'',\mu_2'',\dots\} \geq \sup \textrm{EssRan } m, $ enumerated  in the non-increasing order (accounting for multiplicity) one has
\begin{equation} \label{representation22}
	\mu_k''=\max_{V\subset L^2(Y_\soft ), \textrm{ dim} V=k}\min_{z \in V}\frac{a_\soft^\#(z,z)}{\|z\|_{L^2(Y_\soft )}^2}.
\end{equation}

We next study the properties of the function $\beta$ (see \eqref{ruc22}), which is an essential element of the characterisation of the spectrum of $\AA$. Denote by $L \subset \Sp (\AA_\soft^\# )$ the support of the measure $\mu:=(E_s\mathbf{1}_{Y_\soft^\# },\mathbf{1}_{Y_\soft })$,  where $E_s$ is the resolution of the identity associated with the operator $\AA_\soft^\# $ (cf. Remark \ref{remopeq}).
\begin{proposition}\label{p:5.7}
The function $\beta$ can be naturally extended to $L^c$. It is differentiable on its domain and its derivative is always positive: more precisely,  $\beta'(\lambda) \geq 1-|Y_\soft |$.  Moreover, if $\l_0 \in  L$ is such that $\mu(\{\l_0\})>0$ and there exists $\delta>0$ with  $(\l_0 -\delta,\l_0) \cap L = \varnothing$ (respectively, $(\l_0,\l_0 +\delta) \cap L = \varnothing$),  then  $\lim_{\l \to \l_0^-}\beta(\lambda)=+\infty$ (respectively, $\lim_{\l \to \l_0^+} \beta(\l)=-\infty$).
\end{proposition} 	
\begin{proof}
Using the spectral decomposition we obtain
\begin{equation}\label{beta_defi}
\beta(\l)=\l +\l^2 \int_{\Sp (\AA_\soft^\# )} \frac{1}{s-\l}(dE_s\mathbf{1}_{Y_\soft^\# }, \mathbf{1}_{Y_\soft })=\l +\l^2 \int_{L} \frac{1}{s-\l}d\mu(s). 	
\end{equation} 		
This formula naturally extends $\beta$ onto $L^c$.
 Differentiating with respect to $\l$ the expression on the right-hand side of \eqref{beta_defi} yields
\begin{eqnarray*}
\beta'(\lambda) &=& 1+2\l \int_{L} \frac{1}{s-\l}\,d\mu(s)+\l^2 \int
_{L} \frac{1}{(s-\l)^2}\,d\mu(s)\\ &=& 1-\int_{L} \,d\mu(s)+\int_{L}\frac{s^2}{(s-\l)^2}\, d\mu(s)\\ &\geq& 1-|Y_\soft |.
\end{eqnarray*} 	
In order to prove the second claim, assume that $(\l_0-\delta,\l_0) \cap L = \emptyset$ for some $\delta >0$ (the other case is treated analogously) and consider the  decomposition
\begin{eqnarray*}
\beta(\l)=\l +\l^2 \int_{\{\l_0\}} \frac{1}{s-\l}d\mu(s)+\l^2 \int_{(-\infty,\l_0)} \frac{1}{s-\l}d\mu(s)+\l^2 \int_{(\l_0,+\infty)} \frac{1}{s-\l}d\mu(s).
\end{eqnarray*}
Clearly,  as $\l \to \l_0^-$, the integral
$
 \int_{(-\infty,\l_0)} (s-\l)^{-1}d\mu(s)
$
 remains bounded,
$
	\int_{(\l_0,+\infty)} (s-\l)^{-1}d\mu(s)
$
 is non-negative (possibly tends to $+\infty$), and
$
	\int_{\{\l_0\}} (s-\l)^{-1}d\mu(s)
$
tends to $+\infty$.
\end{proof} 	

We are ready to prove Theorem \ref{thm2.3}, which characterises the spectrum of the limit two-scale operator.
\subsection{Proof of Theorem \ref{thm2.3}}
Assume first that $\l \in (Sp(\AA))^c$,  and let $u+z \in H^1_0(S)+ L^2_\#(S\times Y_\soft)$ be the solution to the coupled problem \eqref{ruc1}-\eqref{ex200} for some $f\in L^2(S)$. By Proposition \ref{propsri1} we have that   $\l \in (Sp(\AA_\soft^\# ))^c$, and by Remark \ref{remopeq} we have
$$ z(x,\cdot)=(\lambda u+f)(\AA_\soft^\# -\l I)^{-1}  \mathbf{1}_{Y_\soft} = (\lambda u+f) b_\l. $$
Plugging this into \eqref{ruc1} yields
\begin{equation}\label{71}
	(-\nabla \cdot A^{\rm hom} \nabla -\beta(\l)I) u= (1+ \l  \langle b_\l \rangle) f.
\end{equation}
Note that  $1+\l \langle b_\l \rangle \neq 0$. Indeed, otherwise \eqref{ruc1}-\eqref{ex200} would have infinitely many solutions.
Then, since problem \eqref{71} is uniquely solvable for any $f \in L^2(S)$, we conclude that $\beta(\l) \notin \Sp (\A_\hom)$.\footnote{Here we again appeal to the bounded inverse theorem, applying it to the operator $\A_\hom: \dom(\A_\hom) \to L^2(S)$, where $\dom(\A_\hom)$ is equipped with the graph norm.}

Now assume that $\l$ is in the complement to the set $\Sp  (\AA_\soft^\# ) \cup \{\beta(\l) \in \Sp (\AA_\hom)\}.$
For an arbitrary $f  \in L^2_\#(S \times Y)$, define $u\in H^1_0(S)$ and $z\in L^2_\#(S\times Y_\soft)$ to be the (unique) solutions to the problems
\begin{equation*} 
	-\nabla  \cdot A^{\rm hom} \nabla u-\b(\l)u=\int_Y \Big( f(\cdot,y) +\l 	(\AA_\soft^\# -\l I)^{-1} f(\cdot,y)\mathbf{1}_{Y_\soft}(y) \Big)\,dy,
\end{equation*} 	
and
\begin{equation*} 
\AA_\soft^\#  z(x,\cdot) -\l z(x,\cdot)= \lambda u(x)  \mathbf{1}_{Y_\soft}	 + f(x,\cdot)\mathbf{1}_{Y_\soft},
\end{equation*} 	
respectively. By direct inspection we see that $u+z$ is the unique solution to the coupled problem \eqref{ruc1}-\eqref{ex200}, hence $\l \in (\Sp (\A))^c$.

\begin{remark}
Note that in the case  $S=\mathbf{R}^d$ we have $\Sp (\A_\hom)=[0,+\infty)$, while in the case of a bounded Lipschitz domain $S$  the spectrum of  $\A_\hom$ is discrete with the only accumulation point at  $+\infty$.
\end{remark}

\subsection{Examples}

The spectrum of the limit operator $\AA$ crucially depends on the spectrum of $\AA_\soft^\#$ via the function $\b$. In particular, $\Sp(\AA)$ is guaranteed to have gaps in the case $S=\R^d$ if $\Sp(\AA_\soft^\#)$ has non-empty discrete spectrum, cf. Proposition \ref{p:5.7}. Therefore, it is important to know what $\Sp(\AA_\soft^\#)$ may look like. In the remainder of the section we provide several examples illustrating various possibilities for the structure  of the spectrum of the microscopic operator   $\AA_\soft^\# $.

\subsubsection{Example of $\AA_\soft^\# $ with purely essential spectrum}
In this  section we construct a $1$-dimensional example of an operator $\AA_\soft^\# $ that has  an empty discrete spectrum.
To simplify the notation, we will work with a shifted periodicity cell $Y=[-\frac{1}{2},\frac{1}{2})$ in this example and take $Y_\soft =(-\frac{1}{4},\frac{1}{4})$.
Define
$$
a(y)=\left\{
\begin{array}{ll}
\frac12\left(\frac y 2+\frac34\right), &\hbox{if }y\in[-\frac32,-\frac12],\\[2mm]
	\frac14,&\hbox{if }y\in[-\frac12,\frac12],\\[2mm]
	\frac12\left(-\frac y 2+\frac34\right), &\hbox{if }y\in[\frac12,\frac32],\\[2mm]
	0,&\hbox{otherwise}.
\end{array}
\right.
$$
Clearly, $\tilde a(y)=\frac12$ for all $y$ (cf. \eqref{andrei100}).
We also set \begin{equation}\label{60}
p(y,\xi)=w(y)w(\xi)\big(1-\mathbf{1}_{Y_\stiff^{\#}}(y)\mathbf{1}_{Y_\stiff^{\#}}(\xi)\big),
\end{equation} where $w$ is a continuous,   positive, $1$-periodic   function   such that
\begin{equation}\label{73a}
	\int_{Y} w(y)dy=1.
\end{equation}
Then, see \eqref{operator2},\eqref{andrey1},
$$
\AA_\soft^\# z(y)=w(y)z(y)-w(y)\int_{Y_\soft } w(\xi) z(\xi)d\xi,\qquad y\in {Y_\soft },
$$
 and the quadratic form associated with  $\AA_\soft^\# $ reads
$$
a_\soft^\#( z,z )= \int\limits_{Y_\soft } w(y) z^2(y)dy-\bigg(\int\limits_{Y_\soft } w(y)z(y)dy\bigg)^2.
$$
Denote
\begin{equation*}
	\delta: = \min_{Y_\soft} w, \qquad w^+: = \max_{Y_\soft} w.
\end{equation*}
Then the essential spectrum of $\AA_\soft^\# $ is given by ${\rm EssRan} \, w = [\delta,w^+]$.

Clearly,
\begin{equation}\label{73}
	a_\soft^\#( z,z )\leqslant w^+ \|z\|^2_{L^2(Y_\soft)}, \quad \forall z\in L^2_\#(Y_\soft).
\end{equation}
Next we will make a specific choice of the function $w$ in such a way that
\begin{equation}\label{74}
	\delta\|z\|^2_{L^2(Y_\soft)}\leqslant a_\soft^\#( z,z ), \quad \forall z\in L^2_\#(Y_\soft).
\end{equation}
Together with \eqref{73} this will imply that the discrete spectrum of  $\AA_\soft^\# $ is empty for the choice of $w$.

Let $\gamma < \frac{1}{3}$ be a small positive number. We introduce a sequence
$$
b_0=\frac14,\quad b_1=\frac12\frac{\gamma}{1-\gamma}=\frac12\sum_{k=1}^{\infty}\gamma^k, \quad
b_j=\gamma^{j-1} b_1 = \frac12\sum_{k=j}^{\infty}\gamma^k\ \hbox{for }j\geqslant 2,
$$
and consider the intervals $J_j=[b_{j+1},b_j] $, for $j \in \mathbf{N}_0$. Observe that  $b_j\to 0$ as $j\to\infty$ and $|J_j|=\frac12\gamma^j$ for $j\geqslant 1$.
We define $w$ as follows. On the interval $[b_1,\frac12]$ we set $w \equiv w_0$, where $w_0$ is a constant to be defined later, at each $b_j, j\geq 2,$ we set $w(b_j ) = w_j : = \frac{1}{j }+\delta$ and require $w$ to be affine on each interval $J_j, j\geq 1$. Moreover, we define $w_1=w_0$ for consistency of the notation. Finally,
we set $w(\frac12)=\delta$, extend $w$ by symmetry to $[-\frac12,0]$, i.e. $w(-y)=w(y)$, and choose the constant $w_0$ so that \eqref{73a} holds.

One can easily check that  $1<w_0<1+2\gamma$, for $\gamma<\frac14$, $\delta<\frac12$.
By construction, the function $w$ is continuous, $\min w=\delta$, $\max w=w_0$.

Clearly, \eqref{74}
is equivalent to the inequality
\begin{equation}\label{ineq1}
	\bigg(\int\limits_{Y_\soft } w(y)z(y)dx\bigg)^2\leqslant \int\limits_{Y_\soft } (w(y)-\delta) z^2(y)dy.
\end{equation}
  Since
$$
\bigg(\int\limits_{Y_\soft } w(y)z(y)dy\bigg)^2\leqslant \bigg(\int\limits_{Y_\soft } w(y)|z(y)|dy\bigg)^2,
$$
it is sufficient to prove \eqref{ineq1} for non-negative $z\in L^2_\#(Y_\soft)$ only.
Denote
$$
 z_j: =\frac 1{2|J_j|}\int\limits_{-J_j\cup J_j}z(y)dy.
$$
Taking into account the inequality $w_{j+1}\leqslant w(y)\leqslant w_j $ for $y\in -J_j\cup J_j$, we have
$$
\int\limits_{Y_\soft } w(y)z(y)dx\leqslant \sum_{j=0}^{\infty}2|J_j|w_j z_j \quad\hbox{and}\quad\int\limits_{Y_\soft } (w(y)-\delta) z^2(y)dy
\geqslant  \sum_{j=0}^{\infty}2 |J_j|(w_{j+1}-\delta)
 z_j^2,
$$
where we have used the positivity of $z$ and $w_j-\delta$, $j\in \N$.   Therefore, it is sufficient to prove that
\begin{equation}\label{ineq2}
	\Big(\sum_{j=0}^{\infty}2|J_j|w_j z_j\Big)^2\leqslant \sum_{j=0}^{\infty}2 |J_j|(w_{j+1}-\delta)  z_j^2.
\end{equation}

For sufficiently small $\gamma$
and $\delta$ condition \eqref{73a} implies the bounds
$$
2|J_0|w_0<0.55,\qquad 2|J_0|(w_1-\delta)=2|J_0|(w_0-\delta)>0.45.
$$
Then, using the inequality $(\alpha+\beta)^2\leqslant \frac{10}9\alpha^2+10\beta^2$ we obtain
\begin{equation*}
	\begin{gathered}
		\Big(\sum_{j=0}^{\infty}2|J_j|w_j z_j\Big)^2 \leqslant \frac{10}{9}\big(2|J_0|w_0z_0\big)^2+10\Big(\sum_{j=1}^{\infty}2|J_j|w_j z_j\Big)^2
		\\
		\leqslant
		\frac{10}9 0.55^2z_0^2+10\Big(\sum_{j=1}^{\infty}2|J_j|w_j z_j\Big)^2
		\leqslant
		0.45z_0^2+10\Big(\sum_{j=1}^{\infty}2|J_j|w_j z_j\Big)^2
		\\
		\leqslant
		2|J_0|(w_0-\delta)z_0^2+10\Big(\sum_{j=1}^{\infty}2|J_j|w_j z_j\Big)^2.
	\end{gathered}
\end{equation*}

It remains to show that for sufficiently small $\gamma$  the inequality
\begin{equation}\label{ineq3}
	10\Big(\sum_{j=1}^{\infty}2|J_j|w_j z_j\Big)^2\leqslant  \sum_{j=1}^{\infty}2 |J_j|(w_{j+1}-\delta)  z_j^2 = \sum_{j=1}^{\infty}\frac{\gamma^j}{j+1}z_j^2
\end{equation}
holds (cf. \eqref{ineq2}). This can be done as follows:
$$
\Big(\sum_{j=1}^{\infty}2|J_j|w_j z_j\Big)^2=\Big(\sum_{j=1}^{\infty}\gamma^j\frac{w_j}{(w_{j+1}-\delta)^\frac12}
(w_{j+1}-\delta)^\frac12 z_j \Big)^2
$$
$$
\leqslant \Big(\sum_{j=1}^{\infty}\gamma^{j/2}\frac{\frac1j+\delta}{\frac1{(j+1)^\frac12}}
\,\gamma^{j/2}\frac1{(j+1)^\frac12} z_j \Big)^2\leqslant
\Big(\sum_{j=1}^{\infty}\gamma^{j}(j+1)\big(\frac1j+\delta\big)^2\Big)
\Big(\sum_{j=1}^{\infty}\gamma^{j}\frac1{j+1}z_j^2\Big)
$$
For small enough $\gamma$  we have
$$
\Big(\sum_{j=1}^{\infty}\gamma^{j}\big[(j+1)\big(\frac1j+\delta\big)\big]^2\Big)<\frac1{10},
$$
and inequality \eqref{ineq3} follows.

\subsubsection{Example of $\AA_\soft^\# $ with   infinitely many points of discrete spectrum}
\begin{wrapfigure}{R}{0.4\textwidth}
	\vspace{-5mm}
		\includegraphics[width=0.4\textwidth, scale=0.2, bb=0 0 1000 1000]{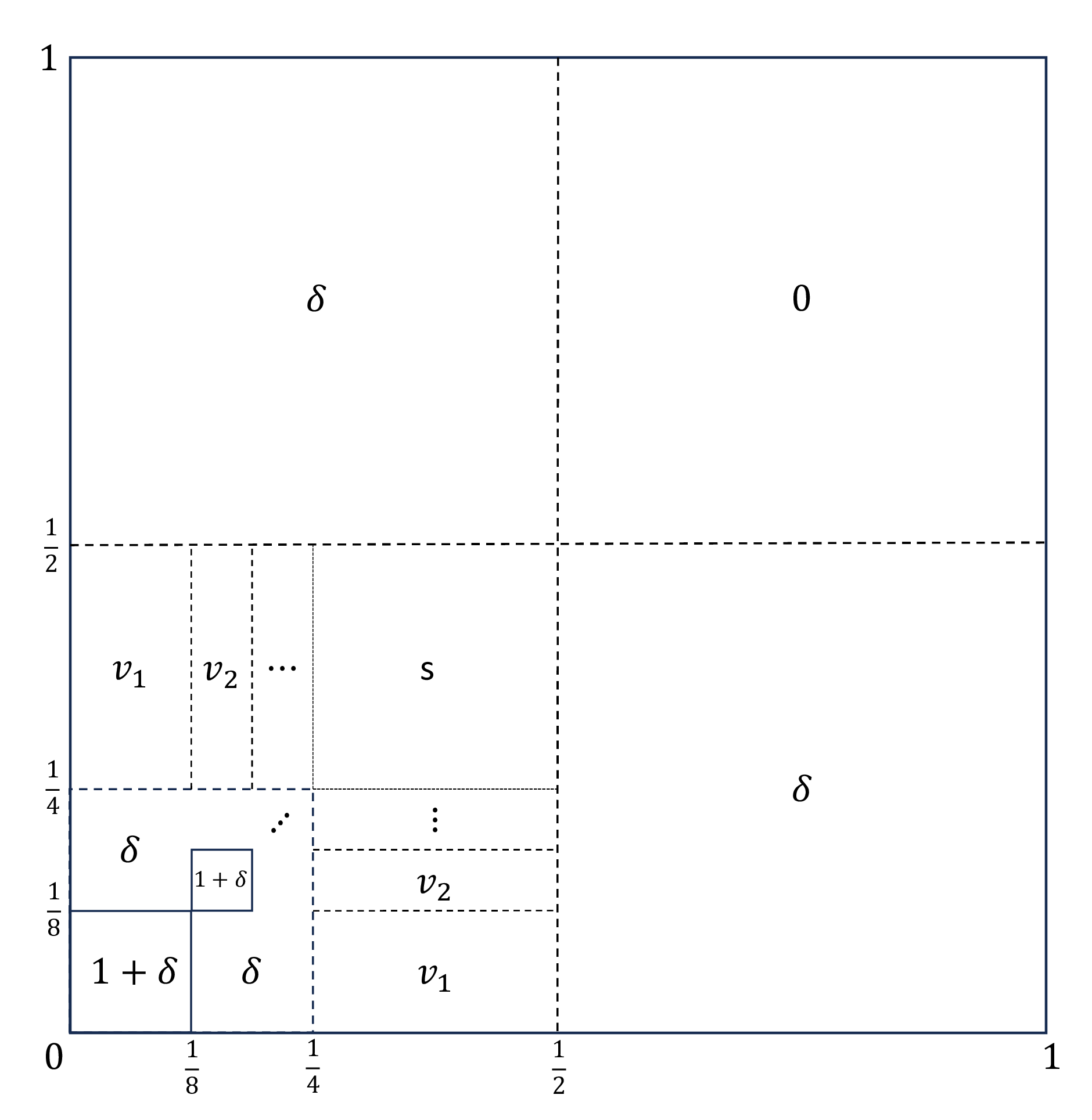}
	\caption{Function $p$ defined on $(0,1)^2$}
\label{figure1}
\end{wrapfigure}
In this example the periodicity cell is $Y=[0,1)$ and the soft component is $Y_\soft =(0,\frac{1}{2})$. Choosing $a(\cdot)$ as  in the previous example we have $\tilde{a}=\frac{1}{2}$. Next we define the function $p(y,\xi)$, see Figure \ref{figure1} for a graphical representation.
Consider a sequence of points $b_0:=0$, $b_j:=\frac{1}{4}\sum_{i=1}^j (1/2)^i$, $j\in \mathbf{N}$, and associated intervals $J_j=(b_{j-1},b_j)$, $j\in \mathbf{N}$. Note that the union of these intervals fills $(0,\frac{1}{4})$ and  $|J_j|=\frac{1}{2^{j+2}}$.
For positive $\delta $ and $s$, which will be specified later,
we define a function $p$ in the following way (see Figure \ref{figure1}):
\begin{itemize}
\item $p(y,\xi)=1+\delta$ on the set $S:=\cup_{j=1}^{\infty} J_j \times J_j$;
\item $p(y,\xi)=\delta$ on the set $(0,\frac{1}{2})\times(\frac{1}{2},1) \cup (\frac{1}{2},1) \times  (0,\frac{1}{2}) \cup \left((0,\frac{1}{4})^2 \backslash S\right)$; 	
\item $p(y,\xi)=s$ on   $(\frac14,\frac12)^2$;
\item $p(y,\xi)=v_j$ on $J_j \times (\frac{1}{4},\frac{1}{2})$, $j \in \mathbf{N}$, where $v_j$ is a constant chosen in such a way that $\int_0^{1} p(y,\xi) \,d\xi=1$ for each $y \in J_j$. It is easy to see that $v_j$ should satisfy the relation
$$ (1+\delta)|J_j|+\frac{\delta}{2}+\delta(\frac{1}{4}-|J_j|)+\frac{v_j}{4}=1. $$
We assume $\delta$ to be sufficiently small so that $v_j>0$  for all $j \in \N$.
\item Finally, we extend $p$ by symmetry, $p(y,\xi)=p(\xi,y)$, onto   $(\frac{1}{4},\frac{1}{2})
\times (0, \frac14)$.
\end{itemize} 	
By construction  $\int_0^1p(y,\xi)\, d\xi=1$ for all $y \in (0,\frac{1}{4})$, and $\int_0^1 p(y,\xi)\, d\xi=\sum_{j\in \N} v_j |J_j|+\frac{s}{4}  + \frac{\delta}{2} = \frac{7}{8} - \frac{\delta}{4}+ \frac{s}{4} =:\nu$   for all $y \in (\frac{1}{4},\frac{1}{2})$.   We choose   $\delta$ and  $s$   so that $\nu>1$.
Appealing to the  representation \eqref{operator2},\eqref{andrey1}  for $\AA_\soft^\# $, we have  that $m(y)=1$ for $y \in (0,\frac{1}{4})$,  $m(y)=\nu$ for $y \in (\frac{1}{4},\frac{1}{2})$, and $K(y,\xi)=p(y,\xi)\mathbf{1}_{Y_\soft^{\#}}(y)\mathbf{1}_{Y_\soft^{\#}}(l)$. Then $\Sp_{\rm ess}(\AA_\soft^\# )=\{1,\nu\}$.

 We next  show that the operator $\AA_\soft^\# $ has infinitely many points of discrete spectrum.
 To this end it is sufficient to find a sequence  $u_j \in L^2(Y_\soft )$, $j \in \N$,  such that
\begin{enumerate}
\item the elements of the sequence are mutually orthogonal;
\item  for any $n \in\mathbf{N}$ and $(\alpha_j)_{j=1}^n\subset \mathbf{R}^n\backslash \{0\}$ we have $a_\soft^\#(\sum_{j=1}^n \alpha_j u_j, \sum_{j=1}^n \alpha_ju_j)<\|\sum_{j=1}^n \alpha_j u_j \|^2_{L^2(\RR^d)}$.
\end{enumerate}

For each $j \in \mathbf{N}$ we set $u_j=\mathbf{1}_{J_j}(y)$. Obviously, the sequence $(u_j)_{j\in \mathbf{N}}$ satisfies a. To check b. we use \eqref{operator2} and the representation:
$$K(y,\xi)|_{(0,\frac{1}{4})\times (0,\frac{1}{4})  }= \sum_{j=1}^{\infty}\mathbf{1}_{J_j\times J_j} (y,\xi)+\delta , $$
 from which we deduce that, for any  $(\alpha_j)_{j=1}^n\subset \mathbf{R}^n\backslash \{0\}$,
\begin{eqnarray*}
 a_\soft^\#(\sum_{j=1}^n \alpha_ju_j, \sum_{j=1}^n \alpha_j u_j)&=&  \|\sum_{j=1}^n \alpha_j u_j \|^2_{L^2(\RR^d)}-\sum_{j=1}^n \alpha_j^2 |J_j|^2-\delta (\sum_{j=1}^n \alpha_j|J_j|)^2<  \|\sum_{j=1}^n \alpha_j u_j \|^2_{L^2(\RR^d)}.
 \end{eqnarray*}
This implies that there exists an infinite increasing sequence $(\mu_j')_{j \in \N}$ such that $\mu_j'<1$ for all $j\in\N $, $\lim\limits_{j\to\infty}\mu_j'=1$,
and
$$
\{\mu_1',\mu_2',\dots\}\cup\{1,\nu\} \subseteq \Sp  (\AA_\soft^\# ).
$$

\begin{remark}
In the above example one can modify $p$ on the square $(\frac{1}{4},\frac{1}{2})^2$ so that    $\Sp_{\rm ess}(\AA_\soft^\# )$  will contain  an interval.
\end{remark} 	

	\subsubsection{Example of $\AA_\soft^\# $  that has   discrete spectrum above and below its essential spectrum}

\bigskip
For sufficiently small  $\delta>0$ and $\varkappa>0$ we define:
\begin{itemize}
	\item $Y_\soft =(\delta,1-\delta)$;
	\item \begin{equation}\label{68a}
			a(\xi)=\,\left\{\begin{array}{ll}
			\frac{1}{16}\varkappa^{-1}, & \hbox{if } \xi\in  -\frac34 + I_\varkappa
			\cup -\frac14 + I_\varkappa
			\cup \frac14 + I_\varkappa
			\cup \frac34 + I_\varkappa,  \\
			0, & \hbox{otherwise},
		\end{array}  \right.
	\end{equation}
	where $I_\varkappa: = (-\varkappa, \varkappa)$;
	\item $p(y,\xi)$ via \eqref{60} with $w(y,\xi)=1$.
\end{itemize} 	
Observe that
 $$
\tilde{a}(y)=\,\left\{\begin{array}{ll}
	\frac18\varkappa^{-1}, & \hbox{if } y\in \frac14 + I_\varkappa
	\cup \frac34 + I_\varkappa,  \\
	0, & Y\setminus (\frac14 + I_\varkappa
	\cup \frac34 + I_\varkappa).
\end{array}  \right.
$$
Recall that $\tilde{a}$ is even and periodic.
Consequently $m(y)=1$ (cf. \eqref{andrey1}). Thus we have
$$
\AA_\soft^\#  z= z-\int_{Y_\soft }\tilde{a}(\xi-\cdot) z(\xi)\,d\xi \quad  z \in L^2_\#(Y_\soft).
$$
The essential spectrum of this operator consists of just one point: $\Sp _{\rm ess}(\AA_\soft^\# )=\{1\}$.

Substituting $z_1(y)=(1- 2\delta)^{-\frac12} \mathbf{1}_{Y^\#_\soft}(y)$ into the quadratic form associated with  $\AA_\soft^\#$ we have
$$
a_\soft^\#(z_1,z_1) =1-\int\limits_{Y_\soft }\int\limits_{Y_\soft } \tilde{a}(\xi-y) z_1(y)z_1(\xi)\, dyd\xi<1.
$$
Since $\|z_1\|_{L^2(Y^\#_\soft)}=1$, this implies that there exists a discrete eigenvalue $\lambda'$ such that $\lambda'<1$ (recall \eqref{representation}).

Now we choose the following test function:
$$z_2(y)= \left\{\begin{array}{lll}
	\varkappa^{-\frac12}, & \hbox{if }|y-\frac18|<\frac14\varkappa,  \\
	-\varkappa^{-\frac12}, & \hbox{if }|y-\frac38|<\frac14\varkappa,  \\
	0, & \hbox{otherwise}
\end{array}
\right.
\quad \hbox{for }y\in Y,
$$
 extended by periodicity. By construction,  $\|z_2\|_{L^2(Y_\soft)} = 1$.
Denoting $I_\varkappa^-= \big(\frac18-\frac14 \varkappa,\frac18+\frac14 \varkappa\big)$ and
$I_\varkappa^+=\big(\frac38-\frac 14 \varkappa,\frac38+\frac14 \varkappa\big)$,
we notice that
$$
z_2(y) z_2(\xi)=\left\{
\begin{array}{ll}
	\varkappa^{-1}&\hbox{if }(y,\xi)\in (I_\varkappa^+\times  I_\varkappa^+)\cup ( I_\varkappa^-\times  I_\varkappa^-),\\
	-\varkappa^{-1}&\hbox{if }(y,\xi)\in (I_\varkappa^-\times  I_\varkappa^+)\cup ( I_\varkappa^+\times  I_\varkappa^-),\\
	0& \hbox{otherwise}.
\end{array}
\right.
$$
Since $\tilde{a}(\xi-y)=0$ for $(y,\xi)\in  (I_\varkappa^+\times  I_\varkappa^+)\cup ( I_\varkappa^-\times  I_\varkappa^-)$ for $\chi$ small enough,
and $\tilde{a}(\xi-y)=\frac18\varkappa^{-1}$ for $(y,\xi)\in ( I_\varkappa^+\times  I_\varkappa^-)\cup ( I_\varkappa^-\times  I_\varkappa^+)$,
we obtain
\begin{eqnarray*}
(\AA_\soft^\#  z_2,z_2)_{L^2 (Y^\#_\soft)}&=&1-\int\limits_{Y_\soft }\int\limits_{Y_\soft } \tilde{a}(\xi-y) z_2(y)z_2(\xi)\, dyd\xi \\
&=&
1+\frac{1}{16}>1.
\end{eqnarray*}
Therefore, there exists a discrete eigenvalue $\lambda''$ such that $\lambda''>1$ (recall \eqref{representation22}).

\bigskip
The convolution kernel defined in \eqref{68a} does not satisfy \eqref{a-2}, therefore we must
 modify $a(\xi)$ to make it comply with this assumption.
To this end we choose a non-negative even $C_0^\infty(\mathbf R)$ function $\varphi$ such that
$\varphi(0)>0$, and define $a_\gamma(\xi)=a(\xi)+\gamma\varphi(\xi)$.   Clearly, for any
$\gamma>0$ the kernel  $a_\gamma$ satisfies condition \eqref{a-2}.
By simple perturbation theory arguments, for sufficiently small $\gamma>0$,
the operator $\AA_\soft^\#$ with the convolution kernel $a_\gamma$ still has a non-trivial discrete spectrum above and below
the essential spectrum.

\section{Analysis of the limiting spectrum } \label{s:6}

 In this section we analyse the structure of the limiting spectrum.  In particular, it is shown that the limit spectrum is in general strictly larger then the spectrum of the limit operator.  In the case of a bounded domain $S$, the limit spectrum depends on the geometry of the soft component in the vicinity of the domain boundary. In general, the boundary spectrum behaves ``erratically'' as $\e\to 0$. Yet, it may be possible to describe it in special cases and for specific subsequences of $\e$.  We consider an example when $S$ is a rectangular box with integer dimensions. Namely, in Theorem \ref{th2.7} we assume that
\begin{equation}\label{69}
	S=\Pi:=(0,l_1)\times \dots \times (0,l_d),\,\,l_1,\dots,l_d \in \N.
\end{equation}
and $\e = 1/N$, $N\in \N$. For future reference we denote $l:=\min\{l_1,\dots,l_d\}.$ We work with this specific choice of $S$ in Subsections \ref{s5.2} and \ref{s5.3}. 

We begin with the proof of Theorem \ref{th2.6}, which is formulated for a generic set $S$ --- i.e. either a bounded open Lipschitz set or the whole space.

\subsection{Proof of Theorem \ref{th2.6}}

Recall the definition of the operator $\AA_{\e,\rm soft}$, see \eqref{21}. We will need the following simple result.
\begin{proposition} \label{pr6.1}
	\begin{equation}\label{83}
		\Sp(\AA_\soft^\#) \subset     \lim_{\e\to 0}\Sp (\AA_{\e,\rm soft}).
	\end{equation}
\end{proposition}
\begin{proof}
	For $z_\e, f_\e  \in L^2(\e Y^\#_\soft \cap S)$, $\l<0$, the equation
	\begin{equation*}
		a_{\e,\soft}(z_\e, v) - \l \int_{R^d} z_\e v = \l \int_{\R^d} f_\e v, \quad \forall v\in L^2(\e Y^\#_\soft \cap S),
	\end{equation*}
	is a particular case of  equation \eqref{starting}. Choosing a sequence $f_\e \wtto (\stto) f(y) \mathbf 1_{S}(x)$ for some $f\in L^2(Y_\soft)$, we have that $z_\e\wtto z\in L^2_\#(S\times Y_\soft)$ up to a subsequence. Passing to the limit as in \eqref{rok1}, \eqref{recal1} with the   test functions of the form $\varphi(x)b(x/\e)$,  $\varphi \in L^2(\mathbf{R}^d)$, $b \in C_\#(Y_\soft)$, and then setting $\varphi  = \mathbf 1_{S} $ we conclude that $u_0 =0$ and $z$ is independent of $x$. In this case  \eqref{recal1} reads
	\begin{multline*}
		\int_{\R^n}\int_{Y}  a(\xi)p(y, y+\xi)(z(y+\xi)-z(y))(b(y+\xi)-b(y)) \,dy\,d\xi\,
		\\
		-\lambda \int_{Y} z(y) b(y) dy= \int_{Y} f(y) b(y)dy.
	\end{multline*}
	This establishes a weak (and, hence, strong, cf. Proposition \ref{prop510}) two-scale resolvent convergence of $\AA_{\e,\soft}$ to  $\AA_{\soft}^\#$. In turn, this implies spectral inclusion \eqref{83} as per Remark \ref{remandrei3}.
\end{proof}
	
 Below we only provide an argument for the case of a bounded domain  $S$. Although the argument can be adapted also for the case $S= \R^d$, we do not present it here, since  Theorem \ref{th2.8} provides a stronger result.

Let $\lambda_\e \in \Sp (\AA_\e )$ be a converging sequence, $\l_\e\to \l$, with $\l \notin \lim_{\e\to 0} \Sp(\AA_{\e,\soft})$. Then for sufficiently small $\e$ we have that $\dist(\l_\e, \Sp(\AA_{\e,\soft})) \geq C>0$. (This will be important later on, when we will resort to Proposition \ref{pr6.1}.) There exist $u_\e , f_\e  \in L^2(S)$, such that $\|u_\e \|_{L^2(\RR^d)}=1$, $\|f_\e\|_{L^2(\RR^d)}:=\delta_\e\to 0$ as $\e\to 0$, and
	\begin{eqnarray} \nonumber
		& &\int_{\R^d} \int_{\R^d} a(\xi)\Lambda_\e (x, x+\e\xi)\left(\frac{u_\e (x+\e \xi)-u_\e (x)}\e  \right) \left(\frac{v(x+\e \xi)-v(x)}\e  \right)\,d\xi\,dx=\\ \label{estimate60} & & 
  \lambda_\e  \int_{\R^d} u_\e  (x)v(x) \,dx+\int_{\R^d} f_\e  (x) v(x) \,dx , \quad \forall v \in
  L^2(S).
	\end{eqnarray}

	By Corollary \ref{corcan3}, $u_\e$ admits the decomposition
	$$ u_\e  = \widetilde{u}_\e +z_\e = \bar{u}_\e +\e \hat{u}+z_\e,$$
	 with
	$$\|\bar{u}_\e \|_{H^1(\RR^d)} \leq C, \quad \|\hat{u} \|_{L^2(\RR^d)} \leq C, \quad \|z_\e \|_{L^2(\RR^d)} \leq C, \quad z_\e =0 \textrm{ on }  Y_\stiff^{\#} \cup S^c.   $$
	 Notice that $\supp \bar{u}_\e \subset S^{\e k}$ (recall that $S^{\e k}$ denoted $\e k$ neighbourhood of the set $S$).
	 Then up to  a subsequence we have that
$$
\bar{u}_\e  \xrightharpoonup{H^1} {u}_0, \quad \nabla  \bar{u}_\e  \xrightharpoonup{2} \nabla {u}_0(x)+\nabla_y \bar{u}_1(x,y),\quad  \hat{u}_\e \xrightharpoonup{2} \hat{u}_0(x,y), \quad z_\e  \xrightharpoonup{2} z(x,y).
$$
	Here $u_0\in H^1_0(S)$, $\bar{u}_1 \in L^2(\R^d;H^1_\#(Y))$, $\hat{u}_1, z, \in L^2_\#(\R^d\times Y)$. Moreover, $\bar{u}_1(x,\cdot)=\hat{u}_0(x,\cdot)=z(x,\cdot)=0$, for $x \notin S$, and  $z(x,y)=0$ for $y \in  Y_\stiff^{\#}$.

	First we argue that $\|u_0\|_{L^2(S)} > 0$. Assume the opposite. Then one necessarily has that  $\widetilde{u}_\e  \to 0$ strongly in $L^2(\R^d)$.
	Substituting in \eqref{estimate60}  test functions of the form
	$ v_\e (x)=v_1(x) v_2(\frac{x}\e )$, where $v_1\in L^2(S)$, $v_2\in L^2_\#(Y_\soft)$ yields
	\begin{eqnarray} \nonumber
		& &\int_{\R^d} \int_{\R^d} a(\xi)p(\frac{x}\e , \frac{x+\e\xi}\e )\left(z_\e (x+\e \xi)-z_\e (x) \right) \left(v_\e (x+\e \xi)-v_\e (x) \right)\,d\xi\,dx \\ \nonumber & &+ \int_{\R^d} \int_{\R^d} a(\xi)p(\frac{x}\e , \frac{x+\e\xi}\e )\left(\widetilde{u}_\e (x+\e \xi)-\widetilde{u}_\e (x) \right) \left(v_\e (x+\e \xi)-v_\e (x) \right)\,d\xi\,dx-\lambda_\e \int_{\R^n} z_\e (x)v_\e (x) \,dx \\ \label{estimate67}& &=\lambda_\e  \int_{\R^d} \widetilde{u}_\e  (x)v_\e (x) dx+\int_{\R^d} f_\e (x) v_\e (x)\,dx.
	\end{eqnarray}
	Since the integrals containing $\widetilde{u}_\e$ define a bounded linear functional acting on the space of test functions $ L^2(S)$,  the above identity can be rewritten in the form
	\begin{equation*}
	a_{\e,\soft}(z_\e, v_\e) - \lambda_\e \int_{\R^n} z_\e (x)v_\e (x) \,dx = \int_{\R^d} (f_\e (x) + \tilde f_\e(x)) v_\e (x)\,dx,
	\end{equation*}
	for some $\tilde f_\e$ satisfying
	\begin{equation*}
	\|\tilde f_\e\|_{L^2(S)} \leq C \|\widetilde u_\e\|_{L^2(R^d)} \to 0 \mbox{ as } \e \to 0.
	\end{equation*}
Now, since by assumption $\lambda_\e$ are bounded away from the spectrum of $\AA_{\e,\soft}$ for sufficiently small $\e$, we immediately conclude that
	$z_\e  \to 0 \textrm{ strongly in } L^2(S)$, and hence
 $u_\e =(\widetilde{u}_\e +z_\e )\to 0$ in $L^2$. This relation contradicts the fact that $\|u_\e \|_{L^2(\RR^d)}=1$, hence  $u_0 \neq 0$.
	
	Passing to the limit in \eqref{estimate67} via   Corollary \ref{corandrei1}, in the same way as in Section \ref{two-scalecon} (cf. \eqref{recal1}), we arrive at
	\begin{eqnarray*} 
		& &\int_{\R^d}v_1(x) \int_{Y} \int_{\R^d} a(\xi)p(y, y+\xi)\left(z(x,y+ \xi)-z(x,y) \right) \left(v_2(y+\xi)-v_2(y)\right)\,d\xi\,dy \,dx \\ & & -\lambda \int_{\R^d} v_1(x)\int_{Y} z(x,y)v_2(y) \,dy\,dx =   \lambda\int_{\R^d}u_0(x) v_1(x)\int_{Y} v_2(y)\, dy \, dx.
	\end{eqnarray*}
	Note that $v_1$ is arbitrary, $x$ plays role of a parameter in the above equation. Recall also that $z$ is periodic in $y$.
	 Since  by Proposition \ref{pr6.1} the resolvent of $\AA_\soft^\#$ is well defined, it follows that
	\begin{equation} \label{estimate80}
		z(x,\cdot)=u_0(x)\lambda(\AA_\soft^\# -\lambda I)^{-1} \mathbf{1}_{Y_\soft }, \textrm{ for a.e. } x \in S.
	\end{equation}
	
	Passing to the limit in \eqref{estimate60} exactly as in the proof of \eqref{ruc1}, cf. \eqref{56} and the subsequent argument in Section \ref{two-scalecon}, we arrive at
	\begin{equation*}
		\int_{S}  A^{\rm hom} \nabla u_0(x) \cdot \nabla v(x)  dx - \lambda\int_{S} \left(u_0(x)+\int_{Y} z(x,y)\, dy\right)v(x)\, dx=0.
	\end{equation*}
	Then combining \eqref{ruc22} and \eqref{estimate80}  we obtain
	$$ \int_{S} A^{\rm hom} \nabla u_0(x) \cdot \nabla v(x)\,dx -\beta(\lambda) \int_{S} u_0v=0, $$
	and since $u_0 \neq 0$, we conclude that
 $$
 \beta(\lambda) \in \Sp (\AA_\hom ).
 $$
\hfill $\square$

\begin{remark}
	The claim of  Theorem \ref{th2.6} remains valid if we replace $\AA_\e$ with $\AA_{1/N}$ and $\AA_{\e,\soft}$ with $\AA_{{1/N},\soft}$.
\end{remark}

\subsection{Characterization of $\text{H-}\lim_{N \to \infty} \Sp (\AA_{1/N,\rm soft})$ for the case $S=\Pi$} \label{s5.2}

As we already noted at the beginning of this section, the boundary spectrum, which  comes essentially from the geometry of the soft component in the boundary layer, is difficult, if not impossible, to characterise in general. Moreover,  one cannot get rid of the boundary spectrum by replacing the soft component near the boundary with the stiff one (as one could do in PDEs setting with dispersed/disconnected soft inclusions), even if the convolution kernel has finite support. In fact, this strategy will work only if the diameter of the support is smaller that the minimal distance between soft inclusions. 

Yet, it should be possible to characterise the limit boundary spectrum in case for specific geometries of $S$.  
In what follows we show that in the case when $S =\Pi$, cf. \eqref{69}, the Hausdorff limit $\text{H-}\lim_{N\to \infty} \Sp (\mathcal{A}_{1/N})$ exists and  provide a detailed description of this limit. To this end we first characterize the limit of $\Sp (\AA_{1/N,\rm soft})$ and then show that $\text{H-}\lim_{N \to \infty} \Sp (\AA_{1/N,\rm soft}) \subset \text{H-}\lim_{N\to \infty} \Sp (\mathcal{A}_{1/N})$. A particular feature of this setting is that the geometry of the microstructure in the boundary layer remains congruent to itself (under appropriate scaling) for each $N\in \N$, thus rendering the boundary layer spectrum to be stable as $N\to \infty$.

Our technique can be easily adapted to a more general case when $S$ is a polytope, provided that there exists a subsequence of $\e$ such that the geometry of the soft component in  the boundary layer in the vicinity of each vertex of the polytope is self-similar for all elements of the subsequence of $\e$.   In this case, similarly to \eqref{57} below, the limiting spectrum of the soft component can be characterised as the union of spectra of the soft component in each vertex of the polytope.

 The main result of this subsection is formulated in  Theorem \ref{th5.5} below. We begin by introducing some notation and preparing auxiliary statements.

We denote by $\Pi_N: = N\Pi$ the scaled box and denote its vertices by $v_i$, $i=1,2,\ldots,2^d$, starting with $v_1 = (0,\dots,0)$. The order of enumeration is not important. For brevity we do not reflect the dependence of coordinates of $v_i$ on $N$.  To each vertex $v_i$ we assign an orthant (or hyperoctant) of the space
\begin{equation}\label{utoest5}
	\mathbf{R}^d_{v_i}=\{ (x_1,\dots,x_d): \forall j=1,\dots,d, \ \pm x_j \geq 0 \}
\end{equation} 	
so that after an appropriate translation, i.e. $\mathbf{R}^d_{v_i} + v_i$, it coincides with $\Pi_N$ in the vicinity of $v_i$.
We consider the operators $\mathcal{A}^N_\soft: L^2(\Pi_N\cap Y^\#_\soft )\to L^2(\Pi_N\cap Y^\#_\soft )$ and  $\AA_\soft^{v_i}: L^2(\mathbf{R}^d_{v_i}\cap Y_\soft^{\#}) \to L^2(\mathbf{R}^d_{v_i}\cap Y_\soft^{\#})$, $i=1,2,\ldots,2^d$:
\begin{equation*}
\begin{gathered}
	\AA_\soft^N z (y) :=m(y)z(y)-\mathbf{1}_{\Pi_N\cap Y^\#_\soft}(y)\int_{\R^d} a(\xi) p(y,y+\xi)z(y+\xi) \, d\xi,
	\\
	\AA_\soft ^{v_i} z (y) =m(y)z(y)-\mathbf{1}_{\mathbf{R}^d_{v_i}\cap Y^\#_\soft}(y)\int_{\R^d} a(\xi) p(y,y+\xi)z(y+\xi) \, d\xi,
	\\
	m(y):=2\int_{\R^d} a(\xi) p(y,y+\xi)\,d\xi,
\end{gathered}
\end{equation*}
alongside with their ``truncated'' counterparts
\begin{equation*} 
	\begin{gathered}
		\AA_{\soft,L}^N z (y) :=m_L(y)z(y)-\mathbf{1}_{\Pi_N\cap Y^\#_\soft}(y)\int_{\R^d} a_L(\xi) p(y,y+\xi)z(y+\xi) \, d\xi,
\\
	\AA_{\soft,L} ^{v_i} z (y) =m_L(y)z(y)-\mathbf{1}_{\mathbf{R}^d_{v_i}\cap Y^\#_\soft}(y)\int_{\R^d} a_L(\xi) p(y,y+\xi)z(y+\xi) \, d\xi,
	\end{gathered}
	\end{equation*}
where
\begin{equation}\label{85}
	\begin{gathered}
	a_L(\xi):=\mathbf{1}_{\square^{L}}(\xi)\,a(\xi),
\\
m_L(y):=2\int_{\R^d} a_L(\xi) p(y,y+\xi)\,d\xi.
\end{gathered}
\end{equation}

By Theorem \ref{thmfolland} we have
\begin{equation} \label{defeta}
	\begin{gathered}
		\|m_L-m\|_{L^\infty} \to 0,\qquad \rho(L):=\max_i \|\AA_\soft ^{v_i}-\mathcal{A}_{\soft,L}^{v_i}\|_{L^2 \to L^2} \to 0,
	\\
		\tilde{\rho}(L):=\sup_{N \in \mathbf{N}} \|\mathcal{A}^N_\soft-\mathcal{A}_{\rm soft,L}^N\|_{L^2 \to L^2} \to 0
		\ \ \hbox{as }L \to \infty.
	\end{gathered}
\end{equation}


Note that the operator $\AA_\soft^N$ is unitary equivalent to   $\AA_{1/N,\soft}$ via rescaling. In particular, we have
$$ \Sp (\mathcal{A}^N_\soft)=\Sp  (\mathcal{A}_{1/N,\rm soft}).$$
Thus, in order to characterise $\text{H-}\lim_{N \to \infty} \Sp (\AA_{1/N,\rm soft})$
it is sufficient to analyse $\Sp (\mathcal{A}^N_\soft)$.

First we need a couple of technical results.
\begin{proposition} \label{cetprop1}
Let $u,f \in L^2(\mathbf{R}^d)$,  be different from zero. 	Then for each $L>0$ there exists $\zeta \in L \mathbf{Z}^d$ such that
	\begin{equation}
		\begin{aligned}
			&{\|u\|_{L^2(\square^{4L}_\zeta)}} \leq  \sqrt{4^d+  2^d}   {\| u\|_{L^2(\square^{L}_\zeta)}},
			\\
		&	\frac{\|  f\|_{L^2(\square^{2L}_\zeta)}}{\|f\|_{L^2(\RR^d)}}\leq \sqrt{4^d+  2^d}  \frac{\| u \|_{L^2(\square^L_\zeta)}}{\|u\|_{L^2(\RR^d)}},  \label{prd1} 	
		\end{aligned}
		 \end{equation}
		 \begin{equation*}
		 	\|  u \|_{L^2(\square^L_\zeta)} \neq 0.
		 \end{equation*}

\end{proposition} 	
\begin{proof}
	Clearly, one has
	\begin{equation*} 
	\sum_{\zeta\in L\mathbf{Z}^d} \|u\|_{L^2(\square^{4L}_\zeta)}^2=4^d\|u\|^2_{L^2(\RR^d)}, \quad		
		\sum_{\zeta \in L\mathbf{Z}^d} \| f\|_{L^2(\square^{2L}_\zeta)}^2=2^d\|f\|^2_{L^2(\RR^d)}.
	\end{equation*} 	
Denote	$\mathcal{N}=\{\zeta\in L\mathbf{Z}^d:   u=0 \mbox{ in } \square^{L}_\zeta\}$.
	Assuming  that for all $\zeta\in L\mathbf{Z}^d \setminus \mathcal{N}$ at least one of the inequalities in \eqref{prd1} does not hold, we arrive at a contradiction:
	\begin{eqnarray*}
		& & 4^d+  2^d =	\frac{1}{\|u\|^2_{L^2(\RR^d)}}\sum_{\zeta\in L\mathbf{Z}^d \setminus \mathcal{N} } \| u\|_{L^2(\square^{4L}_\zeta)}^2+  \frac{1}{\|f\|^2_{L^2(\RR^d)}}	\sum_{\zeta\in L\mathbf{Z}^d \setminus \mathcal{N}} \|  f\|^2_{L^2(\square^{2L}_\zeta)}>\\ & &  \hspace{+10ex}  (4^d+  2^d )	\frac{1}{\|u\|^2_{L^2(\RR^d)}}	\sum_{\zeta\in L\mathbf{Z}^d \setminus \mathcal{N} } \|  u\|^2_{L^2(\square^{L}_\zeta)} =4^d+  2^d.
	\end{eqnarray*} 	
\end{proof}

\begin{proposition}\label{p5.4}
 Let $a_L$ be given by \eqref{85},  $|K(y,\xi)|\leq C$, $(y,\xi)\in \R^{2d}$,  and $\eta_L$ be a cut-off function such that $\eta_L \in C_0^{\infty}(\square^{2L}_{\zeta})$, $\eta_L|_{\square^L_{\zeta}}=1$, $\|\eta_L\|_{L^{\infty}(\RR^d)}\leq 1$, $\|\nabla \eta_L\|_{L^{\infty}(\RR^d)} \leq C/L$.
Then for the operator $\mathcal{K}: L^2(\R^d) \to L^2(\R^d)$
$$
\mathcal{K} u(y) = \int_{\R^d} a_L (\xi-y) K(y,\xi)u(\xi) d\xi
$$
we have
\begin{equation*}
\|\mathcal{K} (\eta_L u)  - \eta_L    \mathcal{K} u \|_{L^2(\R^d)} \leq \frac{C}{L}\| u \|_{L^2(\square^{4L}_{\zeta})},	
\end{equation*}
where the constant $C$ is independent of $L$.
\end{proposition}
\begin{proof}
	\begin{equation*}
	\mathcal{K} (\eta_L u)  - \eta_L    \mathcal{K} u
 =\int_{\R^d} a_L(\xi) K(y,y+\xi)(\eta_L(y+\xi)-\eta_L(y))u(y+\xi) \, d\xi.
\end{equation*}
Since $\supp \eta_L\subset \square^{2L}_{\zeta }$ and $\supp a_L \subset \square^{L}$
the expression
$a_L (\xi)(\eta_L (y+\xi)-\eta_L (y))$ vanishes for  $y \in \R^d \setminus\square^{3L}_{\zeta} $  and every $\xi \in \mathbf{R}^d$. Therefore,
\begin{equation*}
	a_L(\xi)|K(y,y+\xi)||\eta_L (y+\xi)-\eta_L(y)| \leq a_L(\xi)\frac{C|\xi|}{L}\mathbf{1}_{\square^{3L}_{\zeta}}(y).
\end{equation*} 	
Then using the Cauchy–Schwarz inequality we have
\begin{multline*}
	\int_{\R^d}	\left|\int_{\R^d} a_L(\xi) K(y,y+\xi)(\eta_L(y)-\eta_L(y+\xi))u(y+\xi) \, d\xi\right|^2 dy
	\\
	\le \frac{C}{L^2} \int_{\R^d} a_L(\xi) |\xi|^2 d\xi \iint_{\R^{2d}}  a_L(\xi) \,\mathbf{1}_{\square^{3L}_{\zeta}}(y) |u(y+\xi)|^2 d\xi dy
	\le \frac{C}{L^2} \int_{\R^d} a(\xi) |\xi|^2 d\xi  \, \|u\|^2_{L^2(\square^{4L}_{\zeta})}.
\end{multline*}
\end{proof}

Now we are in position to prove the following result.
\begin{theorem} \label{th5.5}
	\begin{equation} \label{tonci1}  \text{H-}\lim_{N \to \infty} \Sp (\AA_{1/N,\soft})=\bigcup_{j=1,\dots 2^d} \mathrm{Sp} (\AA_\soft^{v_j}). \end{equation}
\end{theorem}
\begin{proof}

	First we prove the inclusion $\supseteq$ in \eqref{tonci1}.
	We only show that $\Sp (\AA_\soft ^{v_1})\subset \text{H-}\lim_{N \to \infty} \Sp (\AA_{1/N,\soft})$,  for other $i=2,3,\dots, 2^d,$ the proof is completely analogous.
	For any $\lambda \in \Sp  (\AA_\soft ^{v_1})$ there exists
	a sequence $(u_k)_k \subset L^2(\mathbf{R}^d_{v_1} \cap Y_\soft^{\#})$ such that
	\begin{equation} \label{cetest99}
		\|u_k\|_{L^2(\RR^d)}=1, \quad \AA_\soft ^{v_1} u_k-\lambda u_k=f_k \in L^2(\RR^d_{v_1} \cap Y_\soft^{\#}), \  \|f_k\|_{L^2(\RR^d)}=:\delta_k \to 0. 	
	\end{equation}
	From \eqref{defeta} and  \eqref{cetest99} we also have that
	\begin{equation*} 
		\mathcal{A}_{\rm soft,L}^{v_1} u_k-\lambda u_k=f_{k,L} \in L^2(\RR^d_{v_1} \cap Y_\soft^{\#}), \quad \|f_{k,L}\|_{L^2(\RR^d)} \leq \delta_k +\rho(L)=:\delta_{k,L}.
	\end{equation*}
	By using Proposition \ref{cetprop1}   we find that   for each $k,L$  there exists $\zeta_{k,L} \in L\mathbf{Z}^d$ such that
	\begin{eqnarray}\label{prd6}
		& &\|  u_k\|_{L^2(\square^{4L}_{\zeta_{k,L}})} \leq \sqrt{4^d+  2^d}\, \|  u_k\|_{L^2(\square^{L}_{\zeta_{k,L}})},  \\ \label{utoest23}  & &   \| f_{k,L}\|_{L^2(\square^{2L}_{\zeta_{k,L}})} \leq \sqrt{4^d+  2^d}\, \delta_{k,L} \| u_k\|_{L^2(\square^{L}_{\zeta_{k,L}})}, \\ & & \label{prd8}  \| u_k\|_{L^2(\square^{L}_{\zeta_{k,L}})} \neq 0.
	\end{eqnarray}
		We consider a cut-off function
		\begin{equation}\label{104}
			\eta_L \in C_0^{\infty}(\square^{2L}_{\zeta_{k,L}}) \mbox{ such that } \eta_L|_{\square^L_{\zeta_{k,L}}}=1,\, |\eta_L|\leq 1,\, |\nabla \eta_L|\leq C/L,
		\end{equation}
		for some $C>0$,   and define
	\begin{equation} \label{utoest20}
		u_{k,L} =\frac{\eta_L  u_k}{\|u_k\|_{L^2(\square^L_{\zeta_{k,L} })}}.
	\end{equation}	
	From \eqref{prd6} we have
	\begin{equation} \label{utoest21}
		1 \leq \|u_{k,L}\|_{L^2(\mathbf{R}^d)} \leq \sqrt{4^d+  2^d}.
	\end{equation}
	By direct computation
	\begin{equation*}
		\mathcal{A}_{\rm soft,L}^{v_1} u_{k,L}-\lambda u_{k,L}
		 = \frac{\eta_L(y)f_{k,L} + \mathcal{A}_{\rm soft,L}^{v_1} (\eta_L u_k)  - \eta_L    \mathcal{A}_{\rm soft,L}^{v_1} u_k }{\|u_k\|_{L^2(\square^L_{\zeta_{k,L}})}}  .
	\end{equation*}
Applying Proposition \ref{p5.4} to the terms on the right-hand side involving 	$\mathcal{A}_{\rm soft,L}^{v_1}$, and taking into account \eqref{prd6} and \eqref{utoest23}
 we conclude that
	\begin{equation} \label{utoest24}
		\| \mathcal{A}_{\rm soft,L}^{v_1} u_{k,L}-\lambda u_{k,L}\|_{L^2(\RR^d)} \leq C \delta_{k,L} + \frac{C}{L}.
	\end{equation}

For each $L$ we choose sufficiently large $N = N(L)$, namely such that
\begin{equation}\label{104a}
	lN(L) >3L+ |\zeta_{k,L}|
\end{equation}
   (recall that $l$ denotes the length of the shortest edge of $\Pi$). Then, since the expression
$a_L (\xi) u_{k,L}(y+\xi)$ vanishes for  $y \in \R^d \setminus\square^{3L}_{\zeta_{k,L}} $  and every $\xi \in \mathbf{R}^d$,  we have that
\begin{equation*}
		(\AA_{\soft,L} ^{v_1}u_{k,L} -\AA_{\soft,L}^{N(L)}  u_{k,L})(y) = \mathbf{1}_{\mathbf{R}^d_{v_1}\backslash \Pi_{N(L)}} (y) \int_{\R^d}  a_L(\xi) p(y,y+\xi) u_{k,L}(y+\xi) \, d\xi \equiv 0 	.
\end{equation*}
Then from   \eqref{utoest24}, \eqref{defeta} and the latter we have
	\begin{equation} \label{ana1}
		 \lim_{k \to \infty} \limsup_{L\to \infty}   \|\AA_\soft^{N(L)} u_{k,L} -\lambda u_{k,L}\|_{L^2(\RR^d)}=0.   	
	\end{equation} 	
	Since
	\begin{equation*}
		\textrm{dist} (\lambda,\Sp (\AA_\soft^N) ) \leq \frac{\|\AA_\soft^N u_{k,L} -\lambda u_{k,L}\|_{L^2(\RR^d)}}{\|u_{k,L}\|_{L^2(\RR^d)}}, 	
	\end{equation*} 	
	taking into account \eqref{utoest21}, we conclude that
	\begin{equation*}
		\lim_{N \to \infty} \textrm{dist}  (\lambda,\Sp (\AA_\soft^N) )=0.
	\end{equation*}

	Next  we prove the inclusion $\subseteq$ in \eqref{tonci1}. The argument is analogous to the first part of the proof.
	Consider a sequence $\lambda_N \in \Sp (\AA_\soft^N)$, and assume that $\lambda_N \to \lambda$ as $N \to \infty$. Then there exists a  sequence
	$(u_N)_{N \in \mathbf{N}} \subset L^2(\Pi_N \cap Y_\soft^{\#})$ such that
	\begin{equation*}
		\|u_N\|_{L^2(\RR^d)}=1, \quad \AA_\soft^{N} u_N-\lambda u_N=f_N \in L^2(\Pi_N \cap Y_\soft^{\#}),\quad  \|f_N\|_{L^2(\RR^d)}=:\delta_N \to 0. 	
	\end{equation*}
Then from \eqref{defeta} for the operator with the truncated kernel (recall the notation in \eqref{defeta}) we have that
		\begin{equation*}
		\AA_{\rm soft,L}^{N} u_N-\lambda u_N=f_{N,L}, \quad \|f_{N,L}\|_{L^2(\RR^d)} \leq \delta_N +\tilde{\rho}(L)
	\end{equation*}

	By Proposition \ref{cetprop1} for every pair $L,N \in \mathbf{N}$    we choose $\zeta_{N,L} \in L\mathbf{Z}^d$ such that
	\begin{eqnarray}\label{prd10}
		& &\|  u_N\|_{L^2(\square^{4L}_{\zeta_{N,L}})} \leq \sqrt{4^d+  2^d}\, \|  u_N\|_{L^2(\square^{ L}_{\zeta_{N,L}})},
		 \\
		 & &   \| f_{N,L}\|_{L^2(\square^{2L}_{\zeta_{N,L}})} \leq \sqrt{4^d+  2^d}\, (\delta_N +\tilde{\rho}(L))\| u_N\|_{L^2(\square^{L}_{\zeta_{N,L}})}, \label{118b}
		 \\ & &
           \|  u_N\|_{L^2(\square^{L}_{\zeta_{N,L}})} \neq 0.\nonumber
	\end{eqnarray}
	
	We define
	\begin{equation*} 
		u_{N,L} := \frac{\eta_L u_N}{\|u_N\|_{L^2(\square^L_{\zeta_{N,L} })}}.
	\end{equation*}	
	 where the cut-off function  $\eta_L $ is as in \eqref{104} with $\zeta_{k,L}$ replaced by $\zeta_{N,L}$. Applying Proposition \ref{p5.4} while taking into account \eqref{prd10} and \eqref{118b} we have
		\begin{equation} \label{utoest53}
		\| \AA_{\soft,L}^{N} u_{N,L}-\lambda u_{N,L}\|_{L^2(\RR^d)} \leq C (\delta_N +\tilde{\rho}(L) + \frac{1}{L}).
	\end{equation}

	As before, we observe that the expression
	$a_L (\xi) u_{N,L}(y+\xi)$ vanishes for  $y \in \R^d \setminus\square^{3L}_{\zeta_{N,L}} $  and every $\xi \in \mathbf{R}^d$. It follows that for every $N$ such that $lN > 6L$ there exists (al least one) $v_i$ such that
	\begin{equation}\label{117a}
		\mathbf{1}_{(\mathbf{R}^d_{v_i}+v_i)\backslash \Pi_{N}} (y) \int_{\R^d}  a_L(\xi) p(y,y+\xi) u_{N,L}(y+\xi) \, d\xi \equiv 0.
	\end{equation}
	Indeed, if the interior of the cube $\square^{3L}_{\zeta_{N,L}}$ does not have common points with the boundary of $\Pi_N$ then \eqref{117a} holds for every $v_i$. If, however, the interior of  $\square^{3L}_{\zeta_{N,L}}$ intersects a face of $\Pi_N$, then it does not have common points with the opposite face. In particular, if the interior of  $\square^{3L}_{\zeta_{N,L}}$ contains $v_i$ for some $i=1,\dots,2^d$, \eqref{117a} holds only for this specific $v_i$. A similar observation can be made for faces of larger co-dimensions, edges and vertices.
	
	 Combining  \eqref{117a}, \eqref{utoest53} and \eqref{defeta} we conclude that there exists $v_i$ such that for appropriately translated function $u_{N,L}$ (not relabelled) and $lN > 6L$ it holds
	\begin{equation*} 
		\| \AA_{\soft}^{v_i} u_{N,L}-\lambda u_{N,L}\|_{L^2(\RR^d)} \leq C(\delta_{N} + \rho(L)  + \tilde\rho(L) + \frac{1}{L}).
	\end{equation*} 	
Finally, choosing a sequence $L(N)\to \infty$ satisfying $lN > 6L(N)$ and passing to the limit as $N\to \infty$ we conclude that
\begin{equation*} 
	 \lambda \in \bigcup_{j=1,\dots 2^d} \mathrm{Sp} (\AA_\soft^{v_j}).
\end{equation*}
\end{proof}

\begin{remark}
	If we cover $S$ with cubes of size order $\sqrt{\e}$ consistent with the periodic microstructure (i.e. so that their faces do not cut the periodicity cell in the middle), and replace soft component with stiff in the cubes located in the $\sqrt{\e}$-neighbourhood of $\partial S$, then the boundary limit spectrum can be expressed as the union of the spectra of the soft component operators analogous to the operators $\AA_\soft^{v_i}$ defined on the domains of appropriate unions of the sets $\{ (x_1,\dots,x_d): \forall i=1,\dots,d, \ \pm x_i \geq 0 \}$ which include all possible exterior and interior corners.
\end{remark}

\subsection{Proof of Theorem \ref{th2.7}} \label{s5.3}

In view of Theorems \ref{th2.6} and \ref{th5.5}, and the fact that $\{ \beta(\l) \in \Sp (\AA_\hom)\} \subset \lim_{\e\to 0} \Sp(\AA_{\e})$ (in fact, this is valid for any subsequence $\e\to 0$), which follows from \eqref{14a} and Theorem \ref{thm2.3}, in order to complete the proof of Theorem \ref{th2.7} it is enough to establish the following assertion.

\begin{theorem}\label{krnj1}
	The following inclusion is valid for any subsequence  $(N_i)_{i\in\N}\subset \N$.
	$$	\bigcup_{j=1,\dots,2^d} \mathrm{Sp}(\AA_\soft ^{v_j}) \subseteq \lim_{i \to \infty} \Sp(\AA_{1/{N_i}}).$$
\end{theorem}

	\begin{proof}
	To simplify the notation we will not relabel the subsequence $N\to \infty$. We  provide the argument for $\Sp(\AA_\soft ^{v_1})\subset \lim_{N} \Sp(\AA_{1/{N}})$, for other $v_i$  it is completely analogous.
	As in the first part of the proof of Theorem \ref{th5.5} we assume
 $\lambda \in \Sp  (\AA_\soft ^{v_1})$ and  consider a corresponding Weyl sequence
	  $(u_k)_{k \in \mathbf{N}} \subset L^2(\mathbf{R}^d_{v_1} \cap Y_\soft^{\#})$ as in \eqref{cetest99}. Further, let  $u_{k,L}$ be defines by \eqref{utoest20}, and choose a subsequence $N(L)$ satisfying $L/N(L) \to 0$ as $L\to \infty$, and \eqref{104a}.  In particular we have that  \eqref{ana1} holds.
	
	   For the subsequence $\e = \e(L):= 1/N(L)$ we define a rescaled family of functions
	\begin{equation*}
		w_{k,\e}(x): = \e^{-d/2}  u_{k,L} (x/\e).
	\end{equation*}
	We drop the dependence of the functions $w_{k,\e}$ on $L$ from the notation as it is now encoded  through $\e$.
	Note that the support of $w_{k,\e}$ vanishes as $\e\to 0$:
	\begin{equation*}
		|\supp w_{k,\e}| = |\e\square_{\zeta_{k,L}}^{2L}| \to 0,
	\end{equation*}
	while its norm is bounded on both sides:
	 \begin{equation} \label{sriest10}
	  	1 \leq \|w_{k,\e}\|_{L^2(\mathbf{R}^d)} \leq \sqrt{4^d+2^d}.
	  \end{equation}
It follows directly from \eqref{ana1} that the rescaled family satisfies
	\begin{equation*} 
		\mathcal{A}_{\e,\soft} w_{k,\e} -\lambda w_{k,\e} := f_{k,\e} \in L^2(\e Y_\soft^\# \cap \Pi)
\end{equation*} 	
with
\begin{equation*}
	\lim_{k \to \infty}  \limsup_{\e \to 0}\|f_{k,\e}\|_{L^2(\Pi)}=0.
\end{equation*} 	
For further reference we observe that
\begin{equation}\label{124}
	a_\e(w_{k,\e},w_{k,\e}) = a_{\e,\soft}(w_{k,\e},w_{k,\e}) \leq C.
\end{equation}
	
The remaining argument relies on the following basic result \cite[Lemma E.1]{CCV2019}.
\begin{lemma}\label{l5.9}
	Let $\AA$ be a   non-negative self-adjoint (not necessarily bounded) operator in a Hilbert space $H$ and $\mathfrak{a}$ be the associated bilinear form. Assume that for some $u\in {\rm Dom}(\mathfrak{a}) = {\rm Dom}(\AA^{1/2})$, $\|u\|=1$,  $\l\in \R$, and $0<\epsilon<1$ we have
	\[
	|\mathfrak{a}(u,v) - \l  (u,v)_H| \leq \epsilon \sqrt{\mathfrak{a}(v,v) + (u,v)_H} \quad \forall v\in {\rm Dom}(\mathfrak{a}).
	\]
	Then
	\begin{equation*}
		\dist(\l,\Sp(\AA)) \leq |\l +1| (1-\epsilon)^{-1} \epsilon.
	\end{equation*}
\end{lemma}	
In view of the above lemma, in order to complete the proof, it is sufficient to establish that 	
\begin{equation} \label{sriest20}
	\left|a_\e(w_{k,\e},v)-\lambda ( w_{k,\e},v)_{L^2(\RR^d)}\right| \leq \Upsilon_{k,\e}\sqrt{a_\e (v,v)+\|v\|^2_{L^2(\RR^d)}} \qquad \forall v \in L^2(\Pi),
\end{equation}	
where
	\begin{equation*}
	\lim_{k \to \infty} \limsup_{\e \to 0} \Upsilon_{k,\e}=0. 	
\end{equation*} 	
Applying Corollary \ref{corcan3} to $ v \in L^2(\Pi)$ we have
$$
v=\bar{v}_\e  +\varepsilon \hat{v}_\e +z_\e ,
$$
where
$\bar{v}_\e  \in H^1(\RR^d)\cap C^\infty(\R^d)$, $\hat{v}_\e  \in L^2(\RR^d)$ and $z_\e \in L^2(\e Y_\soft^\# \cap \Pi)$ and
\begin{equation}\label{126a}
\|\bar{v}_\e \|^2_{H^1(\RR^d)} + \|\hat{v}_\e \|^2_{L^2(\RR^d)} + \|z_\e \|^2_{L^2(\RR^d)} \leq  C(a_\e (v,v)+\|v\|^2_{L^2(\RR^d)}).
\end{equation}

Note that since $z_\e$ vanishes outside $\e Y_\soft^\# \cap \Pi$ we have that
$a_\e(w_{k,\e},z_\e)=a_{\e,\soft}(w_{k,\e},z_\e)$, which infers  	
\begin{equation*}
	a_\e(w_{k,\e},z_\e) - \lambda ( w_{k,\e},z_\e)_{L^2(\RR^d)} =  (f_{k,\e},z_\e)_{L^2(\RR^d)}.
\end{equation*} 	
Thus we have
\begin{equation} \label{127}
a_\e(w_{k,\e},v)-\lambda ( w_{k,\e},v)_{L^2(\RR^d)} = (f_{k,\e},z_\e)_{L^2(\RR^d)} + a_\e(w_{k,\e},\bar{v}_\e  +\varepsilon \hat{v}_\e) - \lambda ( w_{k,\e},\bar{v}_\e  +\varepsilon \hat{v}_\e)_{L^2(\RR^d)}.
\end{equation}		
It remains to estimate the last two terms on the right hand side of \eqref{127}.

In subsequent estimates we will be utilising the bounds \eqref{sriest10}, \eqref{124} and \eqref{126a} without mentioning. Using the formula $ \bar{v}_\e (x+\e \xi)-\bar{v}_\e (x) =\e\int_0^1 \xi \cdot\nabla \bar{v}_\e (x+\e    t \xi) \,dt$ and the  Cauchy–Schwarz inequality   we obtain
	\begin{multline}
	\big| a_\e (w_{k,\e},\bar{v}_\e)\big| = \Big|	\int_{\R^d} \int_{\R^d} a(\xi)p(x/\e, x+\xi/\e)\left(w_{k,\e}(x+\e\xi )-w_{k,\e}(x) \right) \left(\bar{v}(x+\e\xi)-\bar{v}(x) \right)\,d\xi\,dx \Big|
	 \\
	 \leq C \e \Big(a_\e(w_{k,\e},w_{k,\e}) \int_{\R^d} a(\xi) |\xi|^2 d\xi\Big)^{1/2} \|\nabla\bar{v}_\e\|_{L^2(\RR^d)}\leq C \e \sqrt{a_\e (v,v)+\|v\|^2_{L^2(\RR^d)}} .
\end{multline}	
The next bound	  is  trivial:
\begin{equation}
	\big|a_\e(w_{k,\e}, \varepsilon \hat{v}_\e)\big| \leq C \e \sqrt{a_\e (v,v)+\|v\|^2_{L^2(\RR^d)}}.
\end{equation}
Using the H\"older and Sobolev inequalities we obtain
\begin{multline} 
\big| ( w_{k,\e},\bar{v}_\e   )_{L^2(\RR^d)}\big|=\big| ( w_{k,\e},\bar{v}_\e   )_{L^2(\e\square_{\zeta_{k,L}}^{2L})}\big|\leq \|w_{k,\e}\|_{L^2(\e\square_{\zeta_{k,L}}^{2L})}\|\bar{v}_\e\|_{L^2(\e\square_{\zeta_{k,L}}^{2L})}
\\
\leq C \|\bar{v}_\e\|_{L^{2*}(\e\square_{\zeta_{k,L}}^{2L})}\big|\e\square_{\zeta_{k,L}}^{2L}\big|^{1/d}\leq C \|\bar{v}_\e\|_{H^1(\e\square_{\zeta_{k,L}}^{2L})}\big|\e\square_{\zeta_{k,L}}^{2L}\big|^{1/d}
\\
\leq C  \big|\e\square_{\zeta_{k,L}}^{2L}\big|^{1/d}\sqrt{a_\e (v,v)+\|v\|^2_{L^2(\RR^d)}}.
\end{multline}
Finally, it is easy to see that
\begin{equation}\label{133}
 \big|( w_{k,\e}, \varepsilon \hat{v}_\e)_{L^2(\RR^d)}\big| \leq C \e \sqrt{a_\e (v,v)+\|v\|^2_{L^2(\RR^d)}}.
\end{equation}	
Combining \eqref{127}--\eqref{133} we arrive at \eqref{sriest20} with $\Upsilon_{k,\e} = C \Big(\e +\|f_{k,\e}\|_{L^2(\Pi)} +  \big|\e\square_{\zeta_{k,L}}^{2L}\big|^{1/d}\Big)$, which completes the proof.	
	\end{proof}

	\subsection{Proof of Theorem \ref{th2.5}}
	It is sufficient to establish the inclusion
	\begin{equation}\label{140}
		\mathrm{Sp}(\mathcal{A}_\soft) \subset \lim_{\e \to 0} \mathrm{Sp}(\AA_\e ).
	\end{equation} 	
	Regardless of whether $S \subset \RR^d$ is a bounded Lipschitz set  or  $S=\mathbf{R}^d$, the argument for \eqref{140} follows almost verbatim the proof of Theorem \ref{krnj1}, upon replacing $\AA_\soft^{v_1}$ with $\AA_\soft$. In fact, it is a bit easier, since one does not need to take care of the boundary of  $\mathbf{R}^d_{v_1}$.
		
	\hfill $\square$

\begin{remark} \label{remvele}
	In a similar way as in the first part of the proof of Theorem \ref{th5.5}  one can show that $\mathrm{Sp}(\AA_\soft) \subset  \lim_{\e \to 0} \Sp (\AA_{\e,\soft})$ (here $\AA_{\e,\soft}$ is defined on a generic set $S$).

\end{remark}

\section{Norm resolvent convergence for the whole space setting }\label{s7}

Recently a new abstract scheme for the norm resolvent estimates that applies to a wide range of problems with small parameter was developed in \cite{CKS2023}.
Following \cite{CKS2023} we use the scaled version of Gelfand transform $\G_\e: L^2(\R^d) \to L^2(Y^*\times  Y ), Y^*:= [-\pi,\pi]^d $,
\begin{equation*}
	(\G_\e f)(\theta, y) := \left(\frac{\e^2}{2\pi}\right)^{d/2} \sum_{n\in \Z^d} f(\e(y+n)) \rme^{-\rmi  \e \theta\cdot(y+n)}.
\end{equation*}
Applying the transform to problem \eqref{starting} and setting $\l = 1$ we obtain an equivalent family of problems for $u^\e_\theta(\cdot): = 	(\G_\e u^\e)(\theta, \cdot)$: for a.e. $\theta \in Y^*$ the function $u^\e_\theta\in L^2_\#(Y)$ solves
\begin{equation*} 
	\begin{split}
		 \int_{Y}\int_{\R^d} a(\xi - y) \left( \e^{-2}\Lambda(y,\xi) +  p(y,\xi)\right) &(\rme^{\rmi \theta\cdot(\xi - y)} u^\e_\theta(\xi)-u^\e_\theta(y))  \overline{(\rme^{\rmi \theta\cdot(\xi - y)}v(\xi)-v(y))}d\xi dy
		\\
		& +  \int_{Y} u^\e_\theta(y)\overline{v(y)} dy=\int_{Y} f^\e_\theta(y)\overline{v(y)} dy, \quad \forall v\in   L^2_\#(Y),
	\end{split}
\end{equation*}
where $f^\e_\theta(\cdot):=	(\G_\e f)(\theta, \cdot)$.  Adopting the notation from \cite{CKS2023} we can write the above problem in the form
\begin{equation} \label{62}
	\e^{-2} a_\theta(u^\e_\theta,v) + b_\theta(u^\e_\theta,v) = (f,v), \quad \forall v\in   L^2(Y),
\end{equation}
where
\begin{equation*} 
	a_\theta(u,v) : = \int_{Y}\int_{\R^d} a(\xi - y) \Lambda_0(y,\xi) (\rme^{\rmi \theta\cdot(\xi - y)} u(\xi)-u(y))\overline{(\rme^{\rmi \theta\cdot(\xi - y)}v(\xi)-v(y))}d\xi dy,
\end{equation*}
and
\begin{multline}\label{64}
	b_\theta(u,v) : =  \int_{Y}\int_{\R^d} a(\xi - y) p(y,\xi) (\rme^{\rmi \theta\cdot(\xi - y)} u(\xi)-u(y))\overline{(\rme^{\rmi \theta\cdot(\xi - y)} v(\xi)-v(y))}d\xi dy \\ +  \int_{Y} u(y)\overline{v(y)} dy.
\end{multline}
(We drop for convenience the dependence of  $f$ on $\e$ and $\theta$ in \eqref{62}.)

For every $\theta\in Y^*$ denote by $\A_\e^\theta$ the bounded symmetric operator on $L^2(Y)$ associated with problem \eqref{62}:
\begin{equation*} 
	\A_\e^\theta u(y) = 2	\int_{\R^d} a(\xi - y) \left( \e^{-2}\Lambda(y,\xi) +  p(y,\xi)\right) (\rme^{\rmi \theta\cdot(\xi - y)} u(\xi)-u(y)) d\xi.
\end{equation*}

The approach presented in \cite{CKS2023} works for a  general class of problems of the form \eqref{62} and, upon verifying a number of abstract hypotheses,  allows to approximate the original family of the operators by a more simple (homogenised) family in the norm-resolvent sense. When applying to a specific problem, like the one at hand, a lot of the argument can often be significantly simplified. In what follows we adapt some of the proofs from  \cite{CKS2023} to our setting while providing reference points to the original for reader's convenience. We will verify hypotheses (H1)--(H4) from \cite{CKS2023}, although in our setting the hypothesis (H4) is weaker, which manifests in the loss in the rate of the norm-resolvent convergence.

\subsection{Basic  properties of $a_\theta$ and $b_\theta$}
We begin our analysis by verifying some basic properties of the sesquilinear forms in \eqref{62}. Many bounds  obtain in  \cite{CKS2023} are in terms of the family of  the norms $\|\cdot \|_\theta$ associated with  the scalar product
\begin{equation*}
	(u, v)_\theta : = a_\theta(u,v)  + b_\theta(u,v).
\end{equation*}
From the continuity of the integral operators involved we immediately conclude that these norms are equivalent to the standard $L^2$-norm   uniformly in $\theta$:
\begin{equation*}
	\|\cdot \|_{L^2(Y)} \leq \|\cdot \|_{\theta} \leq K  	\|\cdot \|_{L^2(Y)}
\end{equation*}
for some $K>0$. Henceforth we will only use the $L^2$-norm in our bounds, rather than the whole family of norms $\|\cdot \|_{\theta} $, as in \cite{CKS2023}.

The following  bounds are trivial:
\begin{equation}\label{143}
	\begin{gathered}
		a_\theta(u,v) +	b_\theta(u,v)\leq C\|u \|_{L^2(Y)} \|v \|_{L^2(Y)},\quad
		b_\theta[u] \geq \|u \|_{L^2(Y)}^2, \quad \forall \theta  \in  Y^*.
	\end{gathered}
\end{equation}

Further, the forms are Lipschitz continuous in $\theta$, i.e.
\begin{equation}\label{144}
	|a_{\theta_1}(u,v) -  a_{\theta_2}(u,v)| \leq C |\theta_1 - \theta_2| \|u\|_{L^2(Y)} \|v\|_{L^2(Y)}\quad \forall u,v\in {L^2(Y)}, \quad \forall \theta_1, \theta_2 \in  Y^*,
\end{equation}
\begin{equation}\label{146}
	|b_{\theta_1}(u,v) -  b_{\theta_2}(u,v)| \leq C|\theta_1 - \theta_2| \|u\|_{L^2(Y)} \|v\|_{L^2(Y)}\quad \forall u,v\in {L^2(Y)}, \quad \forall \theta_1, \theta_2 \in  Y^*.
\end{equation}
Indeed, we have
\begin{equation*}
	\begin{aligned}
		a_{\theta_1}(u,v) -  a_{\theta_2}(u,v)= \int_{Y}\int_{\R^d} a(\xi - y) \Lambda_0(y,\xi) \Big[(\rme^{\rmi \theta_2\cdot(\xi - y)}  - \rme^{\rmi \theta_1\cdot(\xi - y)}) &u(\xi)\overline{v(y)}
		\\  +
		(\rme^{-\rmi \theta_2\cdot(\xi - y)}  - \rme^{-\rmi \theta_1\cdot(\xi - y)})& u(y)\overline{v(\xi)}\Big]  d\xi dy.
	\end{aligned}
\end{equation*}
Then \eqref{144} follows by the application of the identity
\[
\rme^{\pm\rmi \theta_2\cdot(\xi - y)}  - \rme^{\pm\rmi \theta_1\cdot(\xi - y)} = \pm \rmi (\theta_2 - \theta_1)\cdot(\xi - y) \int_0^1 \rme^{\pm\rmi (t\theta_2 + (1-t) \theta_1)\cdot(\xi - y)}dt
\]
  together with the first moment assumption on $a(\xi)$.  The argument for \eqref{146} is analogous.

\medskip

While $b_\theta$ is coercive on $L^2_\#(Y)$, the form $a_\theta$ is not and its kernel
\begin{equation*}
	V_\theta :=\{v\in {L^2_\#(Y)} \, \vert \, a_\theta[v] = 0 \},
\end{equation*}
plays an important role in  this approach to homogenisation  of high-contrast media. Notice that by the Cauchy-Schwartz inequality we have that
\begin{equation}\label{148}
    a_\theta(u,v) = a_\theta(v,u) = 0, \quad \forall v\in V_\theta, \forall u \in L^2(Y).
\end{equation}

We also define the orthogonal complement of the kernel $V_\theta$ with  respect to the inner product $(\cdot,\cdot)_\theta$:
\begin{equation*}
	W_\theta :=\{w\in {L^2_\#(Y)} \, \vert (w,v)_\theta =0, \, \forall v\in V_\theta \}.
\end{equation*}

\subsection{Verifying hypotheses}

\noindent Next we verify  hypotheses (H1)--(H4) from \cite{CKS2023}.

\begin{proposition}[{\cite[Hypothesis (H2)]{CKS2023}}]
	\begin{equation}\label{156}
		V_\theta = \left\{
		\begin{array}{ll}
			\{ v\in {L^2_\#(Y)} \, \vert \, v=0 \mbox{ in } Y_\stiff^\# \}, & \theta \neq 0,
			\\
			\vspace{-8pt}
			\\
			\mathbf C + \{ v\in {L^2_\#(Y)} \, \vert \, v=0 \mbox{ in } Y_\stiff^\# \}, & \theta =0.
		\end{array}
		\right.
	\end{equation}
\end{proposition}
\begin{proof}
For $u\in L^2_\#(Y)$ the expression 
\begin{equation*}
	\int_{\R^d} a(\xi - y) \Lambda_0(y,\xi) |\rme^{\rmi \theta\cdot\xi } u(\xi)-\rme^{\rmi \theta\cdot y}u(y)|^2 d\xi 
\end{equation*} 
is periodic with respect to $y$. Then, using assumption \eqref{a-2}  one can easily see that 
\begin{equation*}
	a_\theta[u] \geq C \int_{(Y_\stiff^\#\cap \square^{2k})^2 \cap D_r} |\rme^{\rmi \theta\cdot\xi } u(\xi)-\rme^{\rmi \theta\cdot y}u(y)|^2d\xi dy,
\end{equation*}
where $k$ is as in \eqref{8a}.
Denote  $\phi(y) : = \rme^{\rmi \theta\cdot y}u(y)$. Applying Lemma \ref{lemma:crucialest} with $\M = Y_\stiff^\#$  and then the extension Lemma \ref{rA.7}, while taking into account the quasi-periodicity of $\phi$, we obtain 
\begin{equation*}
	a_\theta[u] \geq C  \int_{Y_\stiff\times Y_\stiff}  |\phi(\xi)- \phi(y)|^2d\xi dy  \geq C \int_{Y\times Y}|\tilde\phi(y+\xi)- \tilde\phi(y)|^2 dy d\xi,
\end{equation*}
where $\tilde \phi$ is the quasi-periodic extension of $\phi$ into the soft component.
Rewriting this for the corresponding (periodic) extension $\widetilde u$ of $u$ (i.e. $\widetilde u : = \rme^{- \rmi \theta\cdot y} \tilde\phi(y)$) we have
\begin{equation*}
	a_\theta[u] \geq C \int_{Y\times Y}|\rme^{\rmi \theta\cdot\xi } \widetilde u(y+\xi)-\widetilde u(y)|^2 dy d\xi = C \int_{Y}  \sum_{k\in\Z^d} |u_k|^2|\rme^{\rmi (2\pi k + \theta)\cdot \xi } -1|^2 d\xi,
\end{equation*}
where $u_k$  are the coefficients of the Fourier series for $\widetilde u$ on $Y$. Then a direct calculation shows that for $\theta \neq 0$
\begin{equation}\label{76}
	a_\theta[u] \geq C |\theta|^2  \int_{Y}|\widetilde u|^2 \geq C |\theta|^2  \int_{Y_\stiff}|u|^2.
\end{equation}
 When $\theta = 0$ we get that
\begin{equation}\label{77}
	a_0[u] \geq  C   \int_{Y_\stiff}\big|u - \fint_{Y_\stiff} u \big|^2.
\end{equation}
This implies that the elements of $V_\theta$ vanish on $Y_\stiff$ for $\theta \neq 0$, and equal to a constant  on $Y_\stiff$  when $ \theta = 0$.
\end{proof}

\begin{remark}
Hypothesis (H2) in \cite{CCV23} requires that the family of spaces $V_\theta$ has a removable singularity at $\theta = 0$, in the sense that upon replacing the space $V_0$ with another (suitable) space one gets a Lipschitz continuous in $\theta$ (in appropriate sense) family of spaces. This is clearly the case in our setting, cf. \eqref{156}.
\end{remark}

Henceforth, we denote by $L^2(Y_\soft)$ the subspace of $L^2(Y)$ whose elements vanish outside of $Y_\soft$ (and are extended periodically to the whole of $\R^d$ if necessary). Moreover, we identify $\C$ with the subspace of constant functions. Thus we  may write $V_\theta = L^2(Y_\soft)$, $\theta\neq 0$, $V_0 = \C+ L^2(Y_\soft)$.

\begin{proposition}[{\cite[Hypotheses (H1) and (H3)]{CKS2023}}]
		There exists a positive constant $C$ such that	$\forall \theta\in  Y^*$  and  $\forall w\in W_\theta$ one has
		\begin{equation}\label{79}
			a_\theta[w]\geq \nu_\theta \|w\|_{L^2(Y)}^2,
		\end{equation}
	where
	\begin{equation*}
		\nu_\theta = \left\{
		\begin{array}{ll}
			C|\theta|^2, & \theta \neq 0,
			\\
			\vspace{-8pt}
			\\
			C, & \theta =0.
		\end{array}
		\right.
	\end{equation*}

\end{proposition}
\begin{proof}
	
Assume first that $\theta \neq 0$. Denoting
\begin{equation*}
	w = w_0 +w_1 : =  \mathbf{1}_{Y_\soft} w + \mathbf{1}_{Y_\stiff} w,
\end{equation*}
we rewrite \eqref{76} as
\begin{equation*}
	a_\theta[w] \geq   C |\theta|^2  \|w_1 \|_{L^2(Y)}^2.
\end{equation*}
Therefore, it is sufficient to show that
\begin{equation}\label{82}
	\| w_0\|_{L^2(Y)} \leq C\|w_1 \|_{L^2(Y)}.
\end{equation}
According to the orthogonality condition, for every $v\in V_\theta$ we have (cf. \eqref{148})
\begin{equation}\label{166}
	0=	(w,v)_\theta = a_\theta(w,v) + b_\theta(w,v) = b_\theta(w,v).
\end{equation}
Setting $v = w_0$ we obtain (cf. \eqref{64})
\begin{multline*}
	\int_{Y}\int_{\R^d} a(\xi - y) p(y,\xi) \left|\rme^{\rmi \theta\cdot(\xi - y)} w_0(\xi)-w_0(y)\right|^2 d\xi dy
	+  \int_{Y} |w_0(y)|^2 dy
	\\
	= - \int_{Y}\int_{\R^d} a(\xi - y) p(y,\xi) (\rme^{\rmi \theta\cdot(\xi - y)} w_1(\xi)-w_1(y))\overline{(\rme^{\rmi \theta\cdot(\xi - y)} w_0(\xi)-w_0(y))}d\xi dy,
\end{multline*}
which immediately implies  \eqref{82}.

Now we treat the case $\theta = 0$. Decomposing $w = w_0 +w_1$, where
\begin{equation*}
	w_0	 : =  \mathbf{1}_{Y_\soft} w + \mathbf{1}_{Y_\stiff} \fint_{Y_\stiff} w , \quad w_1	 : =  \mathbf{1}_{Y_\stiff}\left( w -  \fint_{Y_\stiff} w\right),
\end{equation*}
we observe that \eqref{77} implies
\begin{equation*}
	a_0[w_1] \geq  C   \int_{Y_\stiff}|w_1|^2,
\end{equation*}
and, by direct calculation, one has
\begin{equation*}
	\int_{Y} w \, \overline{w_0} = 	\int_{Y} |w_0|^2.
\end{equation*}
Then, retracing the argument for the previous  case  we arrive at
\begin{equation*}
	a_0[w] \geq  C  \|w \|_{L^2(Y)}^2.
\end{equation*}
\end{proof}

Next assertion is a straightforward consequence of the coercivity of the form $a_0$ and the Lipschitz property \eqref{144}, see also \cite[Proposition 4.1]{CKS2023}.
\begin{corollary}\label{c7.4}
There exist positive constants $C$ and $\theta_a$ such that
\begin{equation*}
C \|w_0\|_{L^2(Y)}  \leq a_\theta[w_0], \quad \forall w_0\in W_0, \,\, \forall \theta \in Y^*,\,\,|\theta|\leq \theta_a.
\end{equation*}
\end{corollary}

\begin{proposition}[{\cite[Hypothesis (H4)]{CKS2023}}]
Let the	sesquilinear maps $a_0': V_0 \times L^2(Y) \to \mathbf C^n$ and $a_0'': V_0 \times V_0 \to \mathbf C^{n\times n}$ be given by
\begin{equation*}
	\begin{aligned}
			a_0'(v,u) :=  &\int_{Y}\int_{\R^d} a(\xi) \Lambda_0(x,x+\xi)\, {\rmi \xi} \,v \,(\overline{u(x+\xi)-u(x))}\,d\xi dx,
		\\
		a_0'' (v, \hat v):= & \int_{Y}\int_{\R^d} a(\xi) \Lambda_0(x,x+\xi) \,\xi \otimes \xi \,v\, \overline{\hat v}\,d\xi dx.
	\end{aligned}
\end{equation*}
The following bounds hold
	\begin{equation}\label{96}
		|	a_\theta( v, u ) - a_0'(v,u)\cdot \theta |\leq C |\theta|^2 \|v\|_{L^2(Y)} \|u\|_{L^2(Y)}, \quad \forall v\in V_0, \forall u\in {L^2(Y)}, \forall \theta\in  Y^*,
	\end{equation}
	\begin{equation}\label{99}
		|a_\theta(v, \hat v) - a_0'' (v, \hat v) \theta\cdot \theta| \leq|\theta|^2 h(|\theta|) \|v\|_{L^2(Y)} \|\hat v\|_{L^2(Y)}, \qquad\quad \,\, \forall v, \hat v \in V_0, \forall \theta\in  Y^*,
	\end{equation}
	where  $h(t)$ is a non-negative increasing function such that $h(t)\to 0$ as $t\to 0$.
\end{proposition}
\begin{proof}

Recalling that elements of $ V_0$ are constant on $Y_\stiff$ and taking into account the expansion
\begin{equation*}
	\rme^{\rmi \theta\cdot\xi} = 1 + \rmi \theta\cdot\xi + r(\theta\cdot \xi), \mbox{ where } |r(\theta\cdot \xi)| \leq C |\theta|^2 |\xi|^2,
\end{equation*}
and the bound
\begin{equation*}
	|\rme^{\rmi \theta\cdot\xi} - 1|^2 \leq C |\theta|^2 |\xi|^2,
\end{equation*}
we can expand $a_\theta(v,u)$, $u\in {L^2(Y)}$, in terms of $\theta$ as follows,
\begin{equation*}
	\begin{aligned}
		a_\theta(v,u) = \theta \cdot \int_{Y}\int_{\R^d} a(\xi) \Lambda_0(x,x+\xi) ({\rmi \xi} \,v) (\overline{u(x+\xi)-u(x))}d\xi dx +
		\\
		\int_{Y}\int_{\R^d} a(\xi) \Lambda_0(x,x+\xi) r(\theta\cdot \xi) \,v\, \overline{(u(x+\xi)-u(x))}d\xi dx +
		\\
		\int_{Y}\int_{\R^d} a(\xi) \Lambda_0(x,x+\xi) |\rme^{\rmi \theta\cdot\xi} - 1|^2 \,v\,\overline{u(x+\xi)}.
	\end{aligned}
\end{equation*}
This readily implies  \eqref{96}.

Next we show the validity of \eqref{99}. For $v, \hat v \in V_0$ we have
\begin{equation*}
	a_\theta(v, \hat v)  = \int_{Y}\int_{\R^d} a(\xi) \Lambda_0(x,x+\xi) |\rme^{\rmi \theta\cdot\xi} -1|^2 \,v\, \overline{\hat v}\,d\xi dx.
\end{equation*}	
In case when the support of $a$ is finite the assertion is trivial due to the Taylor expansion
\begin{equation}\label{101}
	|\rme^{\rmi \theta\cdot\xi} -1|^2 = |\theta\cdot \xi|^2(1 + \mathcal R(\theta\cdot \xi)),
\end{equation}
where the $|\mathcal R(t)| \leq |t|^2$. Moreover, \eqref{99} holds with  $h(|\theta|) = O(|\theta|^4)$.

Assume now that the support of $a$ is unbounded. In this case the function
\begin{equation}\label{183}
	g(r) : = 	 \int_{|\xi|>r} a(\xi) |\xi|^2 d\xi.
\end{equation}
is strictly positive, continuous, decreasing for all $r>0$, and $g(r)\to 0$ as $r \to \infty$ since $a(\xi)|\xi|^2$ is summable.

Consider the decomposition
\begin{multline}\label{98}
	a_\theta(v, \hat v)  =
	\int_{Y}\int_{|\xi|< r} a(\xi) \Lambda_0(x,x+\xi) |\rme^{\rmi \theta\cdot\xi} -1|^2 \,v\, \overline{\hat v}\,d\xi dx +
	\\
	\int_{Y}\int_{|\xi|>r} a(\xi) \Lambda_0(x,x+\xi) |\rme^{\rmi \theta\cdot\xi} -1|^2 \,v\, \overline{\hat v}\,d\xi dx.
\end{multline}
We can estimate the second term on the right-hand  side as follows,
\begin{multline}\label{182}
	\left|	\int_{Y}\int_{|\xi|>r} a(\xi) \Lambda_0(x,x+\xi) |\rme^{\rmi \theta\cdot\xi} -1|^2 \,v\, \overline{\hat v}d\xi dx \right| =
	\\
	\left|	\int_{Y}\int_{|\xi|>r} a(\xi)|\xi|^2 \Lambda_0(x,x+\xi) \frac{|\rme^{\rmi \theta\cdot\xi} -1|^2}{|\xi|^2} \,v\, \overline{\hat v}d\xi dx \right|\leq C g(r) r^{-2}  \|v\|_{L^2(Y)} \|\hat v\|_{L^2(Y)}.
\end{multline}
Next we deal with the first term on the right-hand  side of \eqref{98} utilising \eqref{101}:
\begin{multline}\label{183a}
	\left|\int_{Y}\int_{|\xi|< r} a(\xi) \Lambda_0(x,x+\xi) |\rme^{\rmi \theta\cdot\xi} -1|^2 \,v\, \overline{\hat v}\,d\xi dx  - a_0'' (v, \hat v)\theta\cdot\theta\right| \leq
	\\
	|\theta|^2	\left|	\int_{Y}\int_{|\xi|> r} a(\xi) \Lambda_0(x,x+\xi) |\xi|^2 \,v\, \overline{\hat v}\,d\xi dx \right| +
	\\
	C |\theta|^2 \left| \int_{Y}\int_{|\xi|< r} a(\xi) \Lambda_0(x,x+\xi)  |\xi|^2 \, \mathcal R(\theta\cdot \xi) \,v\, \overline{\hat v}\,d\xi dx\right| \leq
	\\
	g(r) |\theta|^2 \|v\|_{L^2(Y)} \|\hat v\|_{L^2(Y)} + C |\theta|^2 \mathcal R(|\theta| r) \|v\|_{L^2(Y)} \|\hat v\|_{L^2(Y)}.
\end{multline}
For $t>0$ we define a function $r = r(t)$ via the relation
\begin{equation*}
	\frac{(g(r))^{1/4}}{r} = t.
\end{equation*}
Due to the properties of  $g(r)$,   the function $r(t)$ is well defined and tends to infinity as $t \to 0$. Thus we have that
\begin{equation*}
	g(r(|\theta|)) |\theta|^2 + |\theta|^2\mathcal R(|\theta|r(|\theta|)) + g(r(|\theta|)) (r(|\theta|))^{-2} \leq C \sqrt{g(r(|\theta|))} \, |\theta|^2.
\end{equation*}
Setting
\begin{equation}\label{189}
	h(t): =  C \sqrt{g(r(t))}
\end{equation}
we get \eqref{99}  (the monotonicity of $h$ can be easily inferred from the monotonicity of $g$).
\end{proof}
\begin{remark} \label{r6.6}
	In the case when the kernel $a$ has finite third moment, i.e.
	$\int_{\R^d} a(\xi)|\xi|^3 d\xi< +\infty$, then it is is easy to see that $h(t)=t$. Indeed, by setting $r=1/|\theta|$ in \eqref{98} we immediately get that the right hand sides in  \eqref{182} and \eqref{183a} are of order $|\theta|^3$.
\end{remark}

\subsection{Correctors and the homogenised form}

The hypotheses (H1)--(H4) entail a two-step approximation of the family of problems \eqref{62} by simplified ones. The first approximation  is given in Proposition \ref{th9.1} below (see \cite{CKS2023} for the proof), the second  is formulated in Proposition \ref{th7.5}. Its proof is an adaptation of the general argument from \cite{CKS2023} to the situation of the weaker hypothesis (H4). In  turn, Proposition \ref{th7.5} leads to the bounds on the rate of convergence of spectra of the operator $\AA_\e$, formulated in Proposition \ref{th7.11}.

For $v_0 \in V_0$ consider a ``corrector'' problem: find $\mathcal N_\theta v_0 \in W_0$ satisfying
\begin{equation*}
	\begin{aligned}
	a_\theta(\mathcal N_\theta v_0 , w_0  ) =  - a_\theta(v_0 , w_0  ),\quad \forall w_0 \in W_0.
	\end{aligned}
\end{equation*}
This problem is well posed for $|\theta|\leq \theta_a$ due to the coercivity of the form $a_\theta$ on the space $W_0$, see Corollary \ref{c7.4}. This defines $\mathcal N_\theta$ as a bounded linear operator from $V_0$ to $W_0$.
Note that the values of $v_0$ in $Y_\soft$ are of no importance since  $a_\theta(v_0 , w_0 ) = 0$ for every $v_0 \in L^2(Y_\soft)$.

\begin{proposition}[{\cite[Theorem 5.3]{CKS2023}}]\label{th9.1}
 There exists $\theta_1>0$ ($\theta_1\leq \theta_a$) such that for every $\theta\in Y^*$, $|\theta|\leq \theta_1$, there exists a unique solution $z+ v\in V_0 =  \C +L^2(Y_\soft)$ to
	\begin{equation}\label{116}
		B_\e(z+ v, \widetilde z+ \tilde v) :=\e^{-2} a_\theta(z+ \mathcal N_\theta z, \widetilde z + \mathcal N_\theta \widetilde z) + b_\theta(z+v, \widetilde z+ \tilde v) = ( f, \widetilde z+ \tilde v)_{L^2(Y)}, \quad \forall \widetilde z+ \tilde v \in V_0.
	\end{equation}
	(We do not reflect the dependence of $z+v$ on $\e$ in the notation.)
	Furthermore, the following error bounds hold for	
	the solution $u^\e_\theta$ to \eqref{62}:
	\begin{equation}\label{197}
		\e^{-2}  a_\theta[u^\e_\theta - (z+ \mathcal N_\theta z + v )] + b_\theta[u^\e_\theta - (z+ \mathcal N_\theta z + v )] \leq C \e^2 \|f\|^2_{L^2(Y)},
	\end{equation}
	\begin{equation}\label{118a}
		\e^{-2}  a_\theta[u^\e_\theta - (z+ \mathcal N_\theta z + v )] + b_\theta[u^\e_\theta - (z+v)] \leq C \e^2 \|f\|^2_{L^2(Y)}.
	\end{equation}
\end{proposition}
Note that we use the usual $L^2(Y)$ norm on the right-hand  sides of \eq{197}, \eqref{118a} rather than  a $\theta$-dependent dual norm used in \cite{CKS2023} due to their equivalence in the present setting.

We define a linear operator $N_\theta : V_0 \to W_0$ so that $N_\theta v$ is the unique solution to
\begin{equation}\label{114}
	a_0(N_\theta v, w) = - a_0'(v,w)\cdot \theta, \quad \forall w\in W_0.
\end{equation}
Clearly, this problem is well posed due to \eqref{79}.

Next we define the homogenised form $a_\theta^h: V_0\times V_0 \to \C$:
\begin{equation*}
	a_\theta^h(v,\tilde v): = a_0''(v,\tilde v)\theta \cdot\theta - a_0(N_\theta v, N_\theta \tilde v).
\end{equation*}

\begin{remark}\label{r7.3}
	The values of the arguments of the  form $a_\theta^h(v,\tilde v)$  in the soft component $Y_\soft$ is of no consequence  since the integral kernel is zero in  $Y_\soft$. Moreover, the space of test functions $W_0$ in \eqref{114} can be replaces with ${L^2(Y)}$ as both sides of the equation vanish if we replace $w$ by an element of $V_0$. Therefore, comparing the corrector problems \eqref{ex102} and \eqref{114}, we see that for $z\in \C$  one has $N_\theta z = \rmi \tilde \chi \cdot \theta z$, where $\tilde \chi$ is the vector whose components are the solutions to \eqref{ex102}. Furthermore,
	\begin{equation}\label{200}
		a_\theta^h(z,\widetilde z) = A^\hom \theta \cdot \theta \, z \overline{\widetilde z}, \quad z, \widetilde z \in \C,
	\end{equation}
	cf. \eqref{defA}.
\end{remark}

\begin{proposition}[{\cite[Proposition 5.4]{CKS2023}}]
Let $|\theta|\leq \theta_a$, then
	\begin{equation}\label{120a}
		\|\mathcal N_\theta v -  N_\theta v\|_{L^2(Y)} \leq C|\theta|^2 \| v\|_{L^2(Y)},\quad \forall v\in V_0.
	\end{equation}
\end{proposition}
\noindent The proof follows verbatim the argument in \cite{CKS2023} since it utilises  only the bound \eqref{96} of hypothesis (H4).
In order to simplify the notation without loss of generality we assume henceforth that
\begin{equation*}
	\liminf_{t \to 0} \frac{h(t)}{t} > 0.
\end{equation*}
(In case $\liminf_{t \to 0} \frac{h(t)}{t} = 0$ we simply replace $h(t)$ with $\max\{h(t) , t\}$.)

\begin{proposition}[{\cite[Proposition 5.5]{CKS2023}}]
	For any $|\theta|\leq \theta_a$ we have
	\begin{equation}\label{117}
		|a_\theta(v+\mathcal N_\theta v, \tilde v+\mathcal N_\theta \tilde v) - a_\theta^h(v,\tilde v)| \leq C |\theta|^2 h(|\theta|)\| v\|_{L^2(Y)} \|\tilde v\|_{L^2(Y)}, \quad \forall v,\tilde v\in V_0.
	\end{equation}
	\begin{equation}\label{118}
		a_\theta^h[z]\geq C|\theta|^2 \| z\|_{L^2(Y)}^2, \quad \forall z\in \C.
	\end{equation}
	\begin{proof}
		The only difference in the argument for \eqref{117} compared to \cite{CKS2023} is the use of the weaker bound \eqref{99},  which manifests in the weaker result. The bound \eqref{118} is   a direct consequence of \eqref{200}.
	\end{proof}
\end{proposition}

\subsection{Proof of Theorem \ref{th2.8}}

Recall the operator $\AA_\e^{h,\theta}$ associated with the form \eqref{28}. Utilising notation of this section, the operator $\AA_\e^{h,\theta}+1$ is associated with the form
\begin{equation*}
	B^h_\e(v+z, \tilde v+ \widetilde z):=	\e^{-2} a_\theta^h(z, \widetilde z) + b_\theta(v + z, \tilde v+ \widetilde z), \quad \forall v+z, \tilde v+\widetilde z \in V_0 =\C+ L^2_\#(Y_\soft).
\end{equation*}
The following assertion establishes, in particular, the bound \eqref{205a}. Indeed, it follows directly from \eqref{205} below
and the second inequality in \eqref{143}.

\begin{proposition}[{\cite[Theorem 5.6]{CKS2023}}]  \label{th7.5}
	For every $\theta\in Y^*$ there exists a unique solution $v^h + z^h \in   \C+ L^2_\#(Y_\soft)$ to
\begin{equation}\label{120}
	B^h_\e(v^h + z^h, \tilde v+ \widetilde z) =( f, \tilde v+\widetilde z )_{L^2(Y)}, \quad \forall \tilde v+\widetilde z \in \C+ L^2_\#(Y_\soft).
\end{equation}
	Furthermore, $v^h + (I+N_\theta) z^h$ and $v^h + z^h$ approximate the solution $u^\e_\theta$ to \eqref{62} in the sense that
	\begin{equation}\label{126}
		\e^{-2}  a_\theta[u^\e_\theta - (v^h + (I+N_\theta) z^h)] + b_\theta[u^\e_\theta - (v^h + (I+N_\theta) z^h)] \leq C (\overline{h}(\e))^2 \|f\|^2_{L^2(Y)} ,
	\end{equation}
	\begin{equation}\label{205}
		b_\theta[u^\e_\theta - (v^h + z^h)] \leq C  (\overline{h}(\e))^2 \|f\|^2_{L^2(Y)} ,
	\end{equation}
for some positive function $\overline{h}$ such that $\overline{h}(t) \to 0$ as $t\to 0$, $\lim_{t\to 0} \overline{h}(t)/t >0$.
\end{proposition}
\begin{proof}
	The first part of the argument up to formula \eqref{142} follows closely \cite{CKS2023}. We present it here for the sake of completeness.
	
	Problem \eqref{120} is clearly well posed.
	Testing \eqref{120} with $v^h+z^h$    yields
	\begin{equation*}
		B^h_\e[v^h+z^h] \leq \|f\|_{L^2(Y)} \|v^h+z^h\|_{L^2(Y)}.
	\end{equation*}
	Since
	\begin{equation*}
		\|v^h+z^h\|_{L^2(Y)}^2 \leq b_\theta[v^h+z^h],
	\end{equation*}
	we immediately conclude that
	\begin{equation*}
		B^h_\e[v^h+z^h] \leq  \|f\|_{L^2(Y)}^2.
	\end{equation*}
	Furthermore,
	\begin{equation*}
		\e^{-2} a_\theta^h[z^h ] \leq \|f\|_{L^2(Y)} \sqrt{ b_\theta[v^h+z^h]} -  b_\theta[v^h+z^h]\leq \frac{1}{4} \|f\|_{L^2(Y)}^2.
	\end{equation*}
	
	Next we obtain a stronger bound on $z^h$.  We denote by $P_{V_\theta}$ and $P_{W_\theta}$ the projectors onto the subspaces $V_\theta$ and $W_\theta$ respectively with respect to the scalar product $(\cdot ,\cdot )_\theta$. Recalling \eqref{166} we see that $b_\theta(v^h+z^h, P_{W_\theta} z^h ) = b_\theta[P_{W_\theta} z^h ]$. 	Then, testing \eqref{120} with $\tilde v +\widetilde z = -P_{V_\theta} z^h + z^h = P_{W_\theta} z^h$, we obtain
	\begin{equation*}
		\e^{-2} a_\theta^h[z^h ]  + b_\theta[P_{W_\theta} z^h ]\leq \|f\|_{L^2(Y)}  \| P_{W_\theta} z^h\|_{L^2(Y)} \leq C \|f\|_{L^2(Y)}  \| z^h\|_{L^2(Y)} .
	\end{equation*}
	Together with \eqref{200} this yields
	\begin{equation}\label{135}
		\e^{-2} |\theta|^2 \| z^h\|_{L^2(Y)} \leq C \|f\|_{L^2(Y)}.
	\end{equation}
	
	Next we restrict our attention to the case $|\theta|\leq \theta_1$ and utilise the solution $v+z$ of \eqref{116}  from Proposition \ref{th9.1}
to write a two-way decomposition
	\begin{equation*}
		\begin{aligned}
			u^\e_\theta - (v^h + (I+N_\theta) z^h) = &[u^\e_\theta - (v+ z+\mathcal N_\theta z)] + [(v-v^h) + (z-z^h)+ \mathcal N_\theta(z-z^h)] + [\mathcal N_\theta z^h - N_\theta z^h] =
			\\
			& [u^\e_\theta - (v+z)] +  [(v-v^h) + (z-z^h)] - N_\theta z^h,
		\end{aligned}
	\end{equation*}
	Substituting the first decomposition into $a_\theta$ and the second into $b_\theta$ we have that the left-hand side of \eqref{126} is bounded by
	\begin{multline}\label{138}
		3\left(  	\e^{-2}  a_\theta[u^\e_\theta - (v + z+\mathcal N_\theta  z)] + b_\theta[u^\e_\theta - (v +   z)] \right) +
		3 B_\e[(v-v^h) +(z-z^h)] +
		\\
		3	\e^{-2}  a_\theta[\mathcal N_\theta z^h - N_\theta z^h] +3 b_\theta[N_\theta z^h]
	\end{multline}
	(recall the definition of the form $B_\e(\cdot,\cdot)$ in \eqref{116}).
	By \eqref{118a} the first term is bounded by $C \e^2 \|f\|_{L^2(Y)}^2$.  The bounds \eqref{120a} and \eqref{135} yield
	\begin{equation}\label{209}
		\e^{-2}  a_\theta[\mathcal N_\theta z^h - N_\theta z^h] \leq C \e^{-2}  	\|\mathcal N_\theta  z^h  -  N_\theta  z^h \|_{L^2(Y)}^2 \leq C 	\e^{-2}  |\theta|^4 \|z^h\|_{L^2(Y)}^2 \leq C \e^2  \|f\|^2_{L^2(Y)} .
	\end{equation}
	From Remark \ref{r7.3} and \eqref{135} we get
	\begin{equation}\label{210}
		b_\theta[N_\theta z^h] \leq C \|N_\theta z^h\|_{L^2(Y)}^2 \leq C \e^2  \|f\|^2_{L^2(Y)} .
	\end{equation}
	It remains to analyse the second term in \eqref{138}. By testing both \eqref{116} and \eqref{120} with  $\tilde v +\widetilde z = (v-v^h) +(z-z^h)$ and comparing the resulting identities one concludes that
	\begin{equation}\label{142}
		B_\e[(v-v^h) +(z-z^h)] = 	\e^{-2}  a_\theta^h(z^h, z-z^h) - \e^{-2}   a_\theta(z^h+ \mathcal N_\theta z^h,(z -z^h) + \mathcal N_\theta(z -z^h) ).
	\end{equation}
	First  we need to obtain a bound on $B_\e[(v-v^h) +(z-z^h)]$ from below. Estimates \eqref{117} and \eqref{118} imply that there exists $r_3>0$ such that for all $|\theta|\leq r_3$ (so that $h(|\theta|)$ is sufficiently small)
	\begin{equation*}
		a_\theta[(z -z^h) +\mathcal N_\theta(z-z^h)]\geq	C|\theta|^2  \|z-z^h\|_{L^2(Y)}^2.
	\end{equation*}
	The bound
	\begin{equation*}
		b_\theta[(v-v^h) +(z-z^h)]  \geq	C\|z-z^h\|_{L^2(Y)}^2
	\end{equation*}
	is straightforward. Thus we have
	\begin{equation*}
			B_\e[(v-v^h) +(z-z^h)] \geq C \left( \e^{-2} |\theta|^2  + 1\right)	\|z-z^h\|_{L^2(Y)}^2 .
	\end{equation*}
	Applying  to the right-hand  side of \eqref{142} successively  \eqref{117}, \eqref{135} and the last bound  we arrive at
	\begin{equation*}
		\begin{aligned}
			B_\e[(v-v^h) +(z-z^h)] \leq  &C\e^{-2}  |\theta|^2 h(|\theta|)\| z^h\|_{L^2(Y)} \|z-z^h\|_{L^2(Y)}
			\\
			 \leq& C \, h(|\theta|) \|f\|_{L^2(Y)}\|z-z^h\|_{L^2(Y)}
			\\
			\leq &C \|f\|_{L^2(Y)} \, \frac{h(|\theta|)}{\sqrt{\e^{-2} |\theta|^2  + 1}}B_\e^{1/2}[(v-v^h) +(z-z^h)].
		\end{aligned}
	\end{equation*}
	
	In order to show that the term containing $\e$ and $\theta$ vanishes as $\e \to 0$ we introduce the function
	\begin{equation}\label{219}
		\hat h(t):= t \, \sup_{s\geq t} \frac{h(s)}{s}.
	\end{equation}
	Clearly,
	\begin{equation*}
		\hat h(t) \to 0 \mbox{ as } t\to 0.
	\end{equation*}
	Moreover,
	\begin{equation*}
		\frac{h(|\theta|)}{\hat h(\e)} \leq \frac{|\theta|}\e
	\end{equation*}
	whenever $\e\leq |\theta|$.
	In this case we have
	\begin{equation*}
		\frac{h(|\theta|)}{\sqrt{\e^{-2} |\theta|^2   + 1}}  \leq   \hat h(\e) \frac{\e^{-1} |\theta|}{\sqrt{\e^{-2} |\theta|^2   + 1}} \leq \hat h(\e).
	\end{equation*}
	On the other hand, when  $\e > |\theta|$, recalling the monotonicity of $h$, we trivially have
	\begin{equation*}
		\frac{h(|\theta|)}{\sqrt{\e^{-2} |\theta|^2   + 1}} \leq   h(\e).
	\end{equation*}
Thus we conclude that
	\begin{equation*}
	\begin{aligned}
		B_\e[(v-v^h) +(z-z^h)] \leq  C  (\overline{h}(\e))^2 \|f\|^2_{L^2(Y)},
	\end{aligned}
\end{equation*}
	with
	\begin{equation}\label{225}
		\overline{h}(t) := \max \{h(t) , \hat h(t)\}.
	\end{equation}
	
	Combining \eqref{138}, \eqref{118a}, \eqref{209}, and \eqref{210} yields \eqref{126}. Then \eqref{205} follows by resorting to \eqref{210} once again.
	
	It remains to consider the case $|\theta| > \theta_1$. For such $\theta$ the basic argument is as follows: the part of the solution of each of problems \eqref{62} and \eqref{120} that is supported on the stiff component $Y_\stiff$ is of order $\e^2$, while the remaining part solve a family of problems on the soft component $Y_\soft$, see \eqref{228} below.
	
	We proceed with the argument by considering the decomposition $u_\theta^\e = v_\theta^\e + w_\theta^\e \in V_\theta + W_\theta$. Using  $w_\theta^\e$ as a test function in \eqref{62} and recalling \eqref{79} we have
	\begin{equation*}
		C \e^{-2}  \theta_1^2 \|w_\theta^\e\|_{L^2(Y)}^2 \leq \e^{-2} a_\theta[w_\theta^\e] + b_\theta[w_\theta^\e] \leq \|f \|_{L^2(Y)} \|w_\theta^\e  \|_{L^2(Y)}.
	\end{equation*}
	So we conclude that
	\begin{equation*}
		\|w_\theta^\e \|_{L^2(Y)} \leq C \e^2 \| f\|_{L^2(Y)}.
	\end{equation*}
	At the same time, by taking in \eqref{62} test functions from $V_\theta$ we see that  $v_\theta^\e$ is the solution to the following well posed problem
	\begin{equation}\label{228}
		b_\theta(v_\theta^\e, v) = ( f, v)_{L^2(Y)}, \quad \forall v\in V_\theta.
	\end{equation}
	
	Next we address problem \eqref{120}. We already know that $\| z^h\|_{L^2(Y)} \leq C \e^{2} \|f \|_{L^2(Y)}$, cf. \eqref{135}.  Setting $\widetilde z = 0$ we have that
	\begin{equation}\label{229}
		b_\theta(v^h +  z^h, \tilde v) =b_\theta(v^h + P_{V_\theta} z^h, \tilde v) =   ( f, \tilde v)_{L^2(Y)}, \quad \forall  \tilde  v   \in V_\theta.
	\end{equation}
	Comparing \eqref{228} and \eqref{229} we conclude that $v_\theta^\e = v^h + P_{V_\theta} z^h$. Therefore,
	\begin{equation*}
		\| u_\theta^\e - (v^h + z^h)\|_{L^2(Y)} = \| w_\theta^\e - P_{W_\theta}  z^h\|_{L^2(Y)} \leq C \e^{2} \|f \|_{L^2(Y)},
	\end{equation*}
	and \eqref{126}, \eqref{205} follow easily (with the error bound of order $\e^2$).
	
\end{proof}

We define the self-adjoint operator $\AA_\soft^\theta$ so that     $\AA_\soft^\theta+1$ is generated by the  form       $b_\theta(\cdot,\cdot)$.

\begin{proposition}\label{th7.11}
	The limiting spectrum of the family of operators $\AA_\e$ is given by
	\begin{equation*}
		\lim_{\e\to 0} \Sp (\AA_\e) = \mathcal G : = \Sp(\AA) \cup ( \cup_{\theta\in Y^*} \Sp(\AA_\soft^\theta) ).
	\end{equation*}
	Moreover, for any $\Lambda >0$ one has
	\begin{equation*}
		d_{H, [0, \Lambda]}\Big( \Sp (\AA_\e), \mathcal G \Big) \leq C(\Lambda) \max\{ \overline{h}(\e), \,\e^{2/3}\}.
	\end{equation*}

\end{proposition}
\begin{proof}
	By Proposition \ref{th7.5} the resolvents of the operators   $\AA_\e^\theta$ and $\AA_\e^{h,\theta}$ are $\overline{h}(\e)$ close uniformly in $\theta$. Thus, by classical argument, we have
	\begin{equation*}
		d_{H, [0, \Lambda]}\Big( \Sp (\AA_\e^\theta), \Sp (\AA_\e^{h,\theta}) \Big) \leq  (\Lambda+1)^2\,\,\overline{h}(\e) \mbox{ uniformly in } \theta\in  Y^*.
	\end{equation*}
		The remainder of the proof is devoted to the analysis of the spectra of  $\AA_\e^{h,\theta}$.
	\medskip

	\noindent {\bf Step 1.}
	Assume that $\l\in \Sp (\AA_\e^{h,\theta})\cap [0, \Lambda]$. Then for any $\delta >0$ (we will assume it to be sufficiently small) there exists $v + z \in L^2(Y_\soft) \dotplus \C$, $\|v + z \|_{L^2(Y)}=1$, such that
 	\begin{equation}\label{238}
		\e^{-2} a_\theta^h(z, \widetilde z) + b_\theta(v + z, \tilde v+ \widetilde z) =  (\l +1)(v + z,\tilde v + \widetilde z)_{L^2(Y)} + (f,\tilde v  + \widetilde z)_{L^2(Y)}, \quad \forall\, \tilde v+\widetilde z \in L^2(Y_\soft) \dotplus \C,
	\end{equation}
  with some $f$ satisfying  $\| f\|_{L^2(Y)}\leq \delta$.
	
	Consider first $\theta\in  Y^*$ such that $|\theta |\geq \e^{2/3}$. By \eqref{135} we have that
	\begin{equation*}
		\| z\|_{L^2(Y)} \leq C 	\e^{2} |\theta|^{-2} \|(\l+1)(v + z) + f\|_{L^2(Y)} \leq  C \e^{2/3}.
	\end{equation*}
	Taking $\widetilde z =0$ in \eqref{238} we have
	\begin{equation}\label{242}
		b_\theta(v , \tilde v)  =  (\l +1)(v ,\tilde v )_{L^2(Y_\soft)} - b_\theta( z, \tilde v) +  ((\l+1) z + f,\tilde v )_{L^2(Y_\soft)}, \quad \forall\, \tilde v \in L^2(Y_\soft) .
	\end{equation}
	For a given $z$ the form $b_\theta( z, \tilde v)$ is a bounded linear functional on $L^2(Y_\soft)$, cf. \eqref{143}, and, hence, can be represented as
	\begin{equation*}
		b_\theta( z, \tilde v) =  (f_z,\tilde v )_{L^2(Y_\soft)}
	\end{equation*}
	for some $f_z\in L^2(Y_\soft)$ with $\|f_z \|_{L^2(Y_\soft)}\leq C  \e^{2/3}$. Thus, we can rewrite \eqref{242} as follows,
	\begin{equation*}
		b_\theta(v , \tilde v)  =  (\l +1)(v ,\tilde v )_{L^2(Y_\soft)}  + ((\l+1) z + f_z+ f,\tilde v )_{L^2(Y_\soft)}, \quad \forall\, \tilde v \in L^2(Y_\soft),
	\end{equation*}
	where
	\begin{equation*}
		\|(\l+1) z + f_z+ f\|_{L^2(Y_\soft)} \leq C \e^{2/3} + \delta.
	\end{equation*}
	Since $\delta$ is arbitrary, we conclude by Lemma \ref{l5.9} that
	\begin{equation}\label{243}
		\dist \big( \l , \Sp(\AA_\soft^\theta)\big)  \leq C \e^{2/3}.
	\end{equation}
	
	Next we consider the case $|\theta |< \e^{2/3}$. To this end we rewrite \eqref{238} in the form
	\begin{multline}\label{244}
		\e^{-2} a_\theta^h(z, \widetilde z) + b_0(v + z, \tilde v+ \widetilde z) =
		(\l +1)(v + z,\tilde v + \widetilde z)_{L^2(Y)}
		\\
		+[b_0(v + z, \tilde v+ \widetilde z) - b_\theta(v + z, \tilde v+ \widetilde z)] + (f,\tilde v  + \widetilde z)_{L^2(Y)}.
	\end{multline}
	We use an argument similar to the above. By the Lipschitz continuity of the form $b_\theta$, see \eqref{146}, we see that the expression in the square brackets is a bounded linear functional on ${L^2(Y)}$, and, hence can be represented in the form
	\begin{equation}\label{248}
		b_0(v + z, \tilde v+ \widetilde z) - b_\theta(v + z, \tilde v+ \widetilde z) = ( f_{v+z},\tilde v  + \widetilde z)_{L^2(Y)},
	\end{equation}
	for some $f_{v+z} \in L^2(Y)$  with
 \begin{equation}\label{246}
      \| f_{v+z}\|_{L^2(Y)}\leq C |\theta| \leq  C  \e^{2/3}.
 \end{equation}
 Thus we have
	\begin{multline}\label{245}
		\e^{-2} a_\theta^h(z, \widetilde z) + b_0(v + z, \tilde v+ \widetilde z) =
		(\l +1)(v + z,\tilde v + \widetilde z)_{L^2(Y)}
		+ ( f_{v+z} + f,\tilde v  + \widetilde z)_{L^2(Y)}.
	\end{multline}

	Next we make a  simple observation about the operator in the above equation.  By \eqref{200} we have that
	\begin{equation*}
		\e^{-2} a_\theta^h(z,\widetilde z) = A^\hom \xi \cdot \xi \, z \widetilde z, \mbox{ where } \xi : = \e^{-1} \theta.
	\end{equation*}
	It is a simple exercise (cf. \cite{Zhikov2000}) to check that the  spectrum of the self-adjoint operator $\AA_\xi$ associated with the form
	\begin{equation*}
		A^\hom \xi \cdot \xi \, z \widetilde z + b_0(v + z, \tilde v+ \widetilde z) - (v+z,\tilde v  + \widetilde z)_{L^2(Y)}, \quad \xi \in \R^d,
	\end{equation*}
	satisfies the relation
 \begin{equation}\label{250}
	\{ \l : \, \beta(\l) = A^\hom \xi \cdot \xi \}\subset	\Sp(\AA_\xi) \subset   \Sp(\AA_\soft^0) \cup \{ \l : \, \beta(\l) = A^\hom \xi \cdot \xi \}.
	\end{equation}
	In particular,
	\begin{equation*}
	\{ \l : \, \beta(\l) \geq 0 \} \subset	\big(\cup_{\xi \in \R^d} \Sp(\AA_\xi)\big)\subset \Sp(\AA_\soft^0) \cup \{ \l : \, \beta(\l) \geq 0 \}  = \Sp(\AA),
	\end{equation*}
	cf. \eqref{57}.
	
	From \eqref{246} and \eqref{245}, since $\delta$ is arbitrary, we conclude that $\l$ is $C \e^{2/3} $ close to the spectrum of  $\AA_\xi$ for  $\xi = \e^{-1} \theta \in B_{\e^{-1/3}}$:
 \begin{equation*}
		\dist\Big(  \l , \Sp(\AA_\soft^0) \cup\{ \beta(\l) \in [0;R_0 \e^{-2/3}] \}  \Big) \leq C(\Lambda) \e^{2/3},
	\end{equation*}
	where $R_0$ is defined as $R_0:=\max_{|\xi|\leq 1}  A^\hom \xi \cdot \xi $.
 Combining this with \eqref{243} we arrive at 	
 \begin{equation*}
		\sup_{\l\in  \cup_{\theta\in Y^*} \Sp (\AA_\e^{h,\theta}) \cap [0;\Lambda]}\dist(  \l, \,\mathcal{G} ) \leq C(\Lambda) \e^{2/3}.
	\end{equation*}

 \medskip

\noindent {\bf Step 2.} The following bounds hold:
	\begin{equation}\label{255}
		\sup_{\l\in \Sp(\AA_\soft^\theta), |\theta | < \e^{2/3}}\dist \Big( \l,\,\bigcup_{\theta\in Y^*, |\theta |\geq \e^{2/3}}\Sp(\AA_\soft^\theta)   \Big) \leq C  \e^{2/3},
	\end{equation}
	\begin{equation}\label{255a}
		\sup_{\l\in [0;\Lambda], \beta(\l) > R_0 \e^{-2/3}  } \dist(\l, \Sp(\AA_\soft^0))\leq  C \e^{2/3}.
	\end{equation}
	The first one follows by appealing  to the Lipschitz continuity of the form $b_\theta$ with respect to $\theta$ and arguing similarly to \eqref{248}--\eqref{245}. The second can be easily inferred from the estimate
	\begin{equation*}
		R_0 \e^{-2/3} <\beta(\l) =  \lambda+\lambda^2 \int_{Y_\soft } (\AA_\soft ^0-\l I)^{-1} \mathbf{1}_{Y_\soft^{\#}}\, dy \leq \l+\l^2 |Y_\soft| \|(\AA_\soft ^0-\l I)^{-1}\|.
	\end{equation*}

\medskip

\noindent {\bf Step 3.} It remains to prove the inverse bound. First we assume  $\l\in\Sp(\AA_\soft^\theta)$, $\theta\in  Y^*$. Due to the bound \eqref{255} we can further assume that $|\theta |\geq \e^{2/3}$. Then for any $\delta >0$  there exists $v  \in L^2(Y_\soft) $, $\|v  \|_{L^2(Y)}=1$, such that
 	\begin{equation*}
		 b_\theta(v , \tilde v) =  (\l +1)(v ,\tilde v )_{L^2(Y_\soft)} + (f,\tilde v  )_{L^2(Y_\soft)}, \quad \forall\, \tilde v \in L^2(Y_\soft),
	\end{equation*}
  with some $f$ satisfying  $\| f\|_{L^2(Y_\soft)}\leq \delta$. Consider $z \in \mathbf C$ defined by the identity
  \begin{equation*}
     \e^{-2}  A^\hom \theta \cdot \theta \, z = (\l +1)(v ,1 )_{L^2(Y_\soft)} + (f,1 )_{L^2(Y_\soft)} - b_\theta(v , 1).
  \end{equation*}
  Clearly, we have
  \begin{equation}\label{258}
      |z| \leq C \frac{\e^2}{|\theta|^2} \leq C \e^{2/3}.
  \end{equation}
  By direct inspection we see that $v+z$ satisfies the identity
  	\begin{multline*}
  	    	\e^{-2} a_\theta^h(z, \widetilde z) + b_\theta(v + z, \tilde v+ \widetilde z) =  (\l +1)(v + z,\tilde v + \widetilde z)_{L^2(Y)}
        \\
        + (f - (\l +1) z,\tilde v  + \widetilde z)_{L^2(Y)} + b_\theta( z, \tilde v+ \widetilde z) , \quad \forall\, \tilde v+\widetilde z \in L^2(Y_\soft) \dotplus \C.
  	\end{multline*}
	Since $\delta$ is arbitrary and from \eqref{258} we conclude that
 \begin{equation}\label{260}
    \dist( \l, \Sp (\AA_\e^{h,\theta}) ) \leq C \e^{2/3}.
 \end{equation}

Finally, we consider the case $\l\in [0;\Lambda]\setminus \Sp(\AA_\soft^0)$, $\beta(\l) \geq 0$. Due to \eqref{255a} we can assume without loss of generality that $\beta(\l) \in [0;R_0 \e^{-2/3}] $. By \eqref{250} there exists $\theta$, $|\theta|\leq \e^{2/3}$, such that  $\beta(\l)$ is an eigenvalue of $\AA_\xi$ with $\xi = \e^{-1}\theta$. Then retracing the argument of Step 1 in the reverse direction, cf. \eqref{248}, \eqref{246}, and  \eqref{244}, we arrive at \eqref{260}, which concludes the proof.

\end{proof}

\begin{remark}
	Since $\Sp(\AA_\soft)=\cup_{\theta \in Y^*} \Sp(\AA_\soft^\theta)$, understanding the structure of $\Sp(\AA_\soft^\theta)$ helps us to  characterize the spectrum of $\AA_\soft$.  Similarly to the decomposition of  $\AA_\soft^\#$, see \eqref{operator2},  we have
	\begin{equation*}
		\AA_\soft =\AA_\soft^1-\AA_\soft^2,\quad \AA_\soft^{\theta} =\AA_\soft^{\theta,1} -\AA_\soft^{\theta,2},
	\end{equation*}
	where for $z \in L^2(Y_\soft^\#)$
	\begin{eqnarray*} & &  \AA_\soft^1 z(x):=m(x)z(x),\quad \AA_\soft^2 z:=\int_{\R^d} K(x,y) z(y) \,dy, \\ & &
		m(x):=2\int_{\R^d} a(\xi)p(x,x+\xi)\,d\xi,\quad  K(x,y):=2  a(x-y)p(x,y)\,dy,
	\end{eqnarray*}
	and for $z \in L^2_\#(Y_\soft)$
	\begin{eqnarray*} & &  \AA_\soft^{\theta,1} z(x):=m(x)z(x),\quad \AA_\soft^{\theta,2}z:=\int_{\R^d} K^{\theta}(x,y) z(y) \,dy, \\ & &
		K^{\theta}(x,y):=2   a(x-y )p(x,y)\rme^{\rmi \theta \cdot(x-y )}\,dy.
	\end{eqnarray*}
	Obviously $m \in L_{\#}^{\infty} (Y)$ and $\AA_\soft^{2,\theta}$ is compact for every $\theta \in Y^*$. The essential spectrum of $\AA_\soft^{\theta}$ is the essential image of $m$. The remaining spectrum of  $\AA_\soft^{\theta}$  is at most countable number of finite multiplicity eigenvalues. The spectrum of $\AA_\soft^{\theta}$ is continuous in $\theta$ with respect to the Hausdorff distance.
\end{remark} 	

\begin{appendices}


\section{Extension operator}\label{a1}

In what follows, for an integrable function $u$ and  a set $A$, the notation   $u_A$ stands for the mean value of   $u$ over $A$. A variant of the following lemma for the case $A=B$ can be found in \cite[Lemma 2.4]{BCE2021}.
\begin{lemma}\label{LSE}
	Let $A,B\subset \R^d$   have finite and positive Lebesgue measure. Then, for every $u\in L^p(A\cup B)$ with $1\leq p<\infty$,
	\begin{eqnarray*}
		\int_{B} |u_A-u(x)|^pdx &\leq& \frac{1}{|A|}\int_{A\times B} |u(x)-u(y)|^p\,dx\,dy, \\ \label{SE2}
		|u_A-u_B|^p &\leq& \frac{1}{|A||B|}\int_{A\times B} |u(x)-u(y)|^p\,dx\,dy.
	\end{eqnarray*}
\end{lemma}
\begin{proof}
	Let $p>1$ and denote by $p'$ the conjugate exponent of $p$. By H\"older's inequality:
	\begin{equation}	
	 \begin{aligned}\label{in1}
			&\int_B |u_A-u(x)|^p \,dx=\frac{1}{|A|^p}\int_B\left|\int_A (u(y)-u(x))\, dy \right|^p\,dx
			\\
			&\leq  \frac{|A|^{p/p'}}{|A|^p} \int_B \int_A |u(y)-u(x)|^p \,dy \,dx
	=\frac{1}{|A|} \int_{A \times B} |u(y)-u(x)|^p \,dx \,dy.
	\end{aligned}
\end{equation} 		
	To prove \eqref{SE2} we proceed in a similar way  utilising \eqref{in1}:
	\begin{align*}
		|u_A-u_B|^p&	= \frac{1}{|B|^p}\left| \int_B (u_A-u(x))\, dx \right|^p
		\leq  \frac{|B|^{p/p'}}{|B|^p} \int_B  |u(y)-u(x)|^p \,dx
		\\
		&=\frac{1}{|B|} \int_{B} |u_A-u(x)|^p \,dx
		\leq  \frac{1}{|A||B|} \int_{A \times B} |u(y)-u(x)|^p \,dx \,dy.
	\end{align*} 	
	The case $p=1$ is straightforward.
\end{proof} 	
\begin{lemma}\label{lemma2.6ver} 	
	Let $A,B$ be subsets of $\R^d$ and let $|A \cap B|>0$.   Then, there exists  a linear  continuous extension operator $\Phi:L^p(A\cap B)\to L^p(B)$ such that,
	for all $u\in L^p(A)$,
	\begin{equation*}
		\Phi(u)=u \quad\mbox{in }\hspace{0.1cm}A\cap B,
	\end{equation*}
	\begin{equation*}
		\int_{ B} |\Phi(u)|^p\, dx\le c_1\int_{A\cap B}|u|^p\, dx,
	\end{equation*}
	\begin{equation}\label{3}
		\int_{ B^2}\left| \Phi(u)(x)-\Phi(u)(y)     \right|^p\, dx\, dy\le c_2\int_{(A\cap B)^2}\left|  u(x)- u(y)     \right|^p\, dx\, dy,
	\end{equation}
	where $c_1=1+\frac{| B \backslash A|}{|A \cap  B|}$ and $c_2=1+2\frac{| B \backslash A|}{|A \cap  B|}$.
\end{lemma}
\begin{proof}
	We define 
	\begin{equation*}
		\Phi(u)(x):=
		\begin{dcases}
			u(x), & x\in A\cap B, \\
			u_{A\cap B}, & x\in  B\setminus A.
		\end{dcases}
	\end{equation*}
	Using Jensen's inequality we have:
	\begin{align*}
		\int_{ B} |\Phi(u)|^p\, dx
		= \int_{ A\cap B} |\Phi(u)|^p\, dx+ \int_{ B \backslash A} |\Phi(u)|^p\,dx
		=  \int_{ A\cap B} |u|^p\, dx+ | B \backslash A| |	u_{A\cap B}|^p \\  \leq     \int_{ A\cap B} |u|^p\, dx+\frac{| B \backslash A|}{|A \cap  B|} \int_{A \cap  B} |u|^p \, dx
		= \left(1+\frac{| B \backslash A|}{|A \cap  B|}\right)  \int_{A \cap  B} |u|^p \, dx.
	\end{align*} 	
	Next we prove \eqref{3}. Observing that
	$$  B^2= (A \cap  B)^2\cup \left((A \cap  B) \times (  B \backslash A)\right)\cup \left(( B \backslash A)\times  (A \cap  B)\right) \cup ( B \backslash A)^2,  $$
	we split the integral over $ B^2$ in four parts:
	\begin{equation} \label{int1}
		\int_{(A\cap  B)^2} |\Phi(u)(x)-\Phi(u)(y)|^p \, dx\,dy= 	\int_{(A\cap  B)^2} |u(x)-u(y)|^p \, dx\,dy;
	\end{equation}
	\begin{equation}
		\int_{( B\backslash A)^2} |\Phi(u)(x)-\Phi(u)(y)|^p \, dx\,dy=0;
	\end{equation}
	Applying Lemma \ref{LSE} we conclude:
	\begin{equation}	
	\begin{aligned}
		\label{int3}
		&\int_{( B \backslash A)\times(A \cap  B) }|\Phi(u)(x)-\Phi(u)(y)|^p
		= \int_{(A \cap  B)\times ( B \backslash A)}|\Phi(u)(x)-\Phi(u)(y)|^p \, dx\,dy
		\\
	&	= | B \backslash A|\int_{A \cap  B} |u(x)-u_{A \cap   B}|^p \,dx
		\leq \frac{ | B \backslash A|}{|A \cap   B|} \int_{(A \cap  B)^2} |u(x)-u(y)|^p \,dx \,dy.
	\end{aligned}
\end{equation} 	
	Now \eqref{3} follows immediately from \eqref{int1}-\eqref{int3}.
	
\end{proof}

We now provide an extension result for a general open (not necessarily periodic) set   $\M \subset \R^d$. We make the following geometric assumption, which in case of a periodic set $Y_\stiff^\#$ is equivalent to Assumption \ref{assumptionmisja1}.

\begin{assumption}\label{assumptionmisja}
	\begin{enumerate}
		\item There exists a radius $r_0>0$ and a constant $\kappa_0>0$ such that
		\begin{equation*}
			  \frac{|\M \cap B_{r_0}(x)|}{|B_{r_0}|} \geq \kappa_0 \qquad \forall x\in \M;
		\end{equation*}
		
		\item there exists $r_1, m>0$ and $k, \overline{N}\in \N$, $k \ge r_0 + m$,  such that for any $x\in \R^d$ the set $\M \cap \square_x^m$ is not empty and for any two points $\eta', \eta''\in \M \cap \square_x^m$ there exists a discrete path from $\eta'$ to $\eta''$ contained in $\M\cap \square^k_x$, \textit{i.e.} a set of points
		\begin{equation*}
			\{\eta_0=\eta', \eta_1, \dots, \eta_N, \eta_{N+1}=\eta''\} \subset \M\cap \square^k_x,
		\end{equation*}
		such that $N\le \overline{N}$ and 	 $|\eta_{j+1}-\eta_{j}|\leq r_1$, for $j=0,1,\dots, N$.
	\end{enumerate}
\end{assumption}

\begin{remark}\label{rA.4}
Note that if Assumption \ref{assumptionmisja} b. holds for a given $m$ then it also holds for any larger $m$, possibly at the cost of enlarging $k$ and $\overline{N}$.
\end{remark}

\begin{lemma}
	\label{lemma:crucialest}
	If $\M$ satisfies Assumption \ref{assumptionmisja},
	then for any $r\ge 2 r_0 + r_1$ there exists a constant $c(r)>0$  such that the following inequality holds
	\begin{equation}
		\label{stimadiag2}
		\int_{(\M\cap \square_x^m)^2} |u(x)-u(y)|^pdx dy\leq c(r) \int_{(\M\cap \square^{2k}_x)^2\cap D_r} |u(x)-u(y)|^pdxdy, \quad \forall x\in\R^d.
	\end{equation}
	(The constant $c(r)$ also depends on the parameters from Assumption \ref{assumptionmisja}.)
\end{lemma}
\begin{proof}
	The following argument is an adaptation of \cite[Lemma 3.3]{BCE2021}.
	
	Without loss of generality we set $x=0$. Chose two arbitrary points  $\eta', \eta'' \in \M\cap \square^m$ and let $\eta_0=\eta', \eta_1, \dots, \eta_N, \eta_{N+1}=\eta''$ be a discrete path as in the Assumption \ref{assumptionmisja} b. Denote $\mathcal B_i :=\M\cap B_{r_0}(\eta_i)$, $i=0,\dots,N+1$, and let $\xi_i$ stand for the integration variable in the set $\mathcal B_i$. Note that $B_{r_0}(\eta_i)\subset \square^{2k}$. Thanks to the Jensen inequality we have
	\begin{align}
		&\int_{\mathcal B_0\times \mathcal B_{N+1}} |u(\xi_0)-u(\xi_{N+1})|^pd\xi_0d\xi_{N+1} \notag\\
		&=\prod_{i=1}^N \frac{1}{|\mathcal B_i|} \int_{\mathcal B_0} \cdots \int_{\mathcal B_{N+1}} |u(\xi_0)-u(\xi_1)+u(\xi_1)-\dots\notag\\ &\hspace{7cm}-u(\xi_N)+u(\xi_N)-u(\xi_{N+1})|^pd\xi_0 \dots d\xi_{N+1} \notag\\
		&\leq (N+1)^{p-1}\prod_{i=1}^N \frac{1}{|\mathcal B_i|} \int_{\mathcal B_0} \cdots \int_{\mathcal B_{N+1}} \sum_{j=1}^{N+1}|u(\xi_j)-u(\xi_{j-1})|^p d\xi_0 \dots d\xi_{N+1}\notag\\
		&=(N+1)^{p-1}\Bigg[ \frac{|\mathcal B_{N+1}|}{|\mathcal B_{1}|} \int_{\mathcal B_0 \times \mathcal B_1} |u(\xi_0)-u(\xi_1)|^p d\xi_0d\xi_1 \notag
		\\
		&+ \sum_{j=1}^{N-1}\frac{|\mathcal B_0||\mathcal B_{N+1}|}{|\mathcal B_{j}||\mathcal B_{j+1}|} \int_{\mathcal B_j \times \mathcal B_{j+1}} |u(\xi_j)-u(\xi_{j+1})|^p d\xi_j d\xi_{j+1}\notag
		\\
		&+\frac{|\mathcal B_{0}|}{|\mathcal B_{N}|} \int_{\mathcal B_N \times \mathcal B_{N+1}} |u(\xi_N)-u(\xi_{N+1})|^p d\xi_N d\xi_{N+1} \Bigg]\notag
		\\
		&\le \frac{(N+1)^{p-1} }{\kappa_0^2}\sum_{j=0}^{N} \int_{\mathcal B_j \times \mathcal B_{j+1}} |u(\xi_j)-u(\xi_{j+1})|^p d\xi_j d\xi_{j+1}.\label{firstine}
	\end{align}
	By construction, for $\xi_i \in \mathcal B_i$, $i=0,\dots,N$,   we have
	\begin{align*}
		|\xi_i-\xi_{i+1}|  \le 2r_0+r_1
	\end{align*}
	which implies that $\mathcal B_i\times \mathcal B_{i+1} \subset (\M \cap \square^{2k})^2\cap D_r $ . Then from \eqref{firstine}, in view of Assumption \eqref{assumptionmisja} b., we get
	\begin{align*}
		& \int_{\mathcal B_0\times \mathcal B_{N+1}} |u(\xi_0)-u(\xi_{N+1})|^pd\xi_0d\xi_{N+1}
		\le \frac{(\overline{N}+1)^{p} }{\kappa_0^2} \int_{(\M \cap \square^{2k})^2\cap D_r} |u(\xi)-u(\eta)|^p d\xi d\eta.
	\end{align*}
	Covering $\M \cap \square^m$ with a finite number of balls of radius $r_0$ and summing up the last inequality over all pairs of these balls gives the desired estimate \eqref{stimadiag2}.
\end{proof}
\begin{lemma}\label{lemma2.7}
Let $\M$ satisfy Assumption \ref{assumptionmisja}.
	Then there exists a continuous linear   operator
	$$
	\ext\cdot \hspace{0.02cm}:\hspace{0.02cm} L^p(\M)\rightarrow L^p(\R^d)
	$$
	such that for all $r\geq 2r_0+r_1$ and for all $u\in L^p(\M)$ we have
	\begin{equation}
		\label{Iestimate}
		\ext u=u,\qquad\mbox{a.e.~in} \hspace{0.3cm} \M,
	\end{equation}
	\begin{equation}
		\label{IIestimate}
		\int_{\R^d} |\ext u|^pdx\leq c_1\int_{\M} |u|^pdx,
	\end{equation}
	\begin{equation}
		\label{stimagrad}
		\int_{ D_{m+r_0}} |\ext u(x)-\ext u(y)|^p\,dx \,dy\leq c_2 \int_{\M^2\cap D_{r}}|u(x)-u(y)|^p\,dx \,dy,
	\end{equation}
	where $c_1=1+\sup_x\frac{|\square_x^{m+r_0} \backslash \M|}{|\M \cap \square_x^{m+r_0}|}$, and $c_2=c_2(c_1, k, \overline{N},m, r)$.

\end{lemma}
\begin{proof}  
	
	According to Assumption \ref{assumptionmisja} for any $x\in \R^d$ the set  $\M\cap \square_x^{m+r_0}$ is non-empty. Moreover, $|\M\cap \square_x^{m+r_0}| \geq \kappa_0 |B_{r_0}|$. In what follows we will assume without loss of generality that $m+r_0 = 1$. If it is not the case we simply can scale $\M$ by $(m+r_0)^{-1}$. 
	
	For each $i \in 2\Z^d$ we define the operator
	$\Phi^{i} :L^p(\M \cap \square_i)\to L^p(\square_i)$ as in Lemma \ref{lemma2.6ver} with $A=\M$ and $B=\square_i$:
	\begin{equation}
		\label{defop1}
		\Phi^{i} (u)(x):=
		\begin{dcases}
			u(x), & x\in \M\cap \square_i, \\
			u_{\M\cap \square_i }, & x\in \square_i \setminus \M.
		\end{dcases}
	\end{equation}
	By Lemma \ref{lemma2.6ver}
	\begin{equation}\label{2rep2}
		\int_{\square_i} |\Phi^{i} (u)|^p\, dx\le c_1\int_{\M \cap \square_i}|u|^p\, dx.
	\end{equation}
	For $u \in L^p(\M)$ we define
	$$ \ext u=\sum_{i \in \mathcal{I}} \Phi^{i} (u|_{\M \cap \square_i})\mathbf{1}_{\square_i}. $$
	Properties \eqref{Iestimate}, \eqref{IIestimate} directly follow from  \eqref{defop1}, \eqref{2rep2}.
	It remains to check \eqref{stimagrad}. First we prove that for
	any $i,j \in \Z^d$ we have
	\begin{equation} \label{wanted1}
		\int_{\square_i \times \square_j} |\ext u(x)-\ext u(y)|^p \,dx \,dy \leq  c_1^2 	\int_{(\M \cap \square_i) \times (\M \cap \square_j)} |u(x)-u(y)|^p \,dx\,dy,
	\end{equation} 	
	Similarly to Lemma \ref{lemma2.6ver}, we decompose the set $\square_i \times \square_j$ as follows:
	\begin{multline*}
		\square_i \times \square_j=((\M \cap \square_i) \times (\M \cap \square_j)) \cup ((\square_i \backslash \M) \times (\M \cap \square_j))
		\\
		\cup ((\M \cap \square_i) \times (\square_j\backslash \M))\cup ((\square_i\backslash \M) \times (\square_j\backslash \M)).
	\end{multline*}
	We have:
	\begin{equation} \label{wanted2}
		\int_{(\M \cap \square_i) \times (\M \cap \square_j )}|\ext u(x)-\ext u(y)|^p \,dx \,dy =	\int_{(\M \cap \square_i) \times (\M \cap \square_j)} |u(x)-u(y)|^p \,dx\,dy.
	\end{equation}  	
	Applying Lemma \ref{LSE} we obtain
	\begin{eqnarray}  
		\int_{(\square_i \backslash \M)\times (\M \cap \square_j) } |\ext u(x)-\ext u(y)|^p \,dx\,dy & = &|\square_i \backslash \M|\int_{\M \cap \square_j} |u(x)  - u_{\M \cap \square_i}|^p \,dx\\ &\leq&    c 	\int_{(\M \cap \square_i) \times (\M \cap \square_j)} |u(x)-u(y)|^p \,dx\,dy,
	\end{eqnarray} 
where	$c:=\sup_x\frac{|\square_x^{m+r_0} \backslash \M|}{|\M \cap \square_x^{m+r_0}|}$.	
Swapping the roles of $i$ and $j$ we get
	\begin{equation}
		\int_{(\M \cap \square_i)\times (\square_j \backslash \M) } |\ext u(x)-\ext u(y)|^p \,dx\,dy \leq    c 	\int_{(\M \cap \square_i) \times (\M \cap \square_j)} |u(x)-u(y)|^p \,dx\,dy.
	\end{equation} 	
	Finally, resorting to Lemma \ref{LSE} once again  we have
	\begin{eqnarray} \nonumber
		\int_{(\square_i \backslash \M)\times ( \square_j\backslash \M) } |\ext u(x)-\ext u(y)|^p \,dx\,dy & = &|\square_i \backslash \M||\square_j \backslash \M| |u_{\M \cap \square_i}-u_{\M \cap \square_j} |^p \\ &\leq& \label{wanted4}  c^2 	\int_{(\M \cap \square_i) \times (\M \cap \square_j)} |u(x)-u(y)|^p \,dx\,dy.
	\end{eqnarray} 	
	Combining \eqref{wanted2}-\eqref{wanted4} we arrive at \eqref{wanted1}.

	We have from \eqref{wanted1} and Lemma \ref{lemma:crucialest} (with $m=3$, cf. Remark \ref{rA.4}):
	\begin{multline*}
		\int_{  D_{1}} |\ext u(x)-\ext u(y)|^p\,dx\, dy\leq \sum_{i \in 2\Z^d} \int_{\square_i \times \square_i^3} |\ext u(x)-\ext u(y)|^p\,dx\, dy 
		\\   
		\leq  c_1^2		\sum_{i \in 2\Z^d} \int_{(\M \cap \square_i) \times (\M \cap \square_i^3 )}|u(x)-u(y)|^p\,dx\, dy 
		 \leq   	c_1^2		\sum_{i \in 2\Z^d} \int_{(\M \cap \square_i^3)^2} \,\dots
		 \\
		 \leq   c(r) c_1^2\sum_{i \in 2\Z^d}\int_{(\M\cap \square_i^{2k})^2 \cap D_r} \,\dots 
		 \leq  c_2\int_{\M^2 \cap D_r} |u(x)-u(y)|^p\,dx\,dy.
	\end{multline*}
	In the last inequality we have used the fact that each $x \in \R^d$ is contained in at most $(2k)^d$ sets $\square_i^{2k}$, $i \in 2\Z^d$.
\end{proof}

In conclusion of this part of the Appendix we state the following straightforward result.
\begin{lemma}\label{rA.7}
	Let $\mathcal{M}$ be a non-empty open $\Z^d$-periodic set and $n \in \mathbf{N}$. If $\phi$ is a (quasi-)periodic function on $ \mathcal{M} \cap nY$, then there exists a (quasi-)periodic extension $\tilde{\phi}$ on $nY$ that satisfies
	\begin{eqnarray*}
		\int_Y |\ext{\phi}|^p \,dy &\leq& c \int_{ \mathcal{M}\cap Y} 	|\phi(y)|^p \,dy; \\
		\int_{(nY)^2} |\ext{\phi}(x)-\ext{\phi} (y)|^p\,dx\,dy &\leq& c^2  \int_{(\mathcal{M}\cap nY)^2} |\phi(x)-\phi (y)|^p\,dx\,dy,
	\end{eqnarray*} 	
	where $c=1+\frac{|Y \backslash \mathcal{M}|}{|\mathcal{M} \cap Y|}$.
\end{lemma}
\begin{proof}
	The argument goes by applying the extension of Lemma \ref{lemma2.6ver} on every cube $i+Y \subset nY$ for $i \in \Z^d$.
\end{proof}

\section{Compactness result}\label{a2}

In this section we revisit the compactness theorem of \cite{BrP2021}.

\begin{theorem}\label{th compactness}
	Let $\Omega$ be an open set with Lipschitz boundary, and assume that for a bounded in $L^2(\Omega)$ family $\{u_\e\}$ the estimate
	\begin{equation}\label{1}
		\int_{\Omega_{k\e}} \int_{B_r} \left| \frac{u_\e(x+\e\xi) - u_\e(x)}\e \right|^2 d\xi dx \leq M
	\end{equation}
	is satisfied with some $k\geq r>0$. Then there exists a bounded in $H^1(\Omega_{k\e})$ family $\{w_\e\} \subset C^\infty(\Omega_{k\e})$ such that
	\begin{eqnarray}\label{3a}
		\|u_\e - w_\e\|_{L^2(\Omega_{k\e})}&\leq C\sqrt{M}\e,\quad \|w_\e \|_{H^1(\Omega_{k\e})} \leq C(\sqrt{M}+\|u_\e \|_{L^2(\RR^d)}),
	\end{eqnarray}
	where $C$ doesn't depend on $\e$ and $M$.
	In particular, we have (up to a subsequence)
	\begin{eqnarray*}
		u_\e &\to u& \mbox{ strongly in } L^2(\Omega'),
		\\
		w_\e &\to u& \mbox{ strongly in } L^2(\Omega') \mbox{ and weakly in } H^1(\Omega'),
	\end{eqnarray*}
	for some $u\in H^1(\Omega)$. Here $\Omega'$ is and arbitrary open set satisfying $\Omega'\Subset \Omega$, i.e. $\Omega'$ is bounded and compactly contained in $\Omega$.
\end{theorem}
\begin{proof}
	The following is a straightforward adaptation of the argument used in \cite[Theorem 2]{CC2015} in a similar context.
	
	Let $\varphi$ be a radially symmetric mollifier, {\it i.e.} $\varphi\geq 0$, $\varphi \in C^\infty_0(B_r)$, $\int_{B_r}\varphi\,dx=1$. We define
	\begin{equation} \label{can1}
		w_\e(x):= \int_{B_r} u_\e(x+\e\xi) \varphi(\xi)d\xi,
	\end{equation}
	where
	applying the Cauchy-Schwartz inequality  and utilising  \eqref{1} yields \eqref{3a}:
	\begin{equation*}
		\begin{aligned}
			\|u_\e - w_\e\|_{L^2(\Omega_{k\e})}^2 =	\int_{\Omega_{k\e}}\left[  \int_{B_r} \varphi(\xi) (u_\e(x+\e\xi) - u_\e(x))d\xi \right]^2 dx
			\\
			\leq \int_{\Omega_{k\e}} \left[ \int_{B_r} \varphi^2(\xi) d\xi  \int_{B_r
			}(u_\e(x+\e\xi) - u_\e(x))^2d\xi \right] dx \leq CM\e^2.
		\end{aligned}
	\end{equation*}
	This proves the first inequality in \eqref{3a}.
	Next we show that $\nabla w_\e$ is bounded. First we observe that since the mollifier is radially symmetric, its partial derivative $\partial_i \varphi(x)$ is an odd function of $x_i$   and is even with respect to all other variables (the second property is not essential, but will simplify the notation). Therefore we can write
	\begin{equation*}
		\begin{aligned}
			\partial_i w_\e(x)  =&	 \e^{-1} \int_{B_r} \partial_{\xi_i} u_\e(x+\e\xi) \varphi(\xi)d\xi = -  \e^{-1} \int_{B_r} u_\e(x+\e\xi) \partial_{\xi_i} \varphi(\xi)d\xi
			\\
			= &-  \e^{-1} \int_{B_r \cap \{\xi_i>0\}} (u_\e(x+\e\xi) - u_\e(x-\e\xi))\partial_{\xi_i} \varphi(\xi)d\xi.
		\end{aligned}
	\end{equation*}
	Using Minkowski's integral inequality in step two, applying the triangle inequality to the term $u_\e(x+\e\xi) - u_\e(x) + u_\e(x) - u_\e(x-\e\xi)$ in step three, and then the Cauchy-Schwartz inequality, we obtain:
	\begin{equation} \label{can2}
		\begin{aligned}
			\|\partial_i w_\e\|_{L^2(\Omega_{k\e})}=\e^{-1}\left[	\int_{\Omega_{k\e}}\left[  \int_{B_r\cap \{\xi_i>0\}} (u_\e(x+\e\xi) - u_\e(x-\e\xi))\partial_{\xi_i} \varphi(\xi) d\xi \right]^2 dx \right]^{1/2}
			\\
			\leq \e^{-1}\int_{B_r\cap \{\xi_i>0\}} |\partial_{\xi_i} \varphi(\xi)|\left( \int_{\Omega_{ k\e}} (u_\e(x+\e\xi) - u_\e(x-\e\xi))^2 dx \right)^{1/2} d\xi
			\\
			\leq \sqrt{2}\e^{-1} \| |\nabla\varphi| \|\big._{L^\infty(B_r)}
			\int_{B_r}  \left( \int_{\Omega_{k\e}} (u_\e(x+\e\xi) - u_\e(x))^2 dx \right)^{1/2} d\xi
			\\
			\leq \sqrt{2}\e^{-1} \| |\nabla\varphi| \|\big._{L^\infty(B_r)}
			|B_r|^{1/2}  \left( \int_{\Omega_{k\e}} \int_{B_r} (u_\e(x+\e\xi) - u_\e(x))^2 d\xi dx \right)^{1/2} \leq CM.
		\end{aligned}
	\end{equation}
	As a consequence of \eqref{can1}, Young's inequality and \eqref{can2} we obtain the second inequality of \eqref{3a}.
	The convergence properties are the direct consequences of the boundedness of $w_\e$ in $H^1(\Omega_{k\e})$ and \eqref{3a}.
\end{proof}
\begin{remark} \label{marin1}
	From the construction above we see that if $u_\e =0$ outside some compact set $K$ then $w_\e=0$ in $K^{r\e}$.
\end{remark} 	
Combining Lemma \ref{lemma2.7} and Theorem \ref{th compactness} we arrive at the following corollary.
\begin{corollary} \label{corcan3}
	Let $Y$, $Y_\stiff$, $Y_\soft $, $a_\e $ be as in Section \ref{secpbmsetting},
	and assume that $u \in L^2(\RR^d)$. Then for every $\varepsilon>0$ the function $u$ admits the following decomposition:
	$$
	u=\bar{u}_\e  +\varepsilon \hat{u}_\e +z_\e , $$
	where
	$\bar{u}_\e  \in H^1(\RR^d)\cap C^\infty(\R^d)$, $\hat{u}_\e  \in L^2(\RR^d)$ and $z_\e \in L^2(\e Y_\soft^\#)$ and the following estimates are valid:
	\begin{eqnarray*}
		& & \|\bar{u}_\e \|^2_{H^1(\RR^d)} \leq C(a_\e (u,u)+\|u\|^2_{L^2(\RR^d)}), \quad  \|\hat{u}_\e \|^2_{L^2(\RR^d)} \leq C(a_\e (u,u)+\|u\|^2_{L^2(\RR^d)}), \\ & & \hspace{+5ex}  \|z_\e \|^2_{L^2(\RR^d)} \leq C(a_\e (u,u)+\|u\|^2_{L^2(\RR^d)}),
	\end{eqnarray*}
	where $C>0$ is independent of $\varepsilon>0$.
	
	Furthermore, if $\supp u \subset S$, then $\supp z_\e \subset S\cap \e Y_\soft^\#$, and $\supp \bar{u}_\e, \,\supp\hat{u}_\e \subset S^{k\e}$.
\end{corollary} 	
\begin{proof}
	If $u=0$, the statement is trivial. If $u\not= 0$, we define $u_\e:=(a_\e (u,u)+\|u\|^2_{L^2(\RR^d)})^{-\frac12} u$. From assumption \eqref{a-2} we immediately have that
	\begin{equation*}
		\int\limits_{ \e  Y_\stiff^{\#} \times  \e  Y_\stiff^{\#} \cap \{|x-\eta|\leq \e r_a\}} \left(\frac{ u_\e(x) - u_\e(\eta)}\e \right)^2 dx d\eta \leq C.
	\end{equation*}
	Applying Lemma \ref{lemma2.7} to the restriction of $u_\e(\e \cdot)$ to $Y^\#_{\rm stiff}$ and rescaling back we obtain its extension  $\widetilde{u}_\e $ with
	\begin{align*}
		&\int\limits_{\R^d\times \R^d \cap \{|x-\eta|\leq \e r'\}}
		\left(\frac{\widetilde u_\e(x) -\widetilde u_\e(\eta)}\e \right)^2 dx d\eta  \leq  C
	\end{align*}
	for some $r'>0$. Applying  Theorem \ref{th compactness} to $\widetilde{u}_\e$ we have the following decomposition:
	$$
	u_\e= \widetilde u_\e+ u_\e-\widetilde u_\e=w_\e +(\widetilde{u}_\e -w_\e )+\check z_\e ,
	$$
	where the sequence $w_\e$ is bounded in $H^1(\R^d)$, $\|\widetilde{u}_\e -w_\e\|_{L^2(\R^d)} \leq C\e$, and $\check z_\e := u_\e-\widetilde{u}_\e $ vanishes on $Y^\#_{\rm stiff}$ by construction. Letting
	$$
	\begin{gathered}
		\bar{u}_\e :=(a_\e (u,u)+\|u\|^2_{L^2(\RR^d)})^\frac12w_\e , \quad \hat{u}_\e: =(a_\e (u,u)+\|u\|^2_{L^2(\RR^d)})^\frac12\frac{\widetilde{u}_\e -w_\e }\e,
		\\
		z_\e :=(a_\e (u,u)+\|u\|^2_{L^2(\RR^d)})^\frac12\check z_\e,
	\end{gathered}
	$$
	we obtain the first claim of the Corollary.
	
	The second claim follows by applying to $u$ the extension Lemma \ref{lemma2.7} with $\M:= (\e^{-1}S\cap Y_\stiff^\#) \cup (\R^d\setminus\e^{-1}S)$ for which Assumption \ref{assumptionmisja} holds with the same parameters. In this way we have that $\supp \ext{u}_\e \subset S$ by construction, thus $\supp z_\e \subset S\cap \e Y_\soft^\#$, and $\supp \bar{u}_\e ,\supp\hat{u}_\e \subset S^{k\e}$ by Remark \ref{marin1}.
\end{proof}	

\section{Two-scale convergence for convolution energies}\label{a3}
In this section we provide technical statements on two-scale convergence  that will be used in the derivation of the limit two-scale operator.

\begin{lemma} \label{twoscalelimit1}
	Let $a \in L^1(\R^d)$, $\Gamma \in L^{\infty}_\#(Y \times Y), $ $\Gamma_\e(\cdot,\cdot): = \Gamma(\cdot/\e,\cdot/\e)$,  and  $\mu$ be a bounded (signed)  measure on $[0,1]$. 	Let  $({u}_\e )_{\e>0}$ be a bounded sequence in $L^2(\R^d)$ such that
	$u_\e  \xrightharpoonup{2} u(x,y)\in L^2(\R^d\times Y)$.
	Then
	\begin{multline}\label{37}
		\mathcal{I}_\e:=	\int_{\R^d} \int_{\R^d} a(\xi)\Gamma_\e({x},{x+\e\xi} )\int_0^1 u_\e (x+t\e \xi) d \mu(t) \varphi(x) b({x}/\e ) d\xi\, dx
		\\
		\to \mathcal{I}:=\int_{\R^d} \int_{Y}\int_{\R^d} a(\xi) \Gamma (y,y+\xi)  \int_0^1 u(x,y+t\xi) d \mu(t) \varphi(x) b(y) d\xi\,dy\,dx,
	\end{multline}
	\begin{multline}\label{38a}
		\mathcal{J}_\e:=	\int_{\R^d} \int_{\R^d} a(\xi)\Gamma_\e({x},{x+\e\xi} )\int_0^1 u_\e (x+t\e \xi) d \mu(t) \varphi(x+\e\xi) b({x/\e+\xi} ) d\xi\, dx
		\\
		\to \int_{\R^d} \int_{Y}\int_{\R^d} a(\xi) \Gamma (y,y+\xi)  \int_0^1 u(x,y+t\xi) d \mu(t) \varphi(x) b(y+\xi) d\xi\,dy\,dx
	\end{multline}
	as $\e\to 0$ for all $\varphi \in C_0^\infty(\R^d)$, $b \in C_\#(Y)$.
\end{lemma} 	
\begin{proof}
	Using a change of variables, we have
	\begin{multline}\label{ex10}
		\mathcal{I}_\e
		=  \int_0^1 \int_{\R^d} a(\xi) \int_{\R^n} u_\e (x) \Gamma_\e({x-t\e\xi} ,{x+(1-t)\e \xi} ) \varphi(x)b(x/\e-t\xi) dx\, d\xi\, d\mu(t) \\ +
		\int_0^1 \int_{\R^d} a(\xi) \int_{\R^d} u_\e (x) \Gamma_\e({x-t\e\xi} ,{x+(1-t)\e \xi} )[\varphi(x-t\e \xi)-\varphi(x)]b(x/\e-t\xi) dx\, d\xi\,d\mu(t). 	
	\end{multline}
	Note that for every $\xi \in \R^d$, $t \in [0,1]$ one has
	\begin{multline}\label{38}
		\left|  \int_{\R^d} u_\e (x) \Gamma_\e({x-t\e\xi} ,{x+(1-t)\e \xi} ) [\varphi(x-t\e \xi)-\varphi(x)]b(x/\e-t\xi) dx \right|^2
		\\
		\leq 2\|\Gamma\|^2_{L^{\infty}}\|u_\e \|^2_{L^2}\|\varphi\|^2_{L^2}\|b\|^2_{L^{\infty}}.
	\end{multline}
	Furthermore, the quantity on the left-hand  side of \eqref{38} vanishes   as $\e\to 0$ for every $\xi \in \R^d$, $t \in [0,1]$. Therefore, by the Lebesgue dominated convergence theorem, the second term on the right-hand  side of \eqref{ex10} converges to zero. It remains to deal with the first term therein. 	Denoting
	$$\tilde{b}(y):= \int_0^1 \int_{\R^d} a(\xi) \Gamma(y-t\xi,y+(1-t)\xi)b(y-t\xi) \,d\xi\,d\mu(t),$$
	we rewrite the first term on the right-hand  side of \eqref{ex10} and pass to the limit as $\e\to 0$:
	\begin{equation}\label{13}
		\int_{\R^d}  u_\e (x) \varphi(x)\tilde{b}(x/\e)\, dx
		\to 	 \int_{\R^d} \int_{Y}u(x,y)\varphi(x)\tilde{b}(y) \,dy\,dx = \mathcal{I}.
	\end{equation}

	Using the same change of variables we write
	\begin{multline} 	\label{jos11}
		\mathcal{J}_\e =
		\int_{\R^d} u_\e  (x) \varphi(x) \int_0^1 \int_{\R^d} a(\xi)\Gamma_\e({x-t\e\xi} ,{x+(1-t)\e \xi} ) b({x/\e+(1-t)\xi}) \,d\xi\,d\mu(t)\,dx
		\\
		+ \int_{\R^d} a(\xi) \int_0^1 \int_{\R^d} \Gamma_\e({x-t\e\xi} ,{x+(1-t)\e \xi} ) [\varphi(x+(1-t)\e \xi)-\varphi(x)]u_\e (x) b({x/\e+(1-t)\xi})dx d\mu (t)d\xi.
	\end{multline} 	
	Arguing as before, we conclude that the second term on the right-hand  side of \eqref{jos11} converges to zero. Furthermore, treating
	$$\hat{b}(y):= \int_0^1 \int_{\R^d} a(\xi) \Gamma(y-t\xi,y+(1-t)\xi)b(y+(1-t)\xi) \,d\mu(t)\,d\xi, $$
	as  a test function,  we can pass to the limit  similarly to \eqref{13} in the first term on the right-hand  side of \eqref{jos11}. This completes the proof.
\end{proof}

Taking $\mu$ to be the Dirac measure supported either at $0$ or $1$  in each of the convergence statements \eqref{37} and \eqref{38a}, yields the following

\begin{corollary} \label{corandrei1}
	Under the assumptions of Lemma \ref{twoscalelimit1} one has, as $\e\to 0$,
	\begin{multline*}
		\int_{\R^d} \int_{\R^d} a(\xi)\Gamma_\e({x},{x+\e\xi} )u_\e (x) \varphi(x) b({x}/\e ) d\xi\, dx
		\\
		\to \int_{\R^d} \int_{Y}\int_{\R^d} a(\xi) \Gamma (y,y+\xi)   u(x,y)\varphi(x) b(y) d\xi\,dy\,dx,
	\end{multline*}
	\begin{multline*}
		\int_{\R^d} \int_{\R^d} a(\xi)\Gamma_\e({x},{x+\e\xi} )u_\e (x+\e \xi) \varphi(x) b({x}/\e ) d\xi\, dx
		\\
		\to\int_{\R^d} \int_{Y}\int_{\R^d} a(\xi) \Gamma (y,y+\xi) u (x,y+\xi) \varphi(x) b(y) d\xi\,dy\,dx,
	\end{multline*}
	\begin{multline*}
		\int_{\R^d} \int_{\R^d} a(\xi)\Gamma_\e({x},{x+\e\xi} )u_\e (x)  \varphi(x+\e\xi) b({x/\e+\e\xi} ) d\xi\, dx
		\\
		\to \int_{\R^d} \int_{Y}\int_{\R^d} a(\xi) \Gamma (y,y+\xi) u (x,y) \varphi(x) b(y+\xi) d\xi\,dy\,dx,
	\end{multline*}
	\begin{multline*}
		\int_{\R^d} \int_{\R^d} a(\xi)\Gamma_\e({x},{x+\e\xi} )u_\e (x+\e \xi)\varphi(x+\e\xi) b({x/\e+\e\xi} ) d\xi\, dx
		\\
		\to \int_{\R^d} \int_{Y}\int_{\R^d} a(\xi) \Gamma (y,y+\xi)  u (x,y+\xi)\varphi(x) b(y+\xi) d\xi\,dy\,dx.
	\end{multline*}
	\end{corollary} 	
\end{appendices}


\begin{thebibliography}{DMM86b}
	
	
	

	
	\bibitem{Allaire}
	G. Allaire,
	\newblock  Homogenization and two-scale convergence, {\it SIAM J. Math. Anal.,}
	\newblock   {\bf 23(6)}(1992), 1482--1518.
	
	\bibitem{AlCo1998}  G.~Allaire, C.~Conca, Bloch wave homogenisation and spectral asymptotic analysis,
	{\it J. Math. Pures Appl.}, {\bf 77}(2) (1998),  153--208.
	



		\bibitem{APZ} B. Amaziane, A. Piatnitski, E. Zhizhina,
		\newblock { Homogenization of diffusion in high contrast random media and related Markov semigroups,} {\it Discrete and Continuous Dynamical Systems - B},  {\bf 28}(8) (2023), 4661--4682.
		
\bibitem{ADH1990}
T. Arbogast, J. J. Douglas and U. Hornung, Derivation of the double porosity model of single
phase flow via homogenisation theory, {\it SIAM J. Math. Anal.}, {\bf 21} (1990), 823--836.
	

\bibitem{gloria1}
E. Bonhomme, M. Duerinckx, A. Gloria,
Homogenization of the stochastic double-porosity model, arXiv:2502.02847 (2025).

\bibitem{BrP2021}	
A. Braides, A. Piatnitski,
Homogenization of random convolution energies, {\it J. Lond. Math. Soc.}, {\bf 104} (2) (2021),  295--319.


\bibitem{BrP2022}	
A. Braides, A. Piatnitski,
Homogenization of quadratic convolution energies in periodically perforated domains,{\it  Adv. Calc. Var.}, {\bf 15}(3) (2022),  351--368.

\bibitem{BCE2021}
A. Braides, V.  Chiado Piat, L. D'Elia,  An extension theorem from connected sets and homogenisation of non-local functionals. {\it Nonlinear Anal.}, {\bf 208} (2021), Paper No. 112316, 25 pp.


\bibitem{BCVZ2022}  M.~Bu\v{z}an\v{c}i\'{c}, K.~ Cherednichenko,  I.~ Vel\v{c}i\'{c}, J.~ \v{Z}ubrini\'{c},
{\it Spectral and evolution analysis of composite elastic plates with high contrast}, J. Elast., {\bf 152} (2022),   79--177.

	
\bibitem{CC2015}
		M. Cherdantsev, K. Cherednichenko,
			\newblock  Bending of thin periodic plates,{\em Calc. Var.,}
		\newblock   {\bf 54} (2015), 4079--4117.

	\bibitem{CCV23}
M. Cherdantsev, K. Cherednichenko, I. Vel\v{c}i\'{c},
\newblock  High-contrast random composites: homogenisation framework and new spectral phenomena,
\newblock   arXiv:2110.00395 (2024).



\bibitem{CCV2019}
M. Cherdantsev, K. Cherednichenko, I. Velcic,  Stochastic homogenisation of high-contrast media. {\it Appl. Anal.},
{\bf 98}(1--2) (2019),  91--117.

\bibitem{CKVZ}
K. Cherednichenko, A. V.  Kiselev, I. Vel\v{c}i\'{c}, J. \v{Z}ubrini\'{c},
\newblock
	Effective behaviour of critical-contrast PDEs: micro-resonances, frequency conversion, and time dispersive properties. II,
{\em Comm. Math. Phys.}, {\bf 406}(4) (2025),  Paper No. 72, 72 pp.





\bibitem{Coo2018}
S.~Cooper,
Quasi-periodic two-scale homogenisation and effective spatial dispersion in high-contrast media,
{\it Calc. Var. Partial Differential Equations}, {\bf 57}(3) (2018), Paper No. 76, pp. 33.

\bibitem{CC}
S. Cooper, K. Cherednichenko, \newblock Resolvent estimates for high-contrast homogenisation problems, {\em Arch. Rational Mech. Anal.}, 219 (2016), 1061--1086.

	 \bibitem{CKS2023}	
	 S. Cooper, I. Kamotski, V. P. Smyshlyaev,
Uniform asymptotics for a family of degenerating variational problems and error estimates in homogenisation theory,  arXiv:2307.13151 (2023).

\bibitem{Foll99}
G.~B.~Folland. {\it Real Analysis: Modern Techniques and their Applications}, Second Edition. John Wiley \& Sons, 1999.

\bibitem{Halmos} P. R. Halmos, V. S. Sunder. {\it Bounded Integral Operators on $L^2$ Spaces}. Springer-Verlag, 1978.

\bibitem{mielketimofte} A.~Mielke, A.~Timofte,
Two-scale homogenisation for evolutionary variational inequalities via the energetic formulation,
{\it SIAM J. Math. Anal.}, {\bf 39} (2) (2007), 642--668.



\bibitem{KoKuPir2008}
Yu. Kondratiev, O. Kutoviy, S. Pirogov, Correlation functions and invariant measures in continuous contact model,
{\it Infin. Dimens. Anal. Quantum Probab. Relat. Top.},  {\bf 11}(2) (2008),  231--258.

\bibitem{KMPZ2017}  Yu. Kondratiev, S. Molchanov, S. Pirogov, E. Zhizhina, On ground state of some non local Schrödinger operators,   {\it Appl. Anal.}, {\bf 96}(8) (2017), 1390--1400.

\bibitem{KoPirZh2016} Yu. Kondratiev, S. Pirogov, E. Zhizhina, A quasispecies continuous contact model
in a critical regime, {\it J. Stat. Phys.}, {\bf 163}(2) (2016),  357--373.

	\bibitem{Past} S. E. Pastukhova, On the convergence of hyperbolic semigroups in variable Hilbert spaces,  {\it   J. Math. Sci.,} {\bf 127}(5) (2005), 2263--2283.

\bibitem{PiZhiz2019}
A. Piatnitski,  E.  Zhizhina,      Homogenization of biased convolution type operators. {\it Asymptotic Analysis},
{\bf  115}(3-4) (2019), 241--262.

\bibitem{PiZhiz2020}
A. Piatnitski,  E.  Zhizhina,      Stochastic homogenisation of convolution type operators. {\it J. Math. Pures Appl.};
{\bf  134} (2020), 36--71.

\bibitem{PiZhiz2017} A. Piatnitski, E. Zhizhina,  Periodic homogenisation of nonlocal operators with a convolution-type kernel, {\it SIAM J. Math. Analysis}, {\bf 49}(1) (2017),  64--81.

\bibitem{PSSZ2023}
A. Piatnitski,  V.  Sloushch, T. Suslina,  E.  Zhizhina,       On operator
 estimates in homogenisation of nonlocal operators of convolution type. {\it J. Diff. Equations},
 {\bf  352} (2023), 153--188.

\bibitem{PiZhiz_arxiv} A. Piatnitski,  E.  Zhizhina,     Homogenization of convolution type semigroups in high contrast media, {\it  Analysis and Mathematical Physics}, {\bf 15} (2025), Paper No. 36.  

	    \bibitem{zhikov2}
	V.V. Jikov, S.M. Kozlov, O.A. Oleinik, {\it Homogenization of Differential Operators and Integral
		Functionals}. Springer-Verlag, Berlin, 1994.
	
	
    \bibitem{Zhikov2000}
    V.V. Zhikov,   On an extension of the method of two-scale
    convergence and its applications,  {\itshape Sb. Math.,}
    {\bf 191}(7--8) (2000), 973--1014.

     \bibitem{Zhikov2004}	V.V. Zhikov,  Gaps in the spectrum of some elliptic operators in divergent form with periodic coefficients,  \textit{St. Petersburg Math. J.,} {\bf 16}(5) (2004), 773--790.

\end{thebibliography}
\end{document}